\documentclass[reqno]{amsart}
\usepackage[dvipsnames]{xcolor}
\usepackage{marginnote}
\usepackage{thmtools}
\usepackage{amsfonts, fancyhdr, amsmath, amssymb, hyperref, url, amsthm, subfigure, tikz-cd, verbatim, enumitem,lscape, todonotes, varioref}
\usetikzlibrary{decorations.pathmorphing}
\usepackage[mathscr]{eucal} 

\usepackage[noabbrev, capitalise, nameinlink]{cleveref}

\crefname{equiv}{Equivalence}{Equivalences}
\Crefname{equiv}{Equivalence}{Equivalences}

\makeatletter
\if@cref@nameinlink
  \creflabelformat{equiv}{#2\textup{(#1)}#3}
\else
  \creflabelformat{equiv}{\textup{(#2#1#3)}}
\fi
\makeatother

\newenvironment{equivalence}
  {\begingroup\crefalias{equation}{equiv}\begin{equation}}
  {\end{equation}\endgroup}

\usepackage{longtable}
\usepackage{appendix}

\numberwithin{equation}{section}

\newtheorem{thm}{Theorem}[section]
\newtheorem{lem}[thm]{Lemma}
\newtheorem{quest}[thm]{Question}

\newtheorem{prop}[thm]{Proposition}
\newtheorem{df}[thm]{Definition}
\newtheorem*{df*}{Definition}

\newtheorem{cor}[thm]{Corollary}
\newtheorem{rmk}[thm]{Remark}
\newtheorem{notation}[thm]{Notation}

\newtheorem{exam}[thm]{Example}

\newtheorem{thmx}{Theorem}

\crefname{thm}{Theorem}{Theorems}
\Crefname{thm}{Theorem}{Theorems}
\crefname{thmx}{Theorem}{Theorems}
\Crefname{thmx}{Theorem}{Theorems}
\crefname{lem}{Lemma}{Lemmas}
\Crefname{lem}{Lemma}{Lemmas}
\crefname{prop}{Proposition}{Propositions}
\Crefname{prop}{Proposition}{Propositions}
\crefname{cor}{Corollary}{Corollaries}
\Crefname{cor}{Corollary}{Corollaries}
\crefname{rmk}{Remark}{Remarks}
\Crefname{rmk}{Remark}{Remarks}
\crefname{quest}{Question}{Questions}
\Crefname{quest}{Question}{Questions}
\crefname{notation}{Notation}{Notations}
\Crefname{notation}{Notation}{Notations}

\newcommand{\Z}{\mathbb{Z}}
\newcommand{\uZ}{\underline{\mathbb{Z}}}
\newcommand{\Ftwo}{\mathbb{F}_2}
\newcommand{\uFtwo}{\underline{\mathbb{F}}_2}
\newcommand{\uHom}{\underline{\operatorname{Hom}}}
\newcommand{\uExt}{\underline{\operatorname{Ext}}}

\newcommand{\upi}{\underline{\pi}}
\newcommand{\tr}{\operatorname{tr}}
\newcommand{\Pic}{\operatorname{Pic}}
\newcommand{\res}{\operatorname{res}}

\definecolor{crimson}{rgb}{0.86, 0.08, 0.24}
\definecolor{darkcyan}{rgb}{0.0, 0.55, 0.55}

\hypersetup{
hypertexnames=false,
colorlinks={true},
linkcolor= {darkcyan},
citecolor= darkcyan}

\author[Kong]{Hana Jia Kong}\address{School of Mathematical Sciences, Zhejiang University}\email{hjkong@zju.edu.cn}
\author[Krause]{Achim Krause}\address{Department of Mathematics, University of Oslo}\email{achimkr@math.uio.no}
\author[Liu]{Yutao Liu}\address{Department of Mathematics, University of Washington} \email{yutao492@uw.edu}
\author[Shi]{XiaoLin Danny Shi}\address{Department of Mathematics, University of Washington}\email{dannyshi@uw.edu}
\author[Xu]{Zhouli Xu} \address{Department of Mathematics, UCLA} \email{xuzhouli@ucla.edu}

\title{On the Anderson Dual of $H\underline{\Z}$ and $H\underline{\Ftwo}$}

\begin{document}
\begin{abstract}

In this paper, we give complete classifications of the finite groups $G$ for which the constant Eilenberg--Mac Lane spectra $H\underline{\mathbb Z}$ and $H\underline{\mathbb F}_2$ are Anderson self-dual up to a virtual representation shift, and we determine all such shifts. For both of these spectra, we determine the representation periodicity lattice for every finite group.

As applications, we obtain $RO(G)$-graded universal coefficient theorems relating equivariant homology and cohomology, together with formulas for the Anderson duals of $H\underline{\mathbb Z}$-modules. In the $H\underline{\mathbb F}_2$ self-dual cases, we identify the graded $\mathbb F_2$-linear dual of the coefficient ring $\pi_\star^G H\underline{\mathbb F}_2$ with a representation shift of the ring as a module over itself. This gives perfect multiplication pairings and graded self-injectivity.

\end{abstract}

\maketitle
\tableofcontents
\vspace{-32pt}

\section{Introduction}
\enlargethispage{2\baselineskip}

\subsection{Classifications and representation periodicity}

Classically, ordinary integral homology and cohomology are related by the universal coefficient theorem, and mod-$2$ cohomology is the linear dual of homology. In general, for any spectra $E$ and $X$ and every integer $n$, there is a natural short exact sequence
\[
0 \longrightarrow \operatorname{Ext}^1_{\mathbb Z}\!\bigl(E_{n-1}(X),\mathbb Z\bigr)
\longrightarrow (I_{\mathbb Z}E)^n(X)
\longrightarrow \operatorname{Hom}_{\mathbb Z}\!\bigl(E_n(X),\mathbb Z\bigr)
\longrightarrow 0,
\]
where $I_{\mathbb Z}E$ is the Anderson dual of $E$~\cite{And69}.  Both $H\mathbb Z$ and $H\mathbb F_2$ are Anderson self-dual up to suspension:
\[
I_{\mathbb Z}(H\mathbb Z)\simeq H\mathbb Z,
\qquad
I_{\mathbb Z}(H\mathbb F_2)\simeq\Sigma^{-1}H\mathbb F_2.
\]
 These equivalences recover the two classical statements above.

For a finite group $G$, the equivariant Anderson duals of $H\uZ$ and $H\uFtwo$ are 
\[I_{\mathbb Z}^G (H\uZ) \simeq H\uZ^*,
\qquad
I_{\mathbb Z}^G (H\uFtwo) \simeq \Sigma^{-1}H\uFtwo^*
\]
respectively, where $\uZ^*$ and $\uFtwo^*$ are obtained from the constant Mackey functors by interchanging restrictions and transfers. 

\begin{quest}\label{Quest:MainQuestion}\rm\leavevmode
\begin{enumerate}[label=\textup{(\arabic*)},nosep]
\item For which finite groups $G$ and which virtual representations $V\in RO(G)$ is there an equivalence
\begin{equivalence}\label{mainquestion}
H\uZ^{*}\wedge S^V\simeq H\uZ\,?
\end{equivalence}
\item For which pairs $(G,V)$ is there an equivalence
\begin{equivalence}\label{eq:HF2duality-intro}
I_{\mathbb Z}^{G}(H\uFtwo)\wedge S^V\simeq H\uFtwo\,?
\end{equivalence}
\end{enumerate}
\end{quest}

\begin{df}\rm
A pair $(G,V)$ satisfying \cref{mainquestion} is called an \emph{Anderson-dualizing pair}. We call $G$ an \emph{Anderson-dualizing group} if such a $V$ exists. For \cref{eq:HF2duality-intro}, we use the terms \emph{$H\uFtwo$-Anderson-dualizing pair} and \emph{$H\uFtwo$-Anderson-dualizing group}.
\end{df}

In this paper, we classify all Anderson-dualizing groups and pairs for $H\uZ$ and $H\uFtwo$, and determine the representation periodicity of both spectra for every finite group. The resulting dualities give $RO(G)$-graded universal coefficient theorems and a formula for the Anderson dual of every $H\uZ$-module. For $H\uFtwo$, they give perfect multiplication pairings and graded self-injectivity of the coefficient ring.

\begin{samepage}
\begin{thmx}[Group classifications]
{\label{thm:HF2groups-v2}} Let $G$ be a finite group.\label{localcriterion}
\begin{enumerate}[label=\textup{(\arabic*)},nosep]
\item $G$ is Anderson-dualizing if and only if the following conditions hold:
\begin{enumerate}[label=\textup{(\alph*)},nosep,leftmargin=*]
\item Every Sylow $p$-subgroup of $G$ is cyclic when $p$ is odd.
\item Every Sylow $2$-subgroup is cyclic or dihedral.
\item For every nontrivial cyclic $p$-subgroup $H\mathrel{\leq} G$, every element of $N_G(H)$ acts on $H$ either by the identity or by inversion.
\end{enumerate}
\item $G$ is $H\uFtwo$-Anderson-dualizing if and only if every Sylow $2$-subgroup of $G$ is cyclic or dihedral.
\end{enumerate}
\end{thmx}
\end{samepage}
\pagebreak[0]

The groups satisfying the integral conditions admit an explicit classification.

\begin{thmx}[Explicit group classification]\label{thm:maintheorem}\leavevmode\\*
A finite group $G$ is Anderson-dualizing if and only if $G\cong K\ltimes C_m$, where:
\begin{enumerate}[label=\textup{(\arabic*)}]
\item $K$ is one of
\[
C_{2^n},\qquad D_{2^{n+1}},\qquad T_{3^n},\qquad O_{3^n},\qquad A_5,
\]
with $n\geq0$ in the cyclic family and $n\geq1$ in the other three indexed families.
\item $m$ is odd and $(m,|K|)=1$.
\item Every element of $K$ acts on $C_m$ by an automorphism of order at most $2$.
\end{enumerate}
\end{thmx}

In \cref{thm:maintheorem}, $D_{2r}=C_2\ltimes C_r$ is the dihedral group of order $2r$, where the generator of $C_2$ acts by inversion. In particular, $D_4=C_2\times C_2$. In a semidirect product $A\ltimes B$, the factor $B$ is normal.

Two families extending $A_4$ and $S_4$ occur in the list of \cref{thm:maintheorem}. To introduce these families, let $E=C_2\times C_2$ and identify $\operatorname{Aut}(E)$ with $S_3$ through its action on the three nonidentity elements of $E$. For $n\geq1$, define
\[
\begin{aligned}
T_{3^n}&=C_{3^n}\ltimes E,\\
O_{3^n}&=D_{2\cdot3^n}\ltimes E,
\end{aligned}
\]
using the natural quotient maps $C_{3^n}\twoheadrightarrow C_3\mathrel{\leq} S_3$ and $D_{2\cdot3^n}\twoheadrightarrow D_6\cong S_3$. The subgroup $T_{3^n}\mathrel{\leq} O_{3^n}$ has index $2$. In particular, $T_3\cong A_4$ and $O_3\cong S_4$.

Condition \textup{(3)} says that on each Sylow subgroup of $C_m$, every element of $K$ acts either by the identity or by inversion. In particular, $A_5$ is the only nonabelian simple Anderson-dualizing group.

\begin{cor}\label{listremark}
Let $G$ be a finite subgroup of $SO(3)$, and let $W$ be the restriction to $G$ of the standard three-dimensional representation of $SO(3)$. Then $(G,W-3)$ is an Anderson-dualizing pair, and $(G,W-2)$ is an $H\uFtwo$-Anderson-dualizing pair:
\[
\begin{aligned}
H\uZ^*\wedge S^{W-3}&\simeq H\uZ,\\
I_{\mathbb Z}^{G}(H\uFtwo)\wedge S^{W-2}&{}\simeq H\uFtwo.
\end{aligned}
\]
\end{cor}

An integral Anderson-dualizing pair $(G,V)$ gives an $H\uFtwo$-Anderson-dualizing pair $(G,V+1)$. The converse fails: every finite group of odd order admits the mod-$2$ shift $1$, whereas $C_3\times C_3$ admits no integral dualizing shift.

\begin{thmx}[All Anderson-dualizing shifts]\label{thm:maintheoremB}
{\label{thm:HF2shifts-v2}}\leavevmode\\*
Let $G$ be a finite group and $V\in RO(G)$.
\begin{enumerate}[label=\textup{(\arabic*)},nosep]
\item \(\displaystyle
H\uZ^*\wedge S^V\simeq H\uZ
\quad\mathrel{\Longleftrightarrow}\quad
\begin{cases}
|V|=0,\\
|V^C|=-2 \quad\text{for every } C,
\end{cases}
\)
\par\noindent
where $C$ ranges over the nontrivial cyclic subgroups of $G$ of prime-power order.
\item \(\displaystyle
I_{\mathbb Z}^{G}(H\uFtwo)\wedge S^V\simeq H\uFtwo
\quad\mathrel{\Longleftrightarrow}\quad
\begin{cases}
|V|=1,\\
|V^C|=-1 \quad\text{for every } C,
\end{cases}
\)
\par\noindent
where $C$ ranges over the nontrivial cyclic $2$-subgroups of $G$.
\end{enumerate}
\end{thmx}

In particular, the existence of a virtual representation satisfying either dimension test already forces $G$ to satisfy the corresponding group condition in \cref{localcriterion}.

\begin{samepage}
\begin{df}\label{def:periodicity-lattices}\rm
For a finite group $G$ and a $G$-spectrum $E$, define
\[
\Lambda_G(E)=\{V\in RO(G)\mid \Sigma^V E\simeq E\}.
\]
We call $\Lambda_G(E)$ the \textit{periodicity lattice} of $E$.
\end{df}
\end{samepage}

Whenever a dualizing shift exists for $E$, all such shifts form an affine translate of $\Lambda_G(E)$. \cref{thm:maintheoremC} completely determines $\Lambda_G(H\uZ)$ and $\Lambda_G(H\uFtwo)$ for every finite group.

\begin{thmx}[Representation periodicity]
\label{thm:maintheoremC}
{\label{thm:HF2periodicity-intro}} Let $G$ be a finite group.
\begin{enumerate}[label=\textup{(\arabic*)},nosep]
\item The periodicity lattice of $H\uZ$ is
\begin{equation*}
\Lambda_G(H\uZ)\mathrel{=}
\mathopen{\{}V\in RO(G)\mathrel{\mid}
|V|=0, \,\,\, |V^C|=0\text{ for every }C\mathclose{\}},
\end{equation*}
where $C$ ranges over the nontrivial cyclic subgroups of $G$ of prime-power order. The quotient $RO(G)/\Lambda_G(H\uZ)$ is a free abelian group whose rank is the number of $G$-conjugacy classes of cyclic subgroups of prime-power order, including the trivial subgroup.
\item The periodicity lattice of $H\uFtwo$ is
\begin{equation*}
\Lambda_G(H\uFtwo)\mathrel{=}
\mathopen{\{}W\in RO(G)\mathrel{\mid}
|W|=0, \,\,\, |W^C|=0\text{ for every }C\mathclose{\}},
\end{equation*}
where $C$ ranges over the nontrivial cyclic $2$-subgroups of $G$. The quotient $RO(G)/\Lambda_G(H\uFtwo)$ is a free abelian group whose rank is the number of $G$-conjugacy classes of cyclic $2$-subgroups, including the trivial subgroup.
\end{enumerate}
\end{thmx}

\subsection{Universal coefficients and module duality}

The dualizing shifts determine the relation between equivariant homology and cohomology.

\begin{cor}[$RO(G)$-graded universal coefficient theorem]\label{thm:UCT}
Let $(G,V)$ be an Anderson-dualizing pair. For every $G$-spectrum $X$, there is a natural short exact sequence
\begin{equation}\label{eq:ROGUCT}
0\longrightarrow\uExt^1_{\mathbb Z}\!\bigl(H\uZ_{\star-1}(X),\mathbb Z\bigr)
\longrightarrow H\uZ^{\star-V}(X)
\longrightarrow\uHom_{\mathbb Z}\!\bigl(H\uZ_\star(X),\mathbb Z\bigr)
\longrightarrow0
\end{equation}
of $RO(G)$-graded Mackey functors.
\end{cor}

The integral dualizing shift also determines the Anderson dual of every $H\uZ$-module in terms of its $H\uZ$-linear dual.

\begin{cor}[Anderson dual of $H\uZ$-modules]\label{cor:MAndersonDual}
Let $(G,V)$ be an Anderson-dualizing pair. For every $H\uZ$-module $M$, there is a natural equivalence
\begin{equation}\label{eq:MAndersonDual}
I_{\mathbb Z}^G(M)\simeq\Sigma^{-V}F_{H\uZ}(M,H\uZ).
\end{equation}
In particular, this formula applies to $H\underline M$ for every cohomological Mackey functor $\underline M$.
\end{cor}

For $H\uFtwo$, the universal coefficient sequence gives linear duality. The module structure makes this a statement about multiplication in the coefficient ring. Put $R_G=\pi_\star^G H\uFtwo$. For an $RO(G)$-graded $\Ftwo$-vector space $M$, write
\[
D(M)_\alpha=\operatorname{Hom}_{\Ftwo}(M_{-\alpha},\Ftwo),
\qquad
(\Sigma^W M)_\alpha=M_{\alpha-W}.
\]

\begin{samepage}
\begin{thmx}[Universal coefficients and coefficient-ring duality]
\label{thm:HF2selfinjective-v2}
Let $(G,V)$ be an Anderson-dualizing pair for $H\uFtwo$.
\begin{enumerate}[label=\textup{(\arabic*)},nosep]
\item For every $G$-spectrum $X$, there is a natural isomorphism
\[
H\uFtwo^{\star+1-V}(X)
\cong
\uHom_{\Ftwo}\!\bigl({H\uFtwo}_{\star}(X),\Ftwo\bigr)
\]
of $RO(G)$-graded Mackey functors.
\item There is an isomorphism of $RO(G)$-graded $R_G$-modules
\[
D(R_G)\cong \Sigma^{1-V}R_G.
\]
There is a linear map $\epsilon\colon(R_G)_{V-1}\longrightarrow\Ftwo$ such that multiplication followed by $\epsilon$ gives perfect pairings
\[
(R_G)_\gamma\otimes_{\Ftwo}(R_G)_{V-1-\gamma}\longrightarrow\Ftwo,
\qquad x\otimes y\longmapsto\epsilon(xy)
\]
for every $\gamma\in RO(G)$. The ring $R_G$ is injective as an $RO(G)$-graded module over itself.
\end{enumerate}
\end{thmx}
\end{samepage}

The proof of the self-injectivity of $R_G$ was shown to us by Foling Zou (see also \cite{HogleMayZou}). For a finite subgroup of $SO(3)$, the shift $W-2$ gives the following explicit form of the coefficient-ring duality.

\begin{cor}[\cite{HogleMayZou}]\label{cor:SO3-HF2-selfinjective}
Let $G\mathrel{\leq} SO(3)$ be a finite subgroup, and let $W$ denote the restriction to $G$ of the standard three-dimensional representation of $SO(3)$. Then there is an isomorphism of $RO(G)$-graded $\pi_\star^G H\uFtwo$-modules
\[
D\!\left(\pi_\star^G H\uFtwo\right)
\cong\Sigma^{3-W}\pi_\star^G H\uFtwo.
\]
In particular, the $RO(G)$-graded coefficient ring $\pi_\star^G H\uFtwo$ is injective as an $RO(G)$-graded module over itself.
\end{cor}

The module-duality formula \labelcref{eq:MAndersonDual} also applies to $H\uZ$-module slices in equivariant constructions. We discuss this in \cref{sec:integral-applications}.

\subsection{Outline of Paper}

In \cref{sec:Section2}, we recall Anderson duality and its universal coefficient exact sequences and then use Liu's $p$-local splitting theorem \cite{Liu23} to prove Sylow detection for integral Anderson duality. In \cref{sec:Verification}, we prove the ``if'' direction of \cref{thm:maintheorem} by constructing a dualizing shift for every group in its list. These constructions use the basic cyclic and dihedral shifts and Sylow detection.

In \cref{sec:HZperiodicity}, we prove part \textup{(1)} of \cref{thm:maintheoremC}, which determines the integral periodicity lattice for every finite group. Liu's splitting theorem reduces the problem to $p$-groups, where geometric fixed points and Angeltveit's theorem \cite{AngeltveitPicard} give the fixed-point dimension criterion. A character calculation reduces the test to cyclic $p$-subgroups and computes the quotient rank. We also compute the periodicity lattices for cyclic $p$-groups and dihedral $2$-groups explicitly.

In \cref{sec:LocalClassification}, we prove part \textup{(1)} of \cref{thm:maintheoremB} by expressing all integral dualizing shifts as an affine translate of the periodicity lattice in \cref{thm:maintheoremC}. We also determine the shifts for $A_4$, $S_4$, and $A_5$ explicitly. In \cref{sec:integral-applications}, we prove \cref{thm:UCT} and \cref{cor:MAndersonDual}, which give integral universal coefficients and module duality, and record their coefficient and reduction-modulo-$n$ consequences.

Section~6 establishes the subgroup obstructions used in the group classification. In \cref{sec:Section4}, we prove part \textup{(1)} of \cref{localcriterion}, the local criterion for Anderson-dualizing groups. The subgroup obstructions and a restriction argument give the Sylow and normalizer conditions. Bauer's realization theorem \cite{Bauer89} and Sylow detection prove the ``if'' direction.

In \cref{sec:Section5}, we begin the proof of the ``only if'' direction of \cref{thm:maintheorem}, classifying groups satisfying the local criterion by induction. Burnside's normal $p$-complement theorem gives the first reduction, and we then handle the normal-complement and normalizer cases. In \cref{sec:classification}, we show that the remaining groups have trivial odd core and use Walter's theorem \cite{Wal69} and the Gorenstein--Walter classification \cite{GW65I,GW65II,GW65III} to complete the proof of \cref{thm:maintheorem}.

In \cref{sec:HF2application}, we prove the mod-$2$ results. Section~10.1 identifies the Anderson dual of $H\uFtwo$ and uses Liu's splitting theorem to prove Sylow $2$-detection. Section~10.2 proves part \textup{(2)} of \cref{thm:HF2periodicity-intro}, using integral periodicity on $2$-groups and Sylow detection to determine the mod-$2$ periodicity lattice for every finite group. Section~10.3 establishes the local $2$-group classification. In Section~10.4, we combine this classification with Bauer's realization theorem to prove part \textup{(2)} of \cref{thm:HF2groups-v2} and part \textup{(2)} of \cref{thm:HF2shifts-v2}. Finally, Section~10.5 proves \cref{thm:HF2selfinjective-v2}, using the Anderson universal coefficient sequence and its compatibility with the module structure to obtain coefficient-ring duality and graded self-injectivity.

\subsection{Further Questions}

The module-duality formula \eqref{eq:MAndersonDual} leads to two further classification problems.

\begin{quest}\rm\label{Quest:ModuleSelfDuality}
For a fixed Anderson-dualizing group $G$, which $H\uZ$-modules $M$ satisfy
\[
I_{\mathbb Z}^{G}(M)\simeq\Sigma^W M
\]
for some $W\in RO(G)$?
\end{quest}

At the Eilenberg--Mac Lane level, this asks which cohomological Mackey functors give Anderson self-dual spectra. Examples include $H\uZ/n$ and, for cyclic $p$-groups, Zeng's forms of $\uZ$ \cite[Theorem~5.9]{Zen18}.

\begin{quest}\rm\label{Quest:EMSuspend}
For a finite group $G$, which virtual representations $W\in RO(G)$ have the property that $\Sigma^W H\uZ$ is an Eilenberg--Mac Lane spectrum?
\end{quest}

For an Anderson-dualizing group, every such suspension is Anderson self-dual and supplies an example for \cref{Quest:ModuleSelfDuality}.

The slice-level duality formula obtained from \eqref{eq:MAndersonDual} also suggests an approach to Anderson duality for chromatic spectra.

\begin{quest}\rm\label{Quest:EOAndersonDual}
Let $X$ be a chromatically localized spectrum arising from an equivariant construction, for example one related to the models of Lubin--Tate spectra and their homotopy fixed points $EO_n(G)=E_n^{hG}$ in \cite{BHSZ21}. Can one determine whether its Anderson dual is a suspension of $X$ or differs from one by an exotic Picard twist? When the construction uses $H\uZ$-module slices, can this duality be detected and computed from the slices?
\end{quest}

The module-duality formula \eqref{eq:MAndersonDual} suggests a strategy: identify the slices, compute their Anderson duals, and compare the resulting spectral sequence with the original slice spectral sequence. This strategy is currently being developed, in joint work with Mike Hill, for spectra related to the $EO$-theories constructed in \cite{BHSZ21}.
\par

\subsection{Acknowledgements}

The authors would like to heartily thank Mike Hill for introducing them to this problem during the 2016 Namboodiri Lectures at the University of Chicago, when they were still graduate students, and for the many helpful conversations throughout the years. We are grateful to Foling Zou for showing us the proof of the self-injectivity of $R_G$ during a visit to the Isaac Newton Institute in January 2025, and to Eric Hogle, Clover May, and Foling Zou for sharing a draft of their paper with us. We would also like to thank John Greenlees for helpful conversations.

The authors would like to thank the Isaac Newton Institute for Mathematical Sciences, Cambridge, for support and hospitality during the programme Introductory Workshop for Equivariant Homotopy Theory in Context, where some of the work on this paper was undertaken.

The fourth author is partially supported by NSF Grant No. DMS-2404828 and NSF CAREER Grant No. DMS-2541934. The fifth author is partially supported by the AMS Centennial Research Fellowship and NSF Grant DMS-2506247.

\section{Reduction to Sylow subgroups} \label{sec:Section2}
We recall the construction of Anderson duality and then prove that \cref{mainquestion} can be checked on Sylow subgroups (\cref{decomposetopsubgps}). This reduction allows the later classification to focus on cyclic $p$-groups and dihedral $2$-groups.

\subsection{Anderson duality}\label{sec:AndersonBackground}

A basic question in stable homotopy theory is how the homology and cohomology represented by a spectrum $E$ are related. For $E=H\mathbb F_2$ and any spectrum $X$, cohomology is the linear dual of homology:
\[
H^n(X;\mathbb F_2)\cong
\operatorname{Hom}_{\mathbb F_2}\!\bigl(H_n(X;\mathbb F_2),\mathbb F_2\bigr).
\]
For $E=H\mathbb Z$, the universal coefficient sequence has an additional $\operatorname{Ext}$ term. Anderson duality gives such a sequence for any spectrum $E$.

Let $G$ be a finite group. For an injective abelian group $A$, equivariant Brown representability gives a $G$-spectrum $I_A^G$ with natural isomorphisms
\[
[Y,I_A^G]^G\cong
\operatorname{Hom}_{\mathbb Z}\bigl(\pi_0^G Y,A\bigr).
\]
The injective resolution
\[
0\longrightarrow\mathbb Z\longrightarrow\mathbb Q
\longrightarrow\mathbb Q/\mathbb Z\longrightarrow0
\]
induces a map $I_{\mathbb Q}^G\to I_{\mathbb Q/\mathbb Z}^G$. Its homotopy fiber is the Anderson-dualizing spectrum $I_{\mathbb Z}^G$ \cite{And69,BC76}. For a $G$-spectrum $E$, define
\[
I_{\mathbb Z}^G(E)=F_G(E,I_{\mathbb Z}^G).
\]
When $G=e$, we omit the superscript $G$.
\par

For spectra $E$ and $X$ and every integer $n$, the defining fiber sequence gives the natural short exact sequence \cite{Kai71}
\begin{equation} \label{ses:classicalUCT}
0 \longrightarrow \operatorname{Ext}^1_{\mathbb{Z}}\!\bigl(E_{n-1}(X),\mathbb{Z}\bigr)
\longrightarrow (I_{\mathbb{Z}}E)^n(X)
\longrightarrow \operatorname{Hom}_{\mathbb{Z}}\!\bigl(E_n(X),\mathbb{Z}\bigr)
\longrightarrow 0.
\end{equation}

If $I_{\mathbb Z}E\simeq\Sigma^k E$, this becomes a relation between $E$-homology and $E$-cohomology. In particular, the equivalence
\[
I_{\mathbb Z}(H\mathbb Z)\simeq H\mathbb Z
\]
recovers the classical integral universal coefficient sequence by taking $E=H\mathbb Z$.

Anderson duality has supplied universal coefficient theorems and self-duality statements throughout chromatic homotopy theory. Anderson established the duality of complex and real $K$-theory \cite{And69}. Stojanoska proved Anderson self-duality over $\mathbb Z[1/2]$ for $\mathrm{Tmf}[1/2]$, using Grothendieck--Serre duality and descent \cite{Sto12}. Heard and Stojanoska gave an equivariant refinement of the $K$-theoretic duality and applied it to $KO$ and $K(1)$-local duality \cite{HS14}. Greenlees and Stojanoska related Anderson-duality shifts to Gorenstein duality and local cohomology \cite{GS18}.

For a Mackey functor $M$, we form $\uHom_{\mathbb Z}(M,\mathbb Z)$ and $\uExt^1_{\mathbb Z}(M,\mathbb Z)$ by taking ordinary $\operatorname{Hom}_{\mathbb Z}$ and $\operatorname{Ext}^1_{\mathbb Z}$ at every orbit. Their restrictions are induced by the transfers of $M$, and their transfers are induced by its restrictions. Conjugation by $g$ is induced by conjugation by $g^{-1}$ on $M$. For background on Mackey functors, see \cite{Greenlees92,GreenleesMay92}.

For $G$-spectra $E$ and $X$, applying the defining fiber sequence orbitwise gives a natural short exact sequence of $RO(G)$-graded Mackey functors
\begin{equation} \label{ses:equivariantUCT}
0 \longrightarrow \uExt^1_{\mathbb{Z}}\!\bigl(\underline{E}_{\star-1}(X),\mathbb{Z}\bigr)
\longrightarrow \underline{(I_{\mathbb{Z}}^G E)}^\star(X)
\longrightarrow \uHom_{\mathbb{Z}}\!\bigl(\underline{E}_{\star}(X),\mathbb{Z}\bigr)
\longrightarrow 0.
\end{equation}

The constant Mackey functor $\uZ$ has value $\mathbb Z$ at every orbit, identity restrictions, and transfers given by subgroup indices. Its dual $\uZ^*$ has identity transfers and restrictions given by subgroup indices. Since $\uZ$ is free abelian at every orbit,
\[
\uExt^1_{\mathbb Z}(\uZ,\mathbb Z)=0,
\qquad
\uHom_{\mathbb Z}(\uZ,\mathbb Z)=\uZ^*.
\]
Let $E=H\uZ$ and $X=S^0$. The universal coefficient sequence \labelcref{ses:equivariantUCT} then shows that $I_{\mathbb Z}^G(H\uZ)$ has integer-graded homotopy concentrated in degree zero, with degree-zero Mackey functor $\uZ^*$. Therefore
\[
I_{\mathbb Z}^G(H\uZ)\simeq H\uZ^*.
\]

Anderson-dualizing shifts already play a role in equivariant homotopy theory. For $G=C_2$, Ricka proved that $2\sigma-2$ is an Anderson-dualizing shift, where $\sigma$ is the sign representation, and used this duality to study integral Morava $K$-theories with reality \cite{Ric16}. Hill and Meier used the same duality in their slice-theoretic computation of the Anderson dual of $\mathrm{Tmf}_1(3)$ \cite{HM17}. These applications use the duality of Eilenberg--Mac Lane spectra as input for duality calculations of more complicated spectra.

For cyclic $p$-groups, Zeng established the duality with shift $\lambda(1)-2$, where $\lambda(1)$ is the faithful two-dimensional rotation representation \cite{Zen18}. Basu and Ghosh also use integral Anderson duality for cyclic groups of odd square-free order in their computation of equivariant cohomology \cite{BasuGhosh24}.

\subsection{Sylow detection}\label{sec:SylowDetection}

We first record compatibility with restriction to subgroups.
\par
\par

\begin{prop} \label[prop]{subgpcrit}
Let $H\mathrel{\leq} G$. If $(G,V)$ is an Anderson-dualizing pair, then $(H,V|_H)$ is an Anderson-dualizing pair. More precisely, if
\[
H\uZ^*\wedge S^V\simeq H\uZ
\]
as $G$-spectra, then
\[
H\uZ^*\wedge S^{V|_H}\simeq H\uZ
\]
as $H$-spectra.
\end{prop}

\begin{proof}
The proof is immediate by applying the restriction functor $\res_H^G(-)$.  \end{proof}

\begin{thm}[Sylow detection] \label[thm]{decomposetopsubgps}
Let $G$ be a finite group and let $V\in RO(G)$. Then
\[
H\uZ^*\wedge S^V\simeq H\uZ
\]
as $G$-spectra if and only if, for every prime $p$ and every Sylow $p$-subgroup $P\mathrel{\leq} G$ (with $P=e$ when $p\nmid |G|$), one has
\[
H\uZ^*\wedge S^{V|_P}\simeq H\uZ
\]
as $P$-spectra.
\end{thm}

The proof uses Liu's $p$-local splitting theorem in the following form \cite[Theorem~4.6, Lemma~4.7, and Proposition~4.9]{Liu23}. We will also use this description for the Sylow-detection arguments in \cref{sec:HF2application}.
\par

\begin{prop} \label{prop:LiuLimit}
Let $X$ be an $H\uZ$-module and let $p$ be a prime. After localizing at $p$, $\pi_*^G(X)$ is naturally identified with the limit of the diagram consisting of
\[
\pi_*^H(X)_{(p)}
\]
for all $p$-subgroups $H\mathrel{\leq} G$, and whose maps are the restriction maps
\[
\res_{H_1}^{H_2}\colon \pi_*^{H_2}(X)_{(p)}\longrightarrow
\pi_*^{H_1}(X)_{(p)}, \qquad H_1\mathrel{\leq} H_2,
\]
and conjugation maps
\[
c_g\colon \pi_*^H(X)_{(p)}\longrightarrow
\pi_*^{gHg^{-1}}(X)_{(p)}.
\]
Moreover, if $K\mathrel{\leq} G$, then the diagram computing $\pi_*^K(X)$ is the subdiagram obtained by restricting to the $p$-subgroups of $K$ and to conjugations by elements of $K$. The map between the limits agrees with the restriction map
\[
\res_K^G\colon \pi_*^G(X)_{(p)}\longrightarrow \pi_*^K(X)_{(p)}.
\]
\end{prop}

\begin{lem}
\label[lem]{lem:LiuModuleVanishing}
Let $\mathcal F$ be one of the families appearing in the localized splitting used in \cref{prop:LiuLimit}. If
\[
\widetilde E\mathcal F\wedge H\uZ\simeq *
\]
after the relevant localization, then for every $H\uZ$-module $X$ one also has
\[
\widetilde E\mathcal F\wedge X\simeq *.
\]
Consequently, the geometric-fixed-point terms in Liu's general splitting diagram vanish for $H\uZ$-modules, leaving precisely the full diagram of $p$-subgroups with restrictions and conjugations stated in \cref{prop:LiuLimit}.
\end{lem}

\begin{proof}
Regard $\widetilde E\mathcal F\wedge H\uZ$ as a right $H\uZ$-module. Since $X$ is a left $H\uZ$-module, after the relevant smashing localization there is an equivalence
\[
\widetilde E\mathcal F\wedge X
\simeq
\bigl(\widetilde E\mathcal F\wedge H\uZ\bigr)
\wedge_{H\uZ}X.
\]
The first factor is contractible after the localization under consideration, so the relative smash product is contractible as well. It follows that every geometric term which vanishes for $H\uZ$ also vanishes for every $H\uZ$-module.
\end{proof}

\begin{proof}[Proof of \cref{prop:LiuLimit}]
First suppose that $p\mid |G|$ and invert the prime divisors of $|G|$ other than $p$. By \cite[Theorem~4.6]{Liu23}, $\pi_*^G(X)$ is the limit of a finite diagram consisting of the homotopy groups at $p$-subgroups together with geometric-fixed-point terms arising in the iterated splitting.  The latter terms vanish by \cite[Proposition~4.9]{Liu23} and \cref{lem:LiuModuleVanishing}.  Liu's restriction-compatibility result \cite[Lemma~4.7]{Liu23} identifies the maps between the surviving subgroup terms with the stated restriction maps.  The conjugation maps are those in \cite[Theorem~4.6]{Liu23}, and the same compatibility identifies subgroup restriction with passage to the corresponding subdiagram.  Finally, further localization at $p$ commutes with this finite limit, since localization of abelian groups is exact.

If $p\nmid |G|$, the only $p$-subgroup is $e$.  Since $\underline\pi_*(X)$ is a module over the constant Green functor $\uZ$, it is cohomological.  After $p$-localization, restriction identifies
\[
\pi_*^K(X)_{(p)}\cong \bigl(\pi_*^e(X)_{(p)}\bigr)^K
\]
for every $K\mathrel{\leq} G$: on invariants an inverse is $|K|^{-1}\tr_e^K$.  This is precisely the limit of the diagram with its single subgroup $e$ and the conjugation maps.
\end{proof}

\begin{rmk}\rm \label[rmk]{rmk:SylowLimit}
Let $P\mathrel{\leq} G$ be a Sylow $p$-subgroup. Projection to the $P$-component identifies the inverse limit in \cref{prop:LiuLimit} with the subgroup of fusion-stable elements of $\pi_*^P(X)_{(p)}$ as in \cite[Theorem~4.19]{Liu23}. Explicitly, an element $x\in\pi_*^P(X)_{(p)}$ is stable if, for every $H\mathrel{\leq} P$ and every $g\in G$ such that $gHg^{-1}\mathrel{\leq} P$, one has
\[
c_g\!\left(\res_H^P x\right)
=
\res_{gHg^{-1}}^P x.
\]
In general this stable-elements subgroup need not equal
\[
\bigl(\pi_*^P(X)_{(p)}\bigr)^{W_G(P)},
\qquad
W_G(P)=N_G(P)/P,
\]
because elements of $G$ outside $N_G(P)$ can induce additional fusion maps between proper subgroups of $P$.
\end{rmk}

\begin{proof}[Proof of \cref{decomposetopsubgps}]
The ``only if'' direction follows by restriction. For the ``if'' direction, put
\[
X=H\uZ\wedge S^{-V}.
\]
The hypothesis gives $X|_H\simeq H\uZ^*$ for every subgroup $H$ of prime-power order, including $e$. Applying \cref{prop:LiuLimit} to $X|_K$, for every $K\mathrel{\leq} G$ and every prime $p$, gives
\[
\pi_n^K(X)_{(p)}=0\qquad(n\neq0).
\]
Therefore $X$ has integer-graded homotopy concentrated in degree zero.

Put $M=\underline\pi_0(X)$, with $M(K)=\pi_0^K(X)$, and choose a generator $1_e\in M(e)\cong\Z$. The $G$-action on $M(e)$ is trivial because its restriction to every Sylow subgroup is trivial. For every subgroup $H$ of prime-power order, choose the generator $e_H\in M(H)$ characterized by
\[
\res_e^H(e_H)=|H|\,1_e.
\]
The local identifications with $\uZ^*$ give
\[
\res_L^H(e_H)=[H:L]e_L,
\qquad
c_g(e_H)=e_{gHg^{-1}}.
\]

Fix $K\mathrel{\leq} G$, a prime $p$, and a Sylow $p$-subgroup $Q\mathrel{\leq} K$. For every $a\in\Z_{(p)}$, the elements
\[
x_H=a\frac{|Q|}{|H|}e_H,
\qquad H\mathrel{\leq} K\text{ a }p\text{-subgroup},
\]
form a compatible family in the diagram of \cref{prop:LiuLimit}. Every compatible family is determined by its $Q$-component. Consequently, \cref{prop:LiuLimit} identifies restriction with an isomorphism
\[
\res_Q^K\colon
M(K)_{(p)}\xrightarrow{\;\cong\;}M(Q)_{(p)}.
\]

Since $M$ is cohomological \cite[Proposition~2.5]{Angeltveit22},
\[
\tr_Q^K\res_Q^K=[K:Q].
\]
The index $[K:Q]$ is a unit in $\Z_{(p)}$, so $\tr_Q^K$ is also an isomorphism after $p$-localization. The local identification with $\uZ^*$ makes $\tr_e^Q$ an isomorphism. Therefore
\[
\tr_e^K=\tr_Q^K\tr_e^Q
\]
is an isomorphism after localization at every prime, and therefore integrally.

Identify $M(K)$ with $\Z$ using the generator $\tr_e^K(1_e)$. Transfer transitivity makes all transfers identities. The cohomological relation makes restrictions multiplication by subgroup indices, and conjugation preserves these generators. Therefore $M\cong\uZ^*$, so
\[
X\simeq H\uZ^*.
\]
Smashing with $S^V$ proves the result.
\end{proof}
\par

\begin{rmk}\label[rmk]{lem:HZPostnikov}\rm
The equivalence constructed in the preceding proof may be chosen to be $H\uZ$-linear. More generally, let $X$ and $Y$ be $H\uZ$-modules whose integer-graded homotopy Mackey functors vanish outside degree zero. Every isomorphism $\underline\pi_0(X)\cong\underline\pi_0(Y)$ of Mackey functors is induced by an $H\uZ$-linear equivalence $X\simeq Y$.

To see this, write $M=\underline\pi_0(X)$ and $N=\underline\pi_0(Y)$. The Eilenberg--Mac Lane correspondence for $H\uZ$-modules concentrated in degree zero gives a natural bijection
\[
[X,Y]_{H\uZ}\cong\operatorname{Hom}_{\uZ}(M,N).
\]
The $\uZ$-actions on $M$ and $N$ are integer multiplication at each orbit, so every map of Mackey functors between them is $\uZ$-linear. Therefore the given isomorphism lifts to an $H\uZ$-module map $X\to Y$. This map induces isomorphisms on all integer-graded homotopy Mackey functors, so it is an equivalence.
\end{rmk}
\par

\section{Verification of Anderson-dualizing groups in \texorpdfstring{\protect\cref{thm:maintheorem}}{Theorem B}} \label{sec:Verification}

In this section, we prove the ``if'' direction of \cref{thm:maintheorem} by showing that every group appearing in the list of \cref{thm:maintheorem} is Anderson-dualizing.  The ``only if'' direction is proved later in the paper.  The proof here uses the Sylow detection theorem, \cref{decomposetopsubgps}.  For each group in the list, we construct a virtual representation whose restriction to every Sylow subgroup is one of the two basic local shifts in \cref{pgrouprep}.

For $T_{3^n}$ and $A_5$, the action on $C_m$ is trivial. For $O_{3^n}$, it factors through $O_{3^n}/T_{3^n}\cong C_2$.

\begin{notation} \rm \label{notation:cyclicGroupRep}
Let \(C_m\) be a cyclic group with a chosen generator.  We write \(\lambda(i,m)\) for the real two-dimensional representation on which the generator acts by rotation through angle \(2\pi i/m\).  When \(m\) is clear from context, we also write \(\lambda(i)\) for \(\lambda(i,m)\).  In particular, \(\lambda(1,m)\) is the faithful rotation representation of \(C_m\).  We use the convention that \(\lambda(1,1)=2\), where \(2\) denotes the trivial two-dimensional real representation.
\end{notation}

We first record the general geometric fact used in the cyclic and dihedral calculations.

\begin{lem}\label[lem]{lem:orbit-sphere}
Let \(G\) be a finite group and let \(W\) be a faithful real \(d\)-dimensional representation of \(G\).  Suppose that every element of \(G\) has determinant \(+1\) on \(W\), and that \(S^W/K\) is homeomorphic to \(S^d\) for every subgroup \(K\mathrel{\leq} G\).  Then
\[
H\uZ\wedge S^{d-W}\simeq H\uZ^*.
\]
Equivalently,
\[
H\uZ^*\wedge S^{W-d}\simeq H\uZ.
\]
\end{lem}

\begin{proof}
Choose an orientation of \(W\).  Every element of \(G\) then acts on the representation sphere \(S^W\) by a map of degree \(+1\).  We first show that, if \(L\mathrel{\leq} K\mathrel{\leq} G\), then the natural map
\[
q_{L,K}\colon S^W/L \longrightarrow S^W/K
\]
has degree \([K:L]\).

Let
\[
q_H\colon S^W\longrightarrow S^W/H
\]
be the quotient map for a subgroup \(H\mathrel{\leq} G\). The union of the fixed-point sets of nonidentity elements of \(H\) is a proper closed subset of \(S^W\).  Choose a point in the complement of its image in \(S^W/H\).  Its preimage consists of \(|H|\) points, and each local degree is \(+1\) because the action of \(H\) preserves the orientation of \(W\).  Therefore
\[
\deg(q_H)=|H|.
\]
The quotient maps satisfy
\[
q_K=q_{L,K}\circ q_L.
\]
Therefore, by multiplicativity of degree,
\[
\deg(q_{L,K})=
\frac{\deg(q_K)}{\deg(q_L)}=
\frac{|K|}{|L|}=
[K:L].
\]

For every \(K\mathrel{\leq} G\) and every integer \(q\), there are natural isomorphisms
\[
\begin{aligned}
\pi_q^K\bigl(H\uZ\wedge S^{d-W}\bigr)
&\cong \widetilde H_K^{d-q}(S^W;\uZ) \\
&\cong \widetilde H^{d-q}(S^W/K;\Z).
\end{aligned}
\]
Since the restriction maps of \(\uZ\) are identities, the equivariant cellular cochain complex computing \(\widetilde H_K^*(S^W;\uZ)\) is the ordinary cellular cochain complex of \(S^W/K\).  Under this identification, restriction from \(K\) to \(L\) is pullback along
\[
q_{L,K}\colon S^W/L\longrightarrow S^W/K.
\]
Since \(S^W/K\cong S^d\), the cohomology group above is \(\Z\) for \(q=0\) and vanishes for \(q\neq0\).  In degree zero, restriction is pullback on top cohomology along \(q_{L,K}\), so it is multiplication by \(\deg(q_{L,K})=[K:L]\).  By \cite[Proposition~2.5]{Angeltveit22}, the homotopy Mackey functor of \(H\uZ\wedge S^{d-W}\) is cohomological. Therefore the corresponding transfer maps are all identity maps.  It follows that this homotopy Mackey functor is \(\uZ^*\), concentrated in degree zero, and
\[
H\uZ\wedge S^{d-W}\simeq H\uZ^*.
\]
By \cref{lem:HZPostnikov}, this equivalence may be chosen to be $H\uZ$-linear.
Smashing this equivalence with \(S^{W-d}\) gives the equivalent formulation in the statement.
\end{proof}

We now apply \cref{lem:orbit-sphere} to the cyclic and dihedral local computations.  These are the basic inputs for all of the verifications below.

\begin{prop} \label[prop]{pgrouprep} \hfill
\begin{enumerate}
\item Let \(G=C_m\) be any finite cyclic group, and let \(\lambda(1,m)\) be its faithful two-dimensional rotation representation.  Then
\[
H\uZ^*\wedge S^{\lambda(1,m)-2}\simeq H\uZ.
\]

\item Let \(G=C_2\ltimes C_{2^n}\) be a dihedral \(2\)-group, where \(C_2\) acts on \(C_{2^n}\) by inversion.  Let \(\rho\) be the two-dimensional dihedral representation whose restriction to \(C_{2^n}\) is \(\lambda(1,2^n)\), and let \(\sigma\) be the sign representation \(G\to C_2\) with kernel \(C_{2^n}\).  Then
\[
H\uZ^*\wedge S^{\rho+\sigma-3}\simeq H\uZ.
\]
\end{enumerate}
\end{prop}

\begin{proof}
We apply \cref{lem:orbit-sphere} in both cases.

For (1), take \(G=C_m\), \(W=\lambda(1,m)\), and \(d=2\).  The unit circle \(S(W)\) is acted on by rotations, and \(S(W)/K\) is again a circle for every subgroup \(K\mathrel{\leq} C_m\).  We have\[
S^W/K\cong \Sigma\bigl(S(W)/K\bigr)\cong S^2.
\]
All rotations preserve orientation.  Therefore \cref{lem:orbit-sphere} gives
\[
H\uZ\wedge S^{2-\lambda(1,m)}\simeq H\uZ^*,
\]
and therefore
\[
H\uZ^*\wedge S^{\lambda(1,m)-2}\simeq H\uZ.
\]

For (2), take \(G=C_2\ltimes C_{2^n}\), \(W=\rho\oplus\sigma\), and \(d=3\).  Elements of the rotation subgroup act by rotations on \(\rho\) and trivially on \(\sigma\), so they have determinant \(+1\) on \(W\).  A reflection has determinant \(-1\) on \(\rho\), and it acts by \(-1\) on \(\sigma\).  Therefore a reflection has determinant \((-1)(-1)=+1\) on \(W\).  
It follows that the image of \(G\) lies in \(SO(3)\) and is the usual rotational dihedral group.  Every subgroup \(K\) is cyclic or dihedral.  If \(K\) is cyclic, the quotient of the unit sphere \(S(W)\cong S^2\) by rotations about an axis is again \(S^2\).  If \(K\) is dihedral, let \(R\trianglelefteq K\) be its cyclic rotation subgroup.  We first have \(S(W)/R\cong S^2\).  The residual involution \(K/R\) is induced by a half-turn and is orientation-preserving, so its quotient is again a \(2\)-sphere.  Therefore \(S(W)/K\cong S^2\) in all cases, and we have \[
S^W/K\cong S^3.
\]
Therefore \cref{lem:orbit-sphere} gives
\[
H\uZ\wedge S^{3-\rho-\sigma}\simeq H\uZ^*,
\]
and therefore
\[
H\uZ^*\wedge S^{\rho+\sigma-3}\simeq H\uZ.
\]
\end{proof}

\begin{notation} \rm \label{notation:standardDihedralRep}
Let \(D_{2m}=C_2\ltimes C_m\), where the generator of \(C_2\) acts on \(C_m\) by inversion.  We write \(\rho_m\) for the standard two-dimensional dihedral representation: the subgroup \(C_m\) acts through \(\lambda(1,m)\), and the generator of \(C_2\) acts by complex conjugation.  If \(m=1\), then \(\rho_1=1\oplus \sigma\), where \(\sigma\) is the sign representation of \(C_2\).
\end{notation}

\begin{prop} \label{prop:cyclic-two-family-good}
Let \(G=C_{2^n}\ltimes C_m\) be a group in the cyclic \(2\)-primary family of \cref{thm:maintheorem}.  Then \(G\) is Anderson-dualizing.
\end{prop}

\begin{proof}
We first construct the shift.  If \(n=0\), then \(G=C_m\), and we set
\[
V=\lambda(1,m)-2.
\]
This is the cyclic shift in \cref{pgrouprep}.  We now assume \(n\geq 1\).  Let \(\alpha\) be a generator of \(C_{2^n}\).  Since \(m\) is odd, every automorphism of \(C_m\) of order at most \(2\) restricts on each Sylow subgroup of \(C_m\) either to the identity or to inversion.  Therefore we can write
\[
C_m=C_{m_1}\times C_{m_2},
\]
where \(m_1\) and \(m_2\) are coprime, \(\alpha\) acts by inversion on \(C_{m_1}\), and \(\alpha\) acts trivially on \(C_{m_2}\).

Let \(\sigma\) denote the sign representation of \(G\) obtained from the quotient
\[
G \longrightarrow G/(C_{2^{n-1}}\times C_m)\cong C_2.
\]
Here \(C_{2^{n-1}}\mathrel{\leq} C_{2^n}\) is the subgroup generated by the square of a generator of \(C_{2^n}\). Since the action of \(C_{2^n}\) on \(C_m\) has order at most \(2\), this subgroup acts trivially on \(C_m\), so \(C_{2^{n-1}}\times C_m\) is the kernel of the quotient above. By abuse of notation, we also write \(\sigma\) for the sign representation of \(C_{2^n}\) obtained by restricting this \(G\)-representation to \(C_{2^n}\).

Let \(V_1\) be the pullback of \(\lambda(1,2^n)\) from \(C_{2^n}\).  Let \(V_2\) be the pullback of \(\rho_{m_1}\) from
\[
(C_{2^n}/C_{2^{n-1}})\ltimes C_{m_1},
\]
and let \(V_3\) be the pullback of \(\lambda(1,m_2)\) from \(C_{m_2}\).  Set
\[
V=V_1+V_2+V_3-\sigma-5.
\]
We will check that \((G,V)\) is Anderson-dualizing.

By \cref{decomposetopsubgps}, it is enough to check the restrictions of \(V\) to the Sylow subgroups of \(G\).  The Sylow \(2\)-subgroup is \(C_{2^n}\).  On this subgroup, \(V_1\) restricts to \(\lambda(1,2^n)\), \(V_2\) restricts to \(1\oplus\sigma\), and \(V_3\) restricts to the trivial two-dimensional representation.  Therefore
\[
V|_{C_{2^n}}
=\lambda(1,2^n)+(1+\sigma)+2-\sigma-5
=\lambda(1,2^n)-2.
\]
This is the cyclic shift in \cref{pgrouprep}.

Now let \(P\) be a Sylow \(p\)-subgroup for an odd prime \(p\).  Since \(m_1\) and \(m_2\) are coprime, \(P\) lies in exactly one of \(C_{m_1}\) and \(C_{m_2}\).  If \(P\mathrel{\leq} C_{m_1}\), then \(V_2|_P\cong \lambda(1,|P|)\), while \(V_1\), \(\sigma\), and \(V_3\) restrict to trivial representations of dimensions \(2\), \(1\), and \(2\).  Therefore
\[
V|_P=2+\lambda(1,|P|)+2-1-5=\lambda(1,|P|)-2.
\]
If \(P\mathrel{\leq} C_{m_2}\), the same calculation with \(V_3\) in place of \(V_2\) gives
\[
V|_P=\lambda(1,|P|)-2.
\]
All Sylow restrictions are Anderson-dualizing by \cref{pgrouprep}, so \cref{decomposetopsubgps} proves the result.
\end{proof}

\begin{prop} \label{prop:dihedral-two-family-good}
Let \(G=D_{2^{n+1}}\ltimes C_m\) be a group in the dihedral \(2\)-primary family of \cref{thm:maintheorem}.  Then \(G\) is Anderson-dualizing.
\end{prop}

\begin{proof}
Write
\[
D_{2^{n+1}}=\langle \alpha,\beta\mid \beta^{2^n}=1,\ \alpha^2=1,\ \alpha\beta\alpha=\beta^{-1}\rangle,
\]
where \(\beta\) generates the rotation subgroup.  Since \(\operatorname{Aut}(C_m)\) is abelian, the action of \(D_{2^{n+1}}\) on \(C_m\) factors through the abelian quotient
\[
T=D_{2^{n+1}}/\langle \beta^2\rangle\cong C_2\times C_2.
\]
The action condition in \cref{thm:maintheorem} gives a decomposition
\[
C_m=C_{m_1}\times C_{m_2}\times C_{m_3}\times C_{m_4}
\]
with pairwise coprime factors, where \(\alpha\) acts by inversion on \(C_{m_1}\times C_{m_2}\) and trivially on \(C_{m_3}\times C_{m_4}\), while \(\beta\) acts by inversion on \(C_{m_1}\times C_{m_3}\) and trivially on \(C_{m_2}\times C_{m_4}\).

Let \(V_0\) be the standard two-dimensional dihedral representation of \(D_{2^{n+1}}\), and let \(\sigma_1,\sigma_2,\sigma_3\) be the three sign representations of \(T\), chosen so that \(\sigma_1\) has kernel \(C_{2^n}\), \(\sigma_2(\alpha)=1\), \(\sigma_2(\beta)=-1\), and \(\sigma_3=\sigma_1\sigma_2\).  Let \(V_1,V_2,V_3\) be the pullbacks of the standard dihedral representations from
\[
(T/\langle \alpha\beta\rangle)\ltimes C_{m_1},\qquad
(T/\langle \beta\rangle)\ltimes C_{m_2},\qquad
(T/\langle \alpha\rangle)\ltimes C_{m_3},
\]
and let \(V_4\) be the pullback of \(\lambda(1,m_4)\) from \(C_{m_4}\).  Set
\[
V=V_0+V_1+V_2+V_3+V_4-\sigma_2-\sigma_3-8.
\]
We will check that \((G,V)\) is Anderson-dualizing.

By \cref{decomposetopsubgps}, it is enough to check Sylow subgroups.  First restrict to the Sylow \(2\)-subgroup \(D_{2^{n+1}}\).  The representations \(V_1,V_2,V_3\) restrict respectively to
\[
1\oplus \sigma_3,
\qquad
1\oplus \sigma_1,
\qquad
1\oplus \sigma_2,
\]
and \(V_4\) restricts to the trivial two-dimensional representation.  Therefore
\[
\begin{aligned}
V|_{D_{2^{n+1}}}
&=V_0+(1+\sigma_3)+(1+\sigma_1)+(1+\sigma_2)+2-\sigma_2-\sigma_3-8 \\
&=V_0+\sigma_1-3.
\end{aligned}
\]
This is the dihedral shift in \cref{pgrouprep}.

Now let \(P\) be a Sylow \(p\)-subgroup for an odd prime \(p\).  Then \(P\) lies in exactly one of the factors \(C_{m_i}\).  If \(P\mathrel{\leq} C_{m_i}\), the representation \(V_i\) restricts to \(\lambda(1,|P|)\), and the other \(V_j\)'s restrict to trivial two-dimensional representations.  The representations \(V_0\), \(\sigma_2\), and \(\sigma_3\) restrict to trivial representations of dimensions \(2\), \(1\), and \(1\).  In all four cases,
\[
V|_P=\lambda(1,|P|)-2.
\]
The Sylow detection theorem now gives the result.
\end{proof}

Recall that \(T_{3^n}\) is the group
\[
T_{3^n}=C_{3^n}\ltimes E,\qquad E=C_2\times C_2,
\]
where a generator of \(C_{3^n}\) acts on the three nonidentity elements of \(E\) by a \(3\)-cycle.  Equivalently, the action is the composite \(C_{3^n}\twoheadrightarrow C_3\mathrel{\leq} \operatorname{Aut}(E)\cong S_3\).  In the special case \(n=1\), this group is \(T_3\cong A_4\).

\begin{prop} \label{prop:T-family-good}
Let \(G=T_{3^n}\ltimes C_m\) be a group in the \(T_{3^n}\)-family of \cref{thm:maintheorem}.  Then \(G\) is Anderson-dualizing.
\end{prop}

\begin{proof}
Let \(K=T_{3^n}\).  We first check that the action of \(K\) on \(C_m\) is trivial.  Since \(\operatorname{Aut}(C_m)\) is abelian, the action map \(K\to \operatorname{Aut}(C_m)\) sends conjugate elements of \(E\) to the same automorphism.  The three nonidentity elements of \(E\) are conjugate under \(C_{3^n}\), so their images are equal.  Their product is the identity in \(E\), so the common image has order dividing \(3\).  By the action condition in \cref{thm:maintheorem}, the common image also has order dividing \(2\).  Therefore \(E\) acts trivially.  The image of \(C_{3^n}\) is trivial as well, because every element in the image has order dividing both \(3^n\) and \(2\).  Therefore \(K\) acts trivially on \(C_m\), and \(G=K\times C_m\).

We now construct the shift.  Let \(W\) be the pullback of the standard three-dimensional representation of
\[
K/C_{3^{n-1}}\cong A_4\mathrel{\leq} SO(3).
\]
Let \(\lambda_3\) be the pullback of \(\lambda(1,3)\) from
\[
K/(C_{3^{n-1}}\ltimes E)\cong C_3,
\]
let \(\lambda_{3^n}\) be the pullback of \(\lambda(1,3^n)\) from \(K/E\cong C_{3^n}\), and let \(\lambda_m\) be the pullback of \(\lambda(1,m)\) from \(C_m\).  Set
\[
V=W-\lambda_3+\lambda_{3^n}+\lambda_m-5.
\]

We check Sylow restrictions.  Let \(P=E\) be the Sylow \(2\)-subgroup.  The restriction of \(W\) to \(E\) is the sum of the three nontrivial sign representations of \(E\).  The representations \(\lambda_3\), \(\lambda_{3^n}\), and \(\lambda_m\) restrict to trivial two-dimensional representations.  Therefore
\[
V|_E=\sigma_1+\sigma_2+\sigma_3-2+2+2-5
=\sigma_1+\sigma_2+\sigma_3-3.
\]
This is the dihedral shift for \(C_2\times C_2\) from \cref{pgrouprep}.

Let \(P=C_{3^n}\) be the Sylow \(3\)-subgroup.  The restriction of \(W\) to \(P\) is
\[
1\oplus \lambda(3^{n-1},3^n),
\]
and \(\lambda_3|_P\cong \lambda(3^{n-1},3^n)\), while \(\lambda_{3^n}|_P\cong \lambda(1,3^n)\) and \(\lambda_m|_P\) is trivial of dimension \(2\).  Therefore
\[
V|_P=(1+\lambda(3^{n-1},3^n))-\lambda(3^{n-1},3^n)+\lambda(1,3^n)+2-5
=\lambda(1,3^n)-2.
\]

Finally, if \(P\mathrel{\leq} C_m\) is a Sylow \(p\)-subgroup for an odd prime \(p\neq 3\), then \(\lambda_m|_P\cong \lambda(1,|P|)\), while \(W\), \(\lambda_3\), and \(\lambda_{3^n}\) restrict to trivial representations of dimensions \(3\), \(2\), and \(2\).  Therefore
\[
V|_P=3-2+2+\lambda(1,|P|)-5=\lambda(1,|P|)-2.
\]
All Sylow restrictions are basic shifts, so \cref{decomposetopsubgps} proves the proposition.
\end{proof}

\begin{rmk}\label{rmk:shifts}\rm
The shifts constructed in this section need not have the form $U-|U|$ for an actual representation $U$. For a finite subgroup of $SO(3)$, the shift in \cref{listremark} has this form. In the $T_{3^n}$-family with $n>1$, the shift above contains the term
\[
-\lambda(1,3)+\lambda(1,3^n),
\]
where these representations are inflated from the cyclic quotients of $T_{3^n}$. The negative coefficient of the nontrivial summand $\lambda(1,3)$ shows why the classification of all shifts must be stated in $RO(G)$.
\end{rmk}

Recall next that \(O_{3^n}\) is the group
\[
O_{3^n}=D_{2\cdot 3^n}\ltimes E,
\qquad E=C_2\times C_2,
\]
where the action of \(D_{2\cdot 3^n}\) on \(E\) is the composite
\[
D_{2\cdot 3^n}\twoheadrightarrow D_6\cong \operatorname{Aut}(E).
\]
Equivalently, if \(D_{2\cdot 3^n}=C_2\ltimes C_{3^n}\), then the generator of \(C_{3^n}\) acts on the three nonidentity elements of \(E\) by a \(3\)-cycle, and the generator of \(C_2\) acts by a transposition.  The subgroup
\[
T_{3^n}=C_{3^n}\ltimes E
\]
has index \(2\) in \(O_{3^n}\).  In the special case \(n=1\), the group is
\[
O_3\cong D_6 \ltimes (C_2 \times C_2) \cong S_4.
\]

\begin{prop} \label{prop:O-family-good}
Let \(G=O_{3^n}\ltimes C_m\) be a group in the \(O_{3^n}\)-family of \cref{thm:maintheorem}.  Then \(G\) is Anderson-dualizing.
\end{prop}

\begin{proof}
Let
\[
K=O_{3^n},\qquad T=T_{3^n}\mathrel{\leq} K,\qquad E=C_2\times C_2.
\]
The subgroup \(T\) has index \(2\) in \(K\) and acts trivially on \(C_m\) by the proof of \cref{prop:T-family-good}.  Therefore the action of \(K\) on \(C_m\) factors through
\[
K/T\cong C_2.
\]
Since \(m\) is odd, we can write
\[
C_m=C_{m_1}\times C_{m_2},
\]
where \(m_1\) and \(m_2\) are coprime, the nontrivial element of \(K/T\) acts by inversion on \(C_{m_1}\), and it acts trivially on \(C_{m_2}\).

Let \(W\) be the pullback of the standard three-dimensional representation of
\[
K/C_{3^{n-1}}\cong O_3\cong S_4\mathrel{\leq} SO(3).
\]
Let \(\widetilde\rho_{3^n}\) be the pullback of the standard two-dimensional dihedral representation \(\rho_{3^n}\) from
\[
K/E\cong D_{2\cdot 3^n},
\]
and let \(\widetilde\rho_3\) be the pullback of \(\rho_3\) from
\[
K/(C_{3^{n-1}}\ltimes E)\cong D_6.
\]

Let \(\sigma\) denote the one-dimensional representation of \(G=K\ltimes C_m\) obtained from the quotient
\[
G\longrightarrow G/(T\ltimes C_m)\cong K/T\cong C_2
\]
and the nontrivial character of \(C_2\).  Equivalently,
\[
\ker(\sigma)=T\ltimes C_m.
\]
By abuse of notation, we also write \(\sigma\) for the restriction of this representation to subgroups of \(G\).

Let \(V_1\) be the pullback of the standard two-dimensional dihedral representation \(\rho_{m_1}\) from
\[
(K/T)\ltimes C_{m_1},
\]
where \(K/T\cong C_2\) acts on \(C_{m_1}\) by inversion.  Let \(V_2\) be the pullback of \(\lambda(1,m_2)\) from \(C_{m_2}\).  We use the conventions that \(\rho_1=1\oplus\sigma\) for the group \(C_2\), and that \(\lambda(1,1)\) is the trivial two-dimensional representation.

Set
\[
V=W+\widetilde\rho_{3^n}-\widetilde\rho_3-\sigma+V_1+V_2-6.
\]
We will check that \((G,V)\) is Anderson-dualizing.

By \cref{decomposetopsubgps}, it is enough to check one representative of each Sylow conjugacy class.  Since \((m,|K|)=1\), we may choose the Sylow \(2\)- and Sylow \(3\)-subgroups inside \(K\), while the Sylow subgroups for primes \(q\neq 2,3\) lie inside \(C_m\).

First let \(P\) be a Sylow \(2\)-subgroup.  Then \(P\mathrel{\leq} K\), and
\[
P\cap T=E,\qquad P/E\cong K/T\cong C_2.
\]
The group \(P\) is a dihedral group of order \(8\).  The restriction \(\sigma|_P\) is the sign representation associated to the quotient
\[
P\longrightarrow P/E\cong C_2.
\]
This sign representation has kernel \(E\).  We will only use it in the terms that factor through \(P/E\).

The representation \(W|_P\) is the restriction of the standard three-dimensional representation of \(SO(3)\) to the dihedral subgroup \(P\mathrel{\leq} S_4\mathrel{\leq} SO(3)\). Therefore \(W|_P-3\) is the basic dihedral shift from \cref{pgrouprep}.

The representations \(\widetilde\rho_{3^n}\) and \(\widetilde\rho_3\) have the same restriction to \(P\).  Both representations factor through \(P/E\cong C_2\), and both restrict as
\[
1\oplus \sigma|_P.
\]
Also,
\[
V_1|_P\cong 1\oplus \sigma|_P,
\qquad
V_2|_P\cong 2.
\]
Therefore
\[
\begin{aligned}
V|_P
&=
W|_P+(1+\sigma|_P)-(1+\sigma|_P)-\sigma|_P
 +(1+\sigma|_P)+2-6  \\
&=W|_P-3.
\end{aligned}
\]
This is Anderson-dualizing by the dihedral case of \cref{pgrouprep}.

Next let \(P=C_{3^n}\) be the Sylow \(3\)-subgroup.  Then
\[
W|_P\cong 1\oplus \lambda(3^{n-1},3^n),
\]
because \(W\) factors through
\[
K/C_{3^{n-1}}\cong O_3\cong S_4.
\]
Moreover,
\[
\widetilde\rho_{3^n}|_P\cong \lambda(1,3^n),
\qquad
\widetilde\rho_3|_P\cong \lambda(3^{n-1},3^n).
\]
The representations \(\sigma\), \(V_1\), and \(V_2\) restrict to trivial representations of dimensions \(1\), \(2\), and \(2\), respectively.  Therefore
\[
\begin{aligned}
V|_P
&=
(1+\lambda(3^{n-1},3^n))
+\lambda(1,3^n)
-\lambda(3^{n-1},3^n)
-1+2+2-6  \\
&=\lambda(1,3^n)-2.
\end{aligned}
\]
This is Anderson-dualizing by the cyclic case of \cref{pgrouprep}.

Finally, let \(P\) be a Sylow \(q\)-subgroup of \(C_m\), where \(q\) is an odd prime different from \(3\).  Since \(m_1\) and \(m_2\) are coprime, \(P\) lies in exactly one of \(C_{m_1}\) and \(C_{m_2}\).

If \(P\mathrel{\leq} C_{m_1}\), then
\[
V_1|_P\cong \lambda(1,|P|),
\qquad
V_2|_P\cong 2.
\]
If \(P\mathrel{\leq} C_{m_2}\), then
\[
V_1|_P\cong 2,
\qquad
V_2|_P\cong \lambda(1,|P|).
\]
The representations pulled back from \(K\) restrict trivially to \(P\), so
\[
W|_P\cong 3,\qquad
\widetilde\rho_{3^n}|_P\cong 2,\qquad
\widetilde\rho_3|_P\cong 2,\qquad
\sigma|_P\cong 1.
\]
In either case, we get
\[
V|_P
=
3+2-2-1+\lambda(1,|P|)+2-6
=
\lambda(1,|P|)-2.
\]
This is Anderson-dualizing by the cyclic case of \cref{pgrouprep}.

Every Sylow restriction of \(V\) is one of the basic Anderson-dualizing shifts in \cref{pgrouprep}.  Therefore \cref{decomposetopsubgps} proves that \((G,V)\) is Anderson-dualizing.
\end{proof}

\begin{prop} \label{prop:A5-family-good}
Let \(G=A_5\ltimes C_m\) be a group in the \(A_5\)-family of \cref{thm:maintheorem}.  Then \(G\) is Anderson-dualizing.
\end{prop}

\begin{proof}
The action of \(A_5\) on \(C_m\) is trivial because \(\operatorname{Aut}(C_m)\) is abelian and \(A_5\) is perfect (it has no nontrivial abelian quotients).  Therefore \(G=A_5\times C_m\).

Let \(W\) be the standard three-dimensional representation of \(A_5\mathrel{\leq} SO(3)\), and let \(\lambda_m\) be the pullback of \(\lambda(1,m)\) from \(C_m\).  Set
\[
V=W+\lambda_m-5.
\]

By \cref{decomposetopsubgps}, it is enough to check Sylow subgroups.  Let \(P\mathrel{\leq} A_5\) be a Sylow subgroup.  If \(P\) is cyclic, then \(W|_P\cong 1\oplus \lambda(1,|P|)\) after choosing a generator of \(P\) suitably, so \(V|_P\cong \lambda(1,|P|)-2\).  If \(P\cong C_2\times C_2\), then \(W|_P\) is the sum of the three nontrivial sign representations of \(P\), so \(V|_P\) is the dihedral shift for \(C_2\times C_2\).  In both cases \(\lambda_m\) restricts to the trivial two-dimensional representation, and \cref{pgrouprep} applies.

If \(P\mathrel{\leq} C_m\) is a Sylow \(p\)-subgroup for an odd prime \(p\), then \(W|_P\) is trivial of dimension \(3\), and \(\lambda_m|_P\cong \lambda(1,|P|)\).  Therefore
\[
V|_P=3+\lambda(1,|P|)-5=\lambda(1,|P|)-2.
\]
Again \cref{pgrouprep} applies, and the Sylow detection theorem proves the claim.
\end{proof}

\begin{proof}[Proof of the ``if'' direction of \cref{thm:maintheorem}]
Let \(G=K\ltimes C_m\) be a group in the list of \cref{thm:maintheorem}.  The list has five possible choices for \(K\).  They are \(K=C_{2^n}\), \(D_{2^{n+1}}\), \(T_{3^n}\), \(O_{3^n}\), and \(A_5\).  The required shifts are given by \cref{prop:cyclic-two-family-good}, \cref{prop:dihedral-two-family-good}, \cref{prop:T-family-good}, \cref{prop:O-family-good}, and \cref{prop:A5-family-good}.  It follows that every group in the list of \cref{thm:maintheorem} is Anderson-dualizing.
\end{proof}

The next two sections separate the periodicity problem from the classification of Anderson-dualizing shifts. We first determine the periodicity lattice of $H\uZ$ for every finite group. We then use this lattice to classify all Anderson-dualizing shifts.

\section{The periodicity lattice of \texorpdfstring{$H\uZ$}{HZ}}
\label{sec:HZperiodicity}

In this section, we prove part \textup{(1)} of \cref{thm:maintheoremC}. We first reduce the periodicity problem to $p$-groups. We then show that fixed-point dimensions on cyclic $p$-subgroups determine those on every $p$-subgroup. We finish by computing the periodicity lattices for cyclic $p$-groups and dihedral $2$-groups explicitly.

\begin{thm}[Sylow detection for $H\uZ$-periodicity]
\label[thm]{thm:HZperiodicity}
Let $G$ be a nontrivial finite group and let $V,W\in RO(G)$. Then
\[
H\uZ\wedge S^V\simeq H\uZ\wedge S^W
\]
as $G$-spectra if and only if, for every $p\mid |G|$ and every Sylow $p$-subgroup $P\mathrel{\leq} G$, one has
\[
H\uZ\wedge S^{V|_P}\simeq H\uZ\wedge S^{W|_P}
\]
as $P$-spectra. When these conditions hold, the global equivalence may be chosen to be $H\uZ$-linear.
\end{thm}

\begin{proof}
The ``only if'' direction follows by restriction. For the ``if'' direction, put
\[
X=H\uZ\wedge S^{V-W}.
\]
The Sylow hypotheses give $X|_H\simeq H\uZ$ for every subgroup $H$ of prime-power order. As in the proof of \cref{decomposetopsubgps}, \cref{prop:LiuLimit} shows that $X$ has integer-graded homotopy concentrated in degree zero.

Put $M=\underline\pi_0(X)$. The action of $G$ on $M(e)\cong\Z$ is trivial. For every $p$-subgroup $H$, restriction to $e$ identifies $M(H)$ with $M(e)$. Under these identifications, all restriction and conjugation maps in Liu's diagram are identities. \Cref{prop:LiuLimit} and its compatibility with restriction therefore give
\[
\res_e^K\colon
M(K)_{(p)}\xrightarrow{\;\cong\;}M(e)_{(p)}
\]
for every $K\mathrel{\leq} G$ and every prime $p$. It follows that $\res_e^K$ is an isomorphism integrally.

These identifications make all restrictions identities, and the cohomological relation makes transfers multiplication by subgroup indices. Therefore $M\cong\uZ$. \Cref{lem:HZPostnikov} gives an $H\uZ$-linear equivalence $X\simeq H\uZ$. Smashing with $S^W$ proves the result.
\end{proof}
\par

\begin{lem}\label[lem]{geofixedhz}
Let $H$ be a nontrivial finite $p$-group. Then $\Phi^H H\uZ$ is connective,
\[
\pi_0\Phi^H H\uZ\cong\mathbb F_p,
\]
and it is not equivalent to any nonzero integer suspension of itself.
\end{lem}

\begin{proof}
Let $\mathcal P$ be the family of proper subgroups of $H$. The spectrum
\[
\Phi^H H\uZ\simeq (\widetilde E\mathcal P\wedge H\uZ)^H
\]
is connective. The cofiber sequence defining $\widetilde E\mathcal P$ gives
\[
\pi_0\Phi^H H\uZ
\cong
\mathbb Z\big/\bigl\langle\tr_K^H(1)\mid K\mathrel{<} H\bigr\rangle.
\]
For the constant Mackey functor, $\tr_K^H(1)=[H:K]$. These indices generate the ideal $(p)$, since every proper subgroup has index divisible by $p$ and a maximal subgroup has index $p$. Therefore $\pi_0\Phi^H H\uZ\cong\mathbb F_p$. A connective spectrum with nonzero $\pi_0$ cannot be equivalent to a nonzero integer suspension of itself.
\end{proof}

The converse in the following proposition follows from Angeltveit's theorem for cohomological Mackey functors.

\begin{prop}
\label[prop]{prop:pgroup-periodicity}
Let $P$ be a finite $p$-group and let $V,W\in RO(P)$. Then
\[
H\uZ\wedge S^V\simeq H\uZ\wedge S^W
\]
if and only if
\[
|V^H|=|W^H|
\]
for every subgroup $H\mathrel{\leq} P$.
\end{prop}

\begin{proof}
Set $U=V-W$. If $H\uZ\wedge S^U\simeq H\uZ$, the underlying equivalence gives $|U|=0$. For each nontrivial $H\mathrel{\leq} P$, geometric fixed points and \cref{geofixedhz} give $|U^H|=0$.

Conversely, suppose that $|U^H|=0$ for every $H\mathrel{\leq} P$. The invertible $P$-spectrum $S^U$ then has zero dimension function. Since $\uZ$ is a cohomological Mackey functor, \cite[Theorem~B]{AngeltveitPicard} gives
\[
S^U\wedge H\uZ\simeq H\uZ.
\]
\end{proof}

We next explain why cyclic $p$-subgroups are enough in \cref{thm:maintheoremC}.

\begin{lem}
\label[lem]{lem:cyclic-dimension-detection}
Let $P$ be a finite $p$-group and let $U\in RO(P)$. If
\[
|U|=0
\qquad\text{and}\qquad
|U^C|=0
\]
for every nontrivial cyclic subgroup $C\mathrel{\leq} P$, then $|U^H|=0$ for every subgroup $H\mathrel{\leq} P$.
\end{lem}

\begin{proof}
Let $\chi_U$ be the virtual character of $U$. For every cyclic subgroup $C\mathrel{\leq} P$, set
\[
a_C(U)=\sum_{\langle g\rangle=C}\chi_U(g).
\]
Grouping the elements of $C$ according to the subgroup they generate gives
\[
|C|\,|U^C|=\sum_{D\mathrel{\leq} C}a_D(U).
\]
The hypothesis gives $a_e(U)=|U|=0$. Induction on $|C|$ now gives $a_C(U)=0$ for every cyclic subgroup $C\mathrel{\leq} P$. For any subgroup $H\mathrel{\leq} P$, grouping its elements in the same way gives
\[
|H|\,|U^H|
=\sum_{h\in H}\chi_U(h)
=\sum_{\substack{C\mathrel{\leq} H\\ C\text{ cyclic}}}a_C(U)
=0.
\]
Therefore $|U^H|=0$.
\end{proof}

\begin{lem}
\label[lem]{lem:fixed-dimension-rank}
Let $C_0=e$, and let $C_1,\ldots,C_r$ represent the $G$-conjugacy classes of nontrivial cyclic subgroups of prime-power order. The fixed-dimension map
\[
d_G\colon RO(G)\longrightarrow\mathbb Z^{r+1},
\qquad
U\longmapsto\bigl(|U^{C_0}|,|U^{C_1}|,\ldots,|U^{C_r}|\bigr),
\]
has rank $r+1$.
\end{lem}

\begin{proof}
Order the subgroups so that their orders are nondecreasing. Let
\[
\psi_j=\operatorname{Ind}_{C_j}^G(1)
\]
be the permutation character of $G/C_j$, and choose a generator $g_i$ of $C_i$. The value $\psi_j(g_i)$ is zero unless $C_i$ is conjugate to a subgroup of $C_j$. Therefore the matrix
\[
\bigl(\psi_j(g_i)\bigr)_{0\leq i,j\leq r}
\]
is upper triangular, and its $i$th diagonal entry is $|N_G(C_i):C_i|$. The characters $\psi_0,\ldots,\psi_r$ are therefore linearly independent. Frobenius reciprocity gives
\[
|U^{C_i}|=\langle\chi_U,\psi_i\rangle_G.
\]
Since the character inner product is nondegenerate on $RO(G)\otimes\mathbb Q$, the fixed-dimension functionals are linearly independent. This proves the result.
\end{proof}

\begin{proof}[Proof of part \textup{(1)} of \cref{thm:maintheoremC}]
The result is immediate when $G=e$, so suppose that $G\ne e$. Let $V\in\Lambda_G(H\uZ)$. The underlying nonequivariant equivalence gives $|V|=0$. For every nontrivial subgroup $P\mathrel{\leq} G$ of prime-power order, restriction to $P$ and \cref{prop:pgroup-periodicity} give $|V^P|=0$. This proves one implication.

Conversely, suppose that $|V|=0$ and $|V^C|=0$ for every nontrivial cyclic subgroup $C\mathrel{\leq} G$ of prime-power order. Applying \cref{lem:cyclic-dimension-detection} to every Sylow subgroup shows that $|V^P|=0$ for every subgroup $P\mathrel{\leq} G$ of prime-power order, since every such subgroup is conjugate into a Sylow subgroup. Therefore \cref{prop:pgroup-periodicity} gives
\[
H\uZ\wedge S^{V|_P}\simeq H\uZ
\]
for every Sylow subgroup $P\mathrel{\leq} G$. The reverse implication in \cref{thm:HZperiodicity} proves that $V\in\Lambda_G(H\uZ)$. It also shows that the equivalence may be chosen to be $H\uZ$-linear.

For the rank statement, the description of $\Lambda_G(H\uZ)$ gives
\[
RO(G)/\Lambda_G(H\uZ)\cong\operatorname{im}(d_G).
\]
The image is a subgroup of a free abelian group, so it is free abelian. By \cref{lem:fixed-dimension-rank}, its rank is the number of $G$-conjugacy classes of cyclic subgroups of prime-power order, including the trivial subgroup.
\end{proof}

\begin{exam}\label{exam:C6-periodicity}\rm
Let $G=C_6$, let $\sigma$ be its nontrivial one-dimensional real representation, and let $\lambda(i,6)$ be as in \cref{notation:cyclicGroupRep}. The virtual representation
\[
U=-2+2\sigma-\lambda(1,6)+\lambda(2,6)
\]
has underlying dimension zero. Its fixed-point dimension is also zero on the subgroups $C_2$ and $C_3$. Therefore $U\in\Lambda_{C_6}(H\uZ)$ by \cref{thm:maintheoremC}. On the other hand, $|U^{C_6}|=-2$. This example shows why the criterion tests cyclic subgroups of prime-power order rather than all cyclic subgroups.
\end{exam}

\subsection{\texorpdfstring{Cyclic $p$-groups}{Cyclic p-groups}}

Let $G=C_{p^n}$, and choose a generator $\gamma\in G$. We use the notation from \cref{notation:cyclicGroupRep}: for an integer $i$, the real two-dimensional representation
\[
\lambda(i,p^n)
\]
is the representation on which $\gamma$ acts by rotation through angle $2\pi i/p^n$. The representations $\lambda(i,p^n)$ and $\lambda(-i,p^n)$ are isomorphic, and $\lambda(i,p^n)$ is faithful if and only if $(i,p)=1$.

\begin{notation}\rm \label{notation:cyclicDsets}
For $0\leq j\leq n-1$, let $\mathcal D_j$ denote the set of isomorphism classes of nontrivial two-dimensional rotation representations $\lambda(i,p^n)$ with
\[
v_p(i)=j.
\]
Equivalently, the representations in $\mathcal D_j$ have kernel of order $p^j$. When $p=2$ and $j=n-1$, the representation $\lambda(2^{n-1},2^n)$ is the sum of two copies of the sign representation. We write $\sigma$ for this sign representation, so that
\[
\lambda(2^{n-1},2^n)=2\sigma.
\]

Let $V\in RO(C_{p^n})$. We define $\mu_j(V)$ to be the total coefficient of the rotation representations in $\mathcal D_j$. More explicitly, if $p$ is odd, this is the sum of the coefficients of the irreducible real representations in $\mathcal D_j$. If $p=2$ and $j=n-1$, then
\[
\mu_{n-1}(V)=\frac{1}{2}\cdot
\{\text{coefficient of }\sigma\text{ in }V\}.
\]
The functions $\mu_j$ are additive on $RO(C_{p^n})$ and may take negative values. In the exceptional case $p=2$ and $j=n-1$, $\mu_{n-1}(V)$ may be a half-integer. We write $|V|$ for the underlying virtual real dimension of $V$.
\end{notation}

\begin{cor}
\label[cor]{cor:cyclicperiodicity}
Let $G=C_{p^n}$ with $n\geq 1$. Then $U\in\Lambda_G(H\uZ)$ if and only if
\[
|U|=0
\qquad\text{and}\qquad
\mu_j(U)=0\quad(0\leq j\leq n-1).
\]
More explicitly, the lattice has the following bases.
\begin{enumerate}
\item Suppose that $p$ is odd. Choose one representation $\lambda_j\in\mathcal D_j$ for every $0\leq j\leq n-1$. Then
\[
\Lambda_G(H\uZ)=
\bigoplus_{j=0}^{n-1}
\ \bigoplus_{\lambda\in\mathcal D_j\setminus\{\lambda_j\}}
\mathbb Z(\lambda-\lambda_j).
\]
In particular,
\[
\operatorname{rank}\Lambda_G(H\uZ)=\frac{p^n-1}{2}-n.
\]

\item Suppose that $p=2$ and $n\geq 2$. Choose one representation $\lambda_j\in\mathcal D_j$ for every $0\leq j\leq n-2$. Then
\[
\Lambda_G(H\uZ)=
\bigoplus_{j=0}^{n-2}
\ \bigoplus_{\lambda\in\mathcal D_j\setminus\{\lambda_j\}}
\mathbb Z(\lambda-\lambda_j).
\]
In particular,
\[
\operatorname{rank}\Lambda_G(H\uZ)=2^{n-1}-n.
\]

\item If $G=C_2$, then $\Lambda_G(H\uZ)=0$.
\end{enumerate}
\end{cor}

\begin{proof}
Write the trivial coefficient of $U$ as $a$ and set $b_j=\mu_j(U)$. Since every representation in $\mathcal D_j$ has the same fixed-point dimensions,
\[
|U|=a+2\sum_{j=0}^{n-1}b_j
\]
and, for $1\leq k\leq n$,
\[
|U^{C_{p^k}}|=a+2\sum_{j=k}^{n-1}b_j.
\]
By \cref{thm:maintheoremC}, $U\in\Lambda_G(H\uZ)$ exactly when all these expressions vanish. Their successive differences give $b_j=0$ for every $j$, and then $a=0$. This proves the coefficient criterion.

The kernel of the coefficient-sum map on each $\mathcal D_j$ has the displayed difference basis. There are $(p^n-1)/2$ rotation representations in the odd case and $2^{n-1}-1$ in the even case. Subtracting the number of valuation classes gives the two rank formulas. The coefficient criterion also gives $\Lambda_{C_2}(H\uZ)=0$.
\end{proof}

\subsection{\texorpdfstring{Dihedral $2$-groups}{Dihedral 2-groups}}

\begin{notation}\label{notation:dihedral-periodicity}\rm
Let
\[
G=C_2\ltimes C_{2^n}=\langle r,s\mid r^{2^n}=1,\ s^2=1,\ srs=r^{-1}\rangle.
\]
Here $C_{2^n}=\langle r\rangle$ is the rotation cyclic subgroup.

For $n\geq 1$, let $\sigma_1,\sigma_2,\sigma_3$ denote the three sign representations that factor through
\[
G/\langle r^2\rangle\cong C_2\times C_2.
\]
We choose the notation so that
\[
\sigma_1(r)=1,\qquad \sigma_1(s)=-1,
\]
\[
\sigma_2(r)=-1,\qquad \sigma_2(s)=1,
\]
and $\sigma_3=\sigma_1\sigma_2$. Here $\sigma_1$ is the sign representation with kernel $C_{2^n}$.

For $n\geq 2$ and $1\leq i\leq 2^n-1$, with $i\neq 2^{n-1}$, let $\rho(i,2^n)$ denote the two-dimensional dihedral representation on $\mathbb C$. If $\zeta=e^{2\pi i/2^n}$, then the action is given by
\[
r\cdot z=\zeta^i z,
\qquad
s\cdot z=\overline z.
\]
The representation $\rho(i,2^n)$ is isomorphic to $\rho(-i,2^n)$. When $i=2^{n-1}$, $\rho(2^{n-1},2^n)$ is reducible:
\[
\rho(2^{n-1},2^n)\cong \sigma_2\oplus \sigma_3.
\]
Therefore this case is not included among the two-dimensional irreducible summands below. Its contribution is recorded instead by the coefficients of $\sigma_2$ and $\sigma_3$.

For $0\leq j\leq n-2$, let $\mathcal E_j$ be the set of representations $\rho(i,2^n)$ with $v_2(i)=j$. If $V\in RO(G)$, let $\mu_j(V)$ be the total coefficient of the representations in $\mathcal E_j$. The function $\mu_j$ is additive on $RO(G)$.
\end{notation}

The periodicity criterion depends only on the fixed-point dimensions of these representations. The next lemma records the calculation we need.

\begin{lem}
\label[lem]{lem:dihedral-fixed-dimensions}
If $\rho(i,2^n)$ and $\rho(i',2^n)$ belong to the same $\mathcal E_j$, then
\[
|\rho(i,2^n)^H|=|\rho(i',2^n)^H|
\]
for every subgroup $H\mathrel{\leq} G$.
\end{lem}

\begin{proof}
Let $R=H\cap\langle r\rangle$. The two representations have the same kernel on $\langle r\rangle$. If $H=R$, their $H$-fixed subspaces both have dimension $2$ when $R$ lies in this kernel and dimension $0$ otherwise. If $H$ contains a reflection, their fixed subspaces both have dimension $1$ when $R$ lies in the common kernel and dimension $0$ otherwise.
\end{proof}

\begin{prop}
\label[prop]{prop:dihedral-periodicity-criterion}
Let $G=C_2\ltimes C_{2^n}$, with $n\geq1$, and let $U\in RO(G)$.
\begin{enumerate}
\item If $n=1$, so that $G=C_2\times C_2$, then
\[
U\in\Lambda_G(H\uZ)\quad\Longleftrightarrow\quad U=0.
\]
\item Suppose that $n\geq 2$, and write
\[
U=a+c_1\sigma_1+c_2\sigma_2+c_3\sigma_3
+\sum_{\rho}c_\rho\rho,
\]
where the sum runs over the two-dimensional irreducible representations. Then $U\in\Lambda_G(H\uZ)$ if and only if
\[
a=c_1=c_2=c_3=0
\qquad\text{and}\qquad
\mu_j(U)=0\quad(0\leq j\leq n-2).
\]
\end{enumerate}
\end{prop}

\begin{proof}
Suppose first that $n\geq 2$ and $U\in\Lambda_G(H\uZ)$. Restriction to the rotation subgroup and \cref{cor:cyclicperiodicity} give
\[
\mu_j(U)=0\quad(0\leq j\leq n-2),
\qquad
c_2+c_3=0,
\qquad
a+c_1=0.
\]
Every two-dimensional dihedral representation restricts to $1\oplus\tau$ on each reflection subgroup, where $\tau$ is the sign representation of $C_2$. Its total coefficient is $\sum_j\mu_j(U)=0$. Restricting to $\langle s\rangle$ and $\langle sr\rangle$, and using $\Lambda_{C_2}(H\uZ)=0$, gives
\[
a+c_2=0
\qquad\text{and}\qquad
a+c_3=0.
\]
These equations imply $a=c_1=c_2=c_3=0$.

Conversely, suppose that the displayed coefficient conditions hold. The virtual representation $U$ is a sum of differences of representations in the same $\mathcal E_j$. By \cref{lem:dihedral-fixed-dimensions}, all its fixed-point dimensions vanish. Therefore $U\in\Lambda_G(H\uZ)$ by \cref{thm:maintheoremC}.

If $n=1$, write $U=a+c_1\sigma_1+c_2\sigma_2+c_3\sigma_3$. Restriction to the three subgroups of order $2$ and the equality $\Lambda_{C_2}(H\uZ)=0$ give
\[
a+c_1=a+c_2=a+c_3=0,
\qquad
c_2+c_3=c_1+c_3=c_1+c_2=0.
\]
Therefore all four coefficients vanish.
\end{proof}

\begin{cor}
\label[cor]{cor:dihedralperiodicity}
Let $G=D_{2^{n+1}}=C_2\ltimes C_{2^n}$ with $n\geq 1$. If $n=1$, then $G=C_2\times C_2$,
\[
\Lambda_G(H\uZ)=0.
\]
If $n\geq 2$, choose one representation $\rho_j\in\mathcal E_j$ for every $0\leq j\leq n-2$. Then
\[
\Lambda_G(H\uZ)=
\bigoplus_{j=0}^{n-2}
\ \bigoplus_{\rho\in\mathcal E_j\setminus\{\rho_j\}}
\mathbb Z(\rho-\rho_j).
\]
In particular,
\[
\operatorname{rank}\Lambda_G(H\uZ)=2^{n-1}-n.
\]
\end{cor}

\begin{proof}
The case $n=1$ is \cref{prop:dihedral-periodicity-criterion}. If $n\geq 2$, the same proposition says that the trivial and sign coefficients vanish and that the coefficients in each $\mathcal E_j$ sum to zero. The kernel of each coefficient-sum map has the displayed difference basis. There are $2^{n-1}-1$ two-dimensional representations and $n-1$ valuation classes, which gives the rank formula.
\end{proof}

\section{Classification of all Anderson-dualizing shifts}
\label{sec:LocalClassification}

In this section, we use the periodicity lattice to classify all Anderson-dualizing shifts and prove part \textup{(1)} of \cref{thm:maintheoremB}.

Write $\mathcal S_G$ for the set of \textit{Anderson-dualizing shifts} for $H\uZ$.

\begin{prop}
\label[prop]{prop:affine-shifts}
Let $G$ be a finite group. If $V_0\in\mathcal S_G$, then
\[
\mathcal S_G=V_0+\Lambda_G(H\uZ).
\]
\end{prop}

\begin{proof}
If $U\in\Lambda_G(H\uZ)$, then
\[
H\uZ^*\wedge S^{V_0+U}
\simeq H\uZ\wedge S^U
\simeq H\uZ,
\]
so $V_0+U\in\mathcal S_G$. Conversely, if $V\in\mathcal S_G$, comparing the equivalences for $V$ and $V_0$ gives
\[
H\uZ\wedge S^{V-V_0}\simeq H\uZ.
\]
Therefore $V-V_0\in\Lambda_G(H\uZ)$.
\end{proof}

\subsection{\texorpdfstring{Cyclic $p$-groups}{Cyclic p-groups}}\label{sec:Cyclic}
\leavevmode\par

\begin{thm}[Cyclic restriction theorem]\label[thm]{cyclicpgroup}
Let $G=C_{p^n}$, with $n\geq1$, and let $V\in RO(G)$. Then $(G,V)$ is an Anderson-dualizing pair if and only if
\begin{enumerate}
\item $|V|=0$.
\item $\mu_0(V)=1$.
\item $\mu_j(V)=0$ for $1\leq j\leq n-1$.
\end{enumerate}
\end{thm}

\begin{proof}
Let $V_0=\lambda(1,p^n)-2$. By \cref{pgrouprep}, $V_0\in\mathcal S_G$. Therefore \cref{prop:affine-shifts} gives
\[
V\in\mathcal S_G
\quad\Longleftrightarrow\quad
V-V_0\in\Lambda_G(H\uZ).
\]
The coefficient criterion in \cref{cor:cyclicperiodicity} translates this condition into the three conditions in the statement.
\end{proof}

\subsection{\texorpdfstring{Dihedral $2$-groups}{Dihedral 2-groups}}
\label{sec:Dihedral}

\begin{thm}[Dihedral restriction theorem]\label[thm]{dihedralgoodrep}
Let $G=C_2\ltimes C_{2^n}$, with $n\geq1$, and let $V\in RO(G)$.
\begin{enumerate}
\item If $n=1$, so that $G=C_2\times C_2$, then $(G,V)$ is an Anderson-dualizing pair if and only if
\[
V=\sigma_1+\sigma_2+\sigma_3-3.
\]
\item If $n\geq 2$, then $(G,V)$ is an Anderson-dualizing pair if and only if
\begin{enumerate}
\item $|V|=0$.
\item $\mu_0(V)=1$.
\item $\mu_j(V)=0$ for $1\leq j\leq n-2$.
\item The coefficients of $\sigma_1,\sigma_2,\sigma_3$ in $V$ are $1,0,0$, respectively.
\end{enumerate}
\end{enumerate}
\end{thm}

\begin{proof}
For $n=1$, let $V_0=\sigma_1+\sigma_2+\sigma_3-3$. For $n\geq 2$, let
\[
V_0=\rho(1,2^n)+\sigma_1-3.
\]
In both cases, \cref{pgrouprep} gives $V_0\in\mathcal S_G$. By \cref{prop:affine-shifts}, $V$ is Anderson-dualizing if and only if $V-V_0\in\Lambda_G(H\uZ)$. The result now follows from \cref{prop:dihedral-periodicity-criterion}.
\end{proof}

\subsection{Proof of \texorpdfstring{part \textup{(1)} of \protect\cref{thm:maintheoremB}}{part (1) of Theorem C}}

\begin{proof}[Proof of part \textup{(1)} of \cref{thm:maintheoremB}]
First suppose that $\mathcal S_G\neq\varnothing$, and choose $V_0\in\mathcal S_G$. The underlying equivalence gives $|V_0|=0$. For every nontrivial cyclic subgroup $C\mathrel{\leq} G$ of prime-power order, \cref{subgpcrit} gives $V_0|_C\in\mathcal S_C$. The basic shift $\lambda(1,|C|)-2$ also lies in $\mathcal S_C$ by \cref{pgrouprep}. By \cref{prop:affine-shifts}, their difference lies in $\Lambda_C(H\uZ)$. Therefore \cref{thm:maintheoremC} gives $|V_0^C|=-2$.

Now \cref{prop:affine-shifts} gives
\[
V\in\mathcal S_G
\quad\Longleftrightarrow\quad
V-V_0\in\Lambda_G(H\uZ).
\]
By \cref{thm:maintheoremC}, the latter condition is equivalent to $|V|=0$ and $|V^C|=-2$ for every nontrivial cyclic subgroup $C\mathrel{\leq} G$ of prime-power order. The dimension conditions imply $\mathcal S_G\neq\varnothing$ by \cref{rmk:integral-dimension-existence}, which completes the proof.
\end{proof}

By Sylow detection, the equivalent coefficient conditions are those in \cref{cyclicpgroup} and \cref{dihedralgoodrep}, applied to the restrictions to Sylow subgroups.

For the polyhedral groups, the periodicity lattice gives the following classification of all shifts.

\begin{cor}\label[cor]{cor:polyhedral-shifts}
For $G=A_4$, $S_4$, or $A_5$, let $W$ denote the three-dimensional rotation representation of $G$. For $G=A_5$, let $W'$ denote its other irreducible three-dimensional real representation. Then
\[
\Lambda_{A_4}(H\uZ)=\Lambda_{S_4}(H\uZ)=0,
\qquad
\Lambda_{A_5}(H\uZ)=\mathbb Z(W-W').
\]
Consequently,
\[
\mathcal S_{A_4}=\{W-3\},
\qquad
\mathcal S_{S_4}=\{W-3\},
\]
and
\[
\mathcal S_{A_5}=W-3+\mathbb Z(W-W').
\]
\end{cor}

\begin{proof}
The ranks of $RO(A_4)$, $RO(S_4)$, and $RO(A_5)$ are $3$, $5$, and $5$. These groups have $2$, $4$, and $3$ conjugacy classes of nontrivial cyclic subgroups of prime-power order, respectively. Therefore \cref{thm:maintheoremC} gives
\[
\operatorname{rank}\Lambda_{A_4}(H\uZ)
=\operatorname{rank}\Lambda_{S_4}(H\uZ)=0,
\qquad
\operatorname{rank}\Lambda_{A_5}(H\uZ)=1.
\]
Both three-dimensional irreducible real representations of $A_5$ are faithful rotation representations. Their fixed subspaces on every nontrivial cyclic subgroup have dimension $1$, so $W-W'\in\Lambda_{A_5}(H\uZ)$. This difference is primitive in $RO(A_5)$, since $W$ and $W'$ are distinct irreducible representations. It therefore generates the rank-one lattice $\Lambda_{A_5}(H\uZ)$.

The constructions in \cref{sec:Verification} give the shift $W-3$ in each case. The classifications of shifts follow from \cref{prop:affine-shifts}.
\end{proof}

\subsection{Universal coefficients and module duality}
\label{sec:integral-applications}

We conclude with the universal coefficient and module-duality consequences stated in the introduction.
\par

\begin{proof}[Proof of \cref{thm:UCT}]
Apply \cref{ses:equivariantUCT} with $E=H\uZ$ and substitute the equivalence $I_{\mathbb Z}^G(H\uZ)\simeq\Sigma^{-V}H\uZ$.
\end{proof}

\begin{rmk}\label[rmk]{rmk:integral-coefficients}\rm
Let $(G,V)$ be an Anderson-dualizing pair. Taking $X=S^0$ in \cref{thm:UCT} gives
\begin{equation}\label{eq:piHZUCT}
0 \longrightarrow \uExt^1_{\mathbb{Z}}\!\bigl(\upi_{\star-1}H\uZ,\mathbb{Z}\bigr)
\longrightarrow \upi_{V-\star}H\uZ
\longrightarrow \uHom_{\mathbb{Z}}\!\bigl(\upi_{\star}H\uZ,\mathbb{Z}\bigr)
\longrightarrow 0.
\end{equation}
\end{rmk}

\begin{proof}[Proof of \cref{cor:MAndersonDual}]
Let $R$ be a $G$-commutative ring spectrum. The forgetful functor from $R$-modules to $G$-spectra has right adjoint $F_G(R,-)$, so for every $R$-module $M$ there is a natural equivalence
\begin{equation}\label{eq:module-adjunction}
I_{\mathbb Z}^G(M)\simeq F_R\bigl(M,I_{\mathbb Z}^G(R)\bigr).
\end{equation}
The duality equivalence for $(G,V)$ may be chosen to be $H\uZ$-linear by \cref{lem:HZPostnikov}. Setting $R=H\uZ$ gives
\[
\begin{aligned}
I_{\mathbb Z}^G(M)
&\simeq F_{H\uZ}(M,\Sigma^{-V}H\uZ)\\
&\simeq\Sigma^{-V}F_{H\uZ}(M,H\uZ).
\end{aligned}
\]
\end{proof}

Another motivation for studying the Anderson duality of $H\uZ$ comes from the slice filtration. In many equivariant constructions, including examples arising from norms of the Real bordism spectrum $MU_{\mathbb R}$, the slices are wedges of representation suspensions of Eilenberg--Mac Lane spectra \cite{HHR,BHLSZ25}. For an Anderson-dualizing pair $(G,V)$ and an $H\uZ$-module $H\underline M$, \cref{cor:MAndersonDual} gives
\[
I_{\mathbb Z}^G(\Sigma^W H\underline M)
\simeq\Sigma^{-W-V}F_{H\uZ}(H\underline M,H\uZ).
\]
\Cref{thm:maintheorem} and part \textup{(1)} of \cref{thm:maintheoremB} determine exactly which groups admit this uniform formula for $H\uZ$-module slices, and every representation shift that realizes it. For groups excluded by \cref{thm:maintheorem}, $H\uZ$ itself has no Anderson-dualizing representation shift. The classification therefore also identifies an obstruction to this form of slice-by-slice duality.

\begin{rmk}\label{rmk:reductionmodn}\rm
Let $(G,V)$ be an Anderson-dualizing pair, and let $n\geq2$. Applying $F_{H\uZ}(-,H\uZ)$ to the cofiber sequence
\[
H\uZ\xrightarrow{\,n\,}H\uZ\longrightarrow H\uZ/n
\]
gives
\[
F_{H\uZ}(H\uZ/n,H\uZ)\simeq\Sigma^{-1}H\uZ/n.
\]
\Cref{cor:MAndersonDual} therefore gives
\[
I_{\mathbb Z}^G(H\uZ/n)\simeq\Sigma^{-V-1}H\uZ/n.
\]
Consequently, $(G,V+1)$ is Anderson-dualizing for $H\uZ/n$. The complete classification for $n=2$ is given by part \textup{(2)} of \cref{thm:HF2groups-v2} and part \textup{(2)} of \cref{thm:HF2shifts-v2}.
\end{rmk}
\par

\section{\texorpdfstring{Minimal subgroup obstructions}{Minimal subgroup obstructions}}

One important application of \cref{cyclicpgroup} is to show that certain groups do not satisfy \cref{mainquestion}. 

\begin{prop} \label[prop]{blacklist}\hfill

\begin{enumerate}[label=\textup{\arabic*)}]
\item The following groups do not satisfy \cref{mainquestion}:
\begin{enumerate}[label=\textup{(\alph*)}]
\item $C_2\times C_2\times C_2$.
\item $C_4\times C_2$.
\item $C_p\times C_p$, where $p$ is odd.
\item $Q_{2^n}$ for $n\geq 3$, which is generated by two elements $a,b$ with relations $a^{2^{n-1}}=e$, $b^2=a^{2^{n-2}}$, $bab^{-1}=a^{-1}$ (when $n=3$, this is the usual quaternion group).
\end{enumerate}

\item Moreover, any group containing one of the groups listed above does not satisfy \cref{mainquestion}.
\end{enumerate}
\end{prop}

\begin{proof}
It is enough to prove the first assertion.  The second assertion follows from \cref{subgpcrit}.

First let $G=C_2\times C_2\times C_2$.  Let $x_\chi$ be the coefficient of a nontrivial real character $\chi$ of $G$ in $V$.  For each nontrivial element $u\in G$, restriction to $\langle u\rangle\cong C_2$ and \cref{cyclicpgroup} give
$$
\sum_{\chi(u)=-1}x_\chi=2.
$$
Summing over the seven nontrivial elements $u$, each nontrivial character appears exactly four times.  Therefore
$$
4\sum_{\chi\neq 1}x_\chi=14,
$$
which is impossible.

Next let $G=C_4\times C_2=\langle a,b\rangle$.  Let $x_{10},x_{01},x_{11}$ be the coefficients of the three nontrivial one dimensional characters factoring through $C_2\times C_2$.  Let $y_0,y_1$ be the coefficients of the two real two dimensional representations on which $a$ acts faithfully, while $b$ acts trivially for $y_0$ and by $-1$ for $y_1$.  Restricting to $\langle a\rangle$, $\langle ab\rangle$, $\langle b\rangle$, and $\langle a^2b\rangle$ gives
$$
y_0+y_1=1,\qquad x_{10}+x_{11}=0,
$$
$$
y_0+y_1=1,\qquad x_{10}+x_{01}=0,
$$
$$
x_{01}+x_{11}+2y_1=2,
$$
$$
x_{01}+x_{11}+2y_0=2.
$$
The last two equations give $y_0=y_1$, while the first gives $y_0+y_1=1$.  This is impossible since $y_0,y_1$ are integers.

Now let $G=C_p\times C_p$ with $p$ odd.  Nontrivial real irreducible representations are grouped by their kernel lines.  For each line $L\mathrel{\leq} G$, let $x_L$ be the total coefficient of the real irreducible representations with kernel $L$.  If $H\mathrel{\leq} G$ is a cyclic subgroup of order $p$, then restriction to $H$ and \cref{cyclicpgroup} give
$$
\sum_{L\neq H}x_L=1.
$$
Let $S=\sum_L x_L$.  Then $S-x_H=1$ for every line $H$, so $x_H=S-1$ for all $H$.  Since there are $p+1$ lines, summing over $H$ gives
$$
S=(p+1)(S-1).
$$
Therefore $pS=p+1$, impossible for an integer $S$.

Finally, consider generalized quaternion groups.  Write
$$
Q_{2^n}=\langle a,b\mid a^{2^{n-1}}=1,\ b^2=a^{2^{n-2}},\ bab^{-1}=a^{-1}\rangle.
$$
Every $Q_{2^n}$ with $n\geq 3$ contains a copy of $Q_8$.  If
$$
x=a^{2^{n-3}},
$$
then $x$ has order $4$, $b^2=x^2$, and $bxb^{-1}=x^{-1}$.  It follows that $\langle x,b\rangle\cong Q_8$.  Therefore, by \cref{subgpcrit}, it suffices to exclude $Q_8$. The group $Q_8$ has four one dimensional real representations and one irreducible real representation $\rho$ of dimension four.  Let $d$ be the coefficient of $\rho$ in $V$. For any cyclic subgroup $C_4\mathrel{\leq} Q_8$, the restriction of $\rho$ to $C_4$ is $2\lambda(1,4)$.  The one dimensional representations restrict only to trivial or sign representations.  Therefore \cref{cyclicpgroup} gives
$$
2d=1,
$$
which is impossible.
\end{proof}

\section{The local criterion for Anderson-dualizing groups}
\label{sec:Section4}

In this section, we prove part \textup{(1)} of \cref{localcriterion}. We first establish the restrictions on conjugacy and Sylow subgroups, then use Bauer's realization theorem to prove the ``if'' direction.

\begin{thm} \label[thm]{psubgroupconjugacy}
Assume that $G$ satisfies \cref{mainquestion} and $C_{p^n}$ is a subgroup of $G$. Let $g\in C_{p^n}$ be the generator of the subgroup. Then elements in $C_{p^n}$ other than $g,g^{-1}$ are not in the conjugacy class of $g$.

In other words, every automorphism of $C_{p^n}$ induced by an element of its normalizer is either the identity or inversion.
\end{thm}

\begin{proof}
Let $H=C_{p^n}=\langle g\rangle$.  Suppose that $g^a$ is conjugate to $g$ in $G$, where $a$ is prime to $p$.  Let $V$ be a virtual representation of $G$ which satisfies \cref{mainquestion}. Choose $h\in G$ such that $hgh^{-1}=g^a$.  For every integer $k$,
$$
hg^kh^{-1}=g^{ak}.
$$
Therefore the character of $V|_H$ is invariant under the automorphism of $H$ given by $g\mapsto g^a$ on every element of $H$.  It follows that the virtual representation $V|_H$ itself is invariant under this automorphism.

Consider the action of multiplication by $a$ on the set of faithful real rotation representations of $H$.  We identify $\lambda(i,p^n)$ with $\lambda(-i,p^n)$.  If the image of $a$ in $(\mathbb Z/p^n)^\times/\{\pm 1\}$ has order $d$, then every orbit of this action has length $d$.  If $a^r i\equiv \pm i \pmod {p^n}$ for some $i$ prime to $p$, then $a^r\equiv \pm 1 \pmod {p^n}$.

Since $V|_H$ is preserved by this action, the total coefficient $\mu_0(V|_H)$ is a multiple of $d$.  On the other hand, \cref{cyclicpgroup} gives $\mu_0(V|_H)=1$.  Therefore $d=1$, so $a\equiv \pm 1 \pmod {p^n}$.

Therefore $g^a$ can be conjugate to $g$ only when $g^a=g$ or $g^a=g^{-1}$.
\end{proof}

\bigskip

\begin{cor} \label[cor]{NLC}
Assume that $G$ satisfies \cref{mainquestion} and $H$ is a cyclic $p$-subgroup of $G$. Then $|N_G(H):C_G(H)|$ is either 1 or 2.
\end{cor}
\begin{proof}
Consider the homomorphism $\psi\colon N_G(H)\to \operatorname{Aut}(H)$ induced by conjugation.  Its kernel is $C_G(H)$.  By \cref{psubgroupconjugacy}, the image of $\psi$ is contained in $\{\operatorname{id},\operatorname{inv}\}$. Therefore $|N_G(H):C_G(H)|$ is either $1$ or $2$.
\end{proof}

\begin{cor} \label[cor]{cyclicconjugacy}
Assume that $G$ satisfies \cref{mainquestion} and $C_n$ is a subgroup of $G$ with generator $g$. Then for any $h\in G$ such that $hC_nh^{-1}=C_n$, we have $h^2gh^{-2}=g$.
\end{cor}

\begin{proof}
For each prime $p\mid n$, let $P_p$ be the unique Sylow $p$-subgroup of $C_n$. The subgroup $P_p$ is preserved by $h$, and \cref{psubgroupconjugacy} shows that $h^2$ acts trivially on $P_p$. Since $C_n=\prod_{p\mid n}P_p$, it follows that $h^2$ acts trivially on $C_n$.
\end{proof}
\bigskip

Now we can describe the Sylow $p$-subgroups of $G$:

\begin{thm} \label[thm]{sylowpsubgp}
Assume that $G$ satisfies \cref{mainquestion} and is a finite $p$-group.  Then $G$ is cyclic, except that for $p=2$ it may also be dihedral.
\end{thm}

\begin{proof}
By \cref{blacklist}, a group satisfying \cref{mainquestion} cannot contain any of the bad subgroups listed there.

Suppose first that $p$ is odd.  By \cite[Corollary~6.2]{BoucStancuWebb17}, every finite noncyclic $p$-group contains a subgroup isomorphic to $C_p\times C_p$.  This subgroup is excluded by \cref{blacklist}.  Therefore $G$ is cyclic.

Now let $p=2$.  Again by \cite[Corollary~6.2]{BoucStancuWebb17}, every finite $2$-group which is neither cyclic nor dihedral contains a subgroup isomorphic to one of
$$
C_2\times C_4,\qquad Q_8,\qquad C_2\times C_2\times C_2.
$$
Each is excluded by \cref{blacklist}.  Therefore $G$ is cyclic or dihedral.
\end{proof}
\bigskip

\begin{prop} \label[prop]{abelian}
If $G$ is abelian and satisfies \cref{mainquestion}, then $G$ is either cyclic or $C_2\times C_2\times C_m$ with $m$ odd.
\end{prop}

\begin{proof}
Writing $G=\prod_p G_p$ as the product of its Sylow subgroups, \cref{subgpcrit} and \cref{sylowpsubgp} show that every odd $G_p$ is cyclic and that $G_2$ is cyclic or $C_2\times C_2$, which gives the two stated possibilities.
\end{proof}

\bigskip

\begin{rmk}\label[rmk]{rmk:integral-dimension-existence}\rm
The following argument proves that the dimension conditions in part \textup{(1)} of \cref{thm:maintheoremB} imply that $G$ is Anderson-dualizing. More precisely, suppose that $G$ is any finite group and $V\in RO(G)$ satisfies $|V|=0$ and $|V^C|=-2$ for every nontrivial cyclic subgroup $C\mathrel{\leq} G$ of prime-power order. Then $V\in\mathcal S_G$.

For each such $C$, the difference $V|_C-(\lambda(1,|C|)-2)$ has zero underlying dimension and zero fixed-point dimension on every nontrivial subgroup of $C$. By \cref{thm:maintheoremC} and \cref{prop:affine-shifts}, $V|_C\in\mathcal S_C$. The cyclic restriction calculations in the proof of \cref{blacklist} therefore exclude $C_p\times C_p$ for odd $p$ and $C_2\times C_2\times C_2$, $C_4\times C_2$, and $Q_8$ as subgroups of $G$. By \cite[Corollary~6.2]{BoucStancuWebb17}, the odd Sylow subgroups are cyclic and the Sylow $2$-subgroups are cyclic or dihedral.

For each Sylow subgroup $P$, let $V_P$ be the basic integral shift in \cref{pgrouprep}. Its underlying dimension is zero, and its fixed-point dimension on every nontrivial cyclic subgroup is $-2$. Therefore \cref{thm:maintheoremC} gives $V|_P-V_P\in\Lambda_P(H\uZ)$. By \cref{prop:affine-shifts}, $V|_P\in\mathcal S_P$ for every Sylow subgroup $P$. Now \cref{decomposetopsubgps} gives $V\in\mathcal S_G$. For $G=e$, the dimension condition gives $V=0$, which is the unique shift.
\end{rmk}

We use the following consequence of Bauer's realization theorem.

\begin{lem}\label[lem]{lem:prime-power-realization}
Let $G$ be a finite group, and let $\mathcal P(G)$ be the collection of prime-power subgroups of $G$, including $e$. Let $f\colon\mathcal P(G)\to\mathbb Z$ be invariant under conjugation. Suppose that the restriction of $f$ to each Sylow subgroup is the fixed-point dimension function of a virtual real representation. Suppose also that
\[
f(L)\equiv f(H)\pmod{q^{r-l}}
\]
whenever $p$ and $q$ are primes and
\[
L\triangleleft H\triangleleft M\mathrel{\leq} N_G(L),
\qquad H/L\cong C_p,
\]
where $H$ is a $p$-group and $M/H\cong C_{q^r}$ acts on $H/L$ with kernel of order $q^l$. Then there is $V\in RO(G)$ such that $|V^H|=f(H)$ for every $H\in\mathcal P(G)$.
\end{lem}

\begin{proof}
The local realization hypothesis implies that $f$ satisfies the Borel--Smith conditions on every Sylow subgroup. The stated congruence is Bauer's additional Artin condition. The result follows from \cite[Theorem~1.3]{Bauer89}, in the form recorded in \cite[Theorem~5.1]{ReehYalcin18}.
\end{proof}

\begin{proof}[Proof of part \textup{(1)} of \cref{localcriterion}]
For the ``only if'' direction, the Sylow condition follows from \cref{subgpcrit} and \cref{sylowpsubgp}. The normalizer condition follows from \cref{psubgroupconjugacy}.

Conversely, suppose that $G$ satisfies the Sylow and normalizer conditions. Define $f\colon\mathcal P(G)\to\mathbb Z$ by
\[
f(H)=
\begin{cases}
0,&H=e,\\
-2,&H\neq e\text{ is cyclic},\\
-3,&H\text{ is a noncyclic }2\text{-group}.
\end{cases}
\]
This function is invariant under conjugation. On a cyclic Sylow subgroup $P$, it is the fixed-point dimension function of $\lambda(1,|P|)-2$. On a dihedral Sylow $2$-subgroup, it is the fixed-point dimension function of the basic shift $\rho+\sigma-3$ in \cref{pgrouprep}. A nontrivial cyclic subgroup fixes a line in $\rho\oplus\sigma$, while a noncyclic subgroup fixes no vector.

We check the congruence in \cref{lem:prime-power-realization}. If $p=2$, the action on $H/L\cong C_2$ is trivial, so the modulus is $1$. If $p$ is odd, then $H$ is cyclic. Since $M$ normalizes $H$, the normalizer condition shows that its induced action on $H/L$ has order dividing $2$. The difference $f(L)-f(H)$ is $2$ when $L=e$, and is $0$ otherwise. The congruence follows.

By \cref{lem:prime-power-realization}, there is $V\in RO(G)$ with $|V^H|=f(H)$ for every prime-power subgroup $H$. For each Sylow subgroup $P$, let $V_P$ be its basic integral Anderson-dualizing shift from \cref{pgrouprep}. The representations $V|_P$ and $V_P$ have the same fixed-point dimensions on every subgroup of $P$. By \cref{thm:maintheoremC},
\[
V|_P-V_P\in\Lambda_P(H\uZ).
\]
Therefore \cref{prop:affine-shifts} shows that $V|_P$ is Anderson-dualizing. Sylow detection, \cref{decomposetopsubgps}, proves that $G$ is Anderson-dualizing.
\end{proof}

\section{Inductive reductions for the converse} \label{sec:Section5}

In this section and the next section, we prove that every finite group satisfying the Sylow and normalizer conditions lies in the list of \cref{thm:maintheorem}.

We use $A\ltimes B$ to mean that $B$ is the normal subgroup.

These conditions are inherited by subgroups. If $L\mathrel{\leq} G$, a Sylow $p$-subgroup of $L$ is contained in a Sylow $p$-subgroup of $G$. Subgroups of cyclic groups are cyclic, and subgroups of dihedral $2$-groups are cyclic or dihedral. For every cyclic $p$-subgroup $H\mathrel{\leq} L$, we also have
\[
N_L(H)\mathrel{\leq} N_G(H).
\]

Let $G$ be a finite group satisfying the Sylow and normalizer conditions. The result is clear when $G=1$, so we assume that $|G|>1$. Assume inductively that every group $G'$ of smaller order satisfying these conditions is isomorphic to one of the groups listed in \cref{thm:maintheorem}.
 
We will show that $G$ itself is also in this list.

We summarize the strategy of the remaining proof.

\begin{enumerate}[label=\textup{\arabic*)}]
\item If Burnside's normal $p$-complement theorem applies, then $ G=P\ltimes N $ where $N\triangleleft G$ is a normal $p$-complement.  By induction, $N$ belongs to the list in \cref{thm:maintheorem}.  Analyzing the possible actions of $P$ on $N$ shows that $G$ is still in the list.

\item Therefore we may assume that Burnside's theorem never applies.
The Sylow and normalizer conditions then show that the odd Sylow subgroups are cyclic with inversion in the normalizer and that the Sylow $2$-subgroups are dihedral.
This is \cref{afterinduction}.

\item We next handle the case $G=N_G(H)$ for some cyclic $p$-subgroup $H$ of order at least $3$.  Using induction on $C_G(H)$ and the restrictions on conjugation actions, \cref{inductnormal} shows that $G$ is again one of the groups in \cref{thm:maintheorem}.

\item In the final case, every such normalizer is proper.  We first show that the odd core of $G$ is trivial.  Walter's classification for Klein four Sylow $2$-subgroups and the Gorenstein--Walter classification for larger dihedral Sylow $2$-subgroups then give a short list of possibilities.  The restrictions on cyclic $p$-subgroups exclude the remaining projective-linear groups.
\end{enumerate}
Therefore only the groups in \cref{thm:maintheorem} survive.

\bigskip

We will use the following standard form of Burnside's normal \(p\)-complement theorem, as stated in \cite[Theorem~5.13]{IsaacsFGT}.

\begin{thm} \label[thm]{burnsidecomp}
Let $P$ be a Sylow $p$-subgroup of $G$ which is central in its normalizer $N_GP$. Then
\[
G=P\ltimes N,
\]
where $N\triangleleft G$ is a normal $p$-complement. In particular, $p\nmid |N|$.
\end{thm}

\subsection{The normal-complement case}

If \cref{burnsidecomp} applies, $P$ must be abelian,
and the Sylow condition shows that $P$ is either cyclic or a dihedral $2$-group. The latter is abelian only when $P\cong C_2\times C_2$.

The subgroup $N$ satisfies the Sylow and normalizer conditions, so the induction hypothesis shows that it is a group described in \cref{thm:maintheorem}, with $|N|$ coprime to $|P|$.
\bigskip

\subsubsection{For $p=2$}
If $P=C_2\times C_2$ or $P=C_{2^n}$, then $|N|$ is odd.  Since all odd-order groups in the list of \cref{thm:maintheorem} are cyclic, we have $N=C_m$ for some odd $m$.  For each prime $q\mid m$, let $Q_q\mathrel{\leq} C_m$ be the Sylow $q$-subgroup. Since $C_m$ is cyclic, $Q_q$ is its unique Sylow $q$-subgroup, so the conjugation action of $P$ preserves $Q_q$.
By the normalizer condition, every element of $P$ acts on $Q_q$ either by the identity or by inversion.
Since
\[
\operatorname{Aut}(C_m)\cong
\prod_{q\mid m}\operatorname{Aut}(Q_q),
\]
the image of $P$ in $\operatorname{Aut}(C_m)$ has exponent at most $2$.  This is exactly the action condition in \cref{thm:maintheorem}.  Therefore $G=P\ltimes C_m$ is in the list of \cref{thm:maintheorem}.

\subsubsection{For odd $p$}
For the odd-prime case, we will use the following two facts about automorphisms.

\begin{lem} \label[lem]{crossaut}
Let $N=K\ltimes A$, where $A$ is finite, abelian, and normal in $N$.  Let $U$ be the subgroup of $\operatorname{Aut}(N)$ consisting of automorphisms which fix $A$ pointwise and induce the identity on $N/A\cong K$.  Then there is an injective homomorphism
\[
U\longrightarrow \prod_{k\in K}A.
\]
In particular, every prime divisor of $|U|$ divides $|A|$.
\end{lem}

\begin{proof}
Let $\alpha\in U$.  For every $k\in K$, the image of $\alpha(k)$ in $N/A$ is equal to the image of $k$.  There is therefore a unique element $d_\alpha(k)\in A$ such that
\[
\alpha(k)=k d_\alpha(k).
\]
Define
\[
D(\alpha)=\bigl(d_\alpha(k)\bigr)_{k\in K}\in \prod_{k\in K}A.
\]
If $D(\alpha)=D(\beta)$, then $\alpha$ and $\beta$ agree on $K$.  They also agree on $A$ because both restrict to the identity.  Since $N$ is generated by $K$ and $A$, it follows that $\alpha=\beta$.  Therefore $D$ is injective.

For $\alpha,\beta\in U$ and $k\in K$, we have
\[
(\alpha\beta)(k)
=\alpha\bigl(kd_\beta(k)\bigr)
=kd_\alpha(k)d_\beta(k),
\]
because $\alpha$ fixes $A$ pointwise.  Therefore
\[
d_{\alpha\beta}(k)=d_\alpha(k)d_\beta(k)
\]
for every $k\in K$, so $D$ is a homomorphism into the direct product with pointwise multiplication.  Since
\[
\left|\prod_{k\in K}A\right|=|A|^{|K|},
\]
every prime divisor of $|U|$ divides $|A|$.
\end{proof}

\begin{lem} \label[lem]{coprimeaut}
Let $K$ be one of
\[
C_{2^r},\qquad D_{2^{r+1}},\qquad T_{3^r},\qquad O_{3^r},\qquad A_5
\]
from \cref{thm:maintheorem}, where $r\geq 0$ for $C_{2^r}$ and $r\geq 1$ for $D_{2^{r+1}}$, $T_{3^r}$, and $O_{3^r}$.  Let $p$ be an odd prime such that $p\nmid |K|$.  Then every $p$-subgroup of $\operatorname{Aut}(K)$ is trivial, except when
\[
K\cong C_2\times C_2
\qquad\text{and}\qquad
p=3.
\]
In this exceptional case, $\operatorname{Aut}(K)\cong S_3$, and its nontrivial $3$-subgroup cyclically permutes the three nonidentity elements of $K$.
\end{lem}

\begin{proof}
First suppose that $K=C_{2^r}$.  The group $\operatorname{Aut}(K)$ has order $\varphi(2^r)$, which is a power of $2$.  It has no nontrivial subgroup of odd order.

Next suppose that $K=D_{2^{r+1}}$.  If $r=1$, then $K\cong C_2\times C_2$ and $\operatorname{Aut}(K)\cong S_3$.  This gives the exceptional case.  Now assume that $r\geq 2$ and write
\[
K=\langle a,s\mid a^{2^r}=s^2=1,\ sas=a^{-1}\rangle.
\]
The rotation subgroup $\langle a\rangle$ is the unique cyclic subgroup of order $2^r$, so every automorphism of $K$ preserves it. Every automorphism of $K$ has the form
\[
a\longmapsto a^u,
\qquad
s\longmapsto a^v s,
\]
where $u$ is odd modulo $2^r$ and $v$ is arbitrary modulo $2^r$.  Therefore
\[
|\operatorname{Aut}(K)|=2^r\varphi(2^r)=2^{2r-1},
\]
so $\operatorname{Aut}(K)$ has no nontrivial subgroup of odd order.

Now let $K=T_{3^r}=C_{3^r}\ltimes E$, where $E=C_2\times C_2$.  The quotient $K/E$ is abelian, and the commutators of $C_{3^r}$ with $E$ generate $E$.  Therefore
\[
[K,K]=E.
\]
Every automorphism of $K$ preserves the derived subgroup $[K,K]$, and therefore preserves $E$.  Restriction to $E$ and passage to the quotient give a homomorphism
\[
\operatorname{Aut}(K)
\longrightarrow
\operatorname{Aut}(E)\times\operatorname{Aut}(C_{3^r}).
\]
The kernel fixes $E$ pointwise and induces the identity on $K/E$.  By \cref{crossaut}, this kernel is a $2$-group.  Since
\[
\operatorname{Aut}(E)\cong S_3
\qquad\text{and}\qquad
|\operatorname{Aut}(C_{3^r})|=2\cdot 3^{r-1},
\]
the only prime divisors of $|\operatorname{Aut}(K)|$ are $2$ and $3$.  Since the odd prime $p$ does not divide $|K|=4\cdot 3^r$, we have $p\ne 3$, so $p$ does not divide $|\operatorname{Aut}(K)|$.

Next let $K=O_{3^r}=D_{2\cdot 3^r}\ltimes E$.  The commutators with $E$ generate $E$, and the derived subgroup of $D_{2\cdot 3^r}$ is $C_{3^r}$.  It follows that
\[
[K,K]=T_{3^r}
\qquad\text{and}\qquad
[T_{3^r},T_{3^r}]=E.
\]
Both $T_{3^r}$ and $E$ are therefore preserved by automorphisms of $K$.  Restriction to $E$ and passage to the quotient give
\[
\operatorname{Aut}(K)
\longrightarrow
\operatorname{Aut}(E)\times\operatorname{Aut}(D_{2\cdot 3^r}).
\]
The kernel is a $2$-group by \cref{crossaut}.  Every automorphism of
\[
D_{2\cdot 3^r}=\langle a,s\mid a^{3^r}=s^2=1,\ sas=a^{-1}\rangle
\]
has the form
\[
a\longmapsto a^u,
\qquad
s\longmapsto a^v s,
\]
where $u$ is a unit modulo $3^r$ and $v$ is arbitrary modulo $3^r$.  Consequently,
\[
|\operatorname{Aut}(D_{2\cdot 3^r})|
=3^r\varphi(3^r)
=2\cdot 3^{2r-1}.
\]
Only the primes $2$ and $3$ divide $|\operatorname{Aut}(K)|$, and neither can be the odd prime $p$ with $p\nmid |K|$.

Finally, $\operatorname{Aut}(A_5)\cong S_5$.  To recall the argument, conjugation gives an injection
\[
S_5\longrightarrow\operatorname{Aut}(A_5)
\]
because $C_{S_5}(A_5)=1$.  Conversely, $\operatorname{Aut}(A_5)$ acts on the five Sylow $2$-subgroups of $A_5$.  The subgroup of inner automorphisms $\operatorname{Inn}(A_5)\cong A_5$ acts faithfully on this set.  If $L$ is the kernel of the action of $\operatorname{Aut}(A_5)$, then
\[
L\cap\operatorname{Inn}(A_5)=1.
\]
Both subgroups are normal in $\operatorname{Aut}(A_5)$, so they centralize each other.  Since $A_5$ is centerless, the centralizer of $\operatorname{Inn}(A_5)$ in $\operatorname{Aut}(A_5)$ is trivial.  Therefore $L=1$, and $\operatorname{Aut}(A_5)$ embeds in $S_5$.  We obtain $\operatorname{Aut}(A_5)\cong S_5$.  Its prime divisors are $2$, $3$, and $5$, all of which divide $|A_5|$.
\end{proof}

Assume that Burnside's normal $p$-complement theorem applies for an odd prime $p$.
By the Sylow condition, $P$ is cyclic.
Write $P=C_{p^n}$.  Then
\[
G=P\ltimes N
\]
with $N\triangleleft G$ a normal $p$-complement.
The subgroup $N$ satisfies the Sylow and normalizer conditions. Since $|N|<|G|$, the induction hypothesis gives
\[
N=K\ltimes C_m
\]
as one of the groups in \cref{thm:maintheorem}, with $m$ odd and $(m,|K|)=1$.  Since $N$ is a $p$-complement,
\[
p\nmid |K|m.
\]

We first prove that $P$ centralizes $C_m$.  Let $q\mid m$, and let $Q\mathrel{\leq} C_m$ be the Sylow $q$-subgroup.  Then $C_m\triangleleft N$ implies $Q\triangleleft N$.  Since $q\nmid |K|$, the subgroup $Q$ is a Sylow $q$-subgroup of $N$.  Therefore $Q$ is the unique Sylow $q$-subgroup of $N$.  Since $N\triangleleft G$, the conjugation action of $P$ preserves $Q$.
By the normalizer condition, each element of $P$ acts on the cyclic group $Q$ either trivially or by inversion.
The image of the odd $p$-group $P$ in $\operatorname{Aut}(Q)$ is therefore contained in a group of order at most $2$, so it is trivial.  This holds for every $q\mid m$.  It follows that $P$ centralizes $C_m$.

The action of $P$ on $N$ induces
\[
\theta\colon P\longrightarrow
\operatorname{Aut}(N/C_m)\cong\operatorname{Aut}(K).
\]
By \cref{coprimeaut}, $\theta$ is trivial except possibly when $K=E=C_2\times C_2$ and $p=3$.

Suppose first that $\theta$ is trivial.  Since $P$ also fixes $C_m$ pointwise, its image in $\operatorname{Aut}(N)$ lies in the cross-term subgroup from \cref{crossaut}.  Every prime divisor of the order of this subgroup divides $m$, while $p\nmid m$.  Therefore the image of the $p$-group $P$ is trivial, and $P$ centralizes all of $N$.  It follows that
\[
G\cong P\times N\cong K\ltimes(C_m\times P).
\]
The product $C_m\times P$ is cyclic because $(m,p)=1$.  Moreover, $(mp^n,|K|)=1$ because $(m,|K|)=1$ and $p\nmid |K|$.  The action of $K$ on this cyclic group is the old action on $C_m$ together with the trivial action on $P$, so every element of $K$ still acts by an automorphism of order at most $2$.  Therefore $G$ satisfies all the conditions in \cref{thm:maintheorem}.

It remains to consider the case $p=3$ and $K=E=C_2\times C_2$, with $\theta$ nontrivial.  Its image is the order-$3$ subgroup of $\operatorname{Aut}(E)\cong S_3$ and cyclically permutes the three nonidentity elements of $E$.  Let
\[
P_0=\ker(\theta).
\]
Then $[P:P_0]=3$.  The subgroup $P_0$ fixes $C_m$ pointwise and induces the identity on $N/C_m\cong E$.  By \cref{crossaut} and the fact that $3\nmid m$, the subgroup $P_0$ acts trivially on all of $N$.

Since $C_m$ is abelian, the conjugation action of $N$ on $C_m$ factors through $N/C_m$.  Identify $N/C_m$ with $E$, and write this action as
\[
\beta\colon E\longrightarrow\operatorname{Aut}(C_m).
\]
For $x\in P$ and $e\in E$, the fact that $x$ fixes $C_m$ pointwise gives
\[
\beta\bigl(\theta(x)(e)\bigr)=\beta(e).
\]
The three nonidentity elements $e_1,e_2,e_3$ of $E$ therefore have a common image $u\in\operatorname{Aut}(C_m)$.  Since each $e_i$ has order $2$, we have $u^2=1$.  Since $e_1e_2e_3=1$, we also have $u^3=1$.  Therefore $u=1$, the action of $E$ on $C_m$ is trivial, and
\[
N=E\times C_m.
\]

Now $E$ is the unique Sylow $2$-subgroup of $N$.  Consequently $P$ preserves $E$.  The action factors through $P/P_0\cong C_3$ and is the standard $3$-cycle on the three nonidentity elements of $E$.  It follows that
\[
P\ltimes E\cong T_{3^n}.
\]
Both $P$ and $E$ centralize $C_m$, so
\[
G\cong T_{3^n}\times C_m,
\]
which is the $T_{3^n}$-family in \cref{thm:maintheorem}.
\bigskip

\subsection{Sylow structure and the normalizer-equality case}

From now on, we assume that Burnside's normal $p$-complement theorem does not apply to any Sylow subgroup of $G$.

\begin{prop} \label[prop]{afterinduction}
Suppose that $G$ satisfies the Sylow and normalizer conditions and that \cref{burnsidecomp} does not apply to any Sylow subgroup of $G$. Then:

\begin{enumerate}[label=\textbf{(\arabic*)}]
\item for any odd prime $p\mid |G|$, every Sylow $p$-subgroup $P$ of $G$ is cyclic, with $P\mathrel{\leq} C_GP\mathrel{\leq} N_GP$, and $|N_GP : C_GP|=2$.
\item every Sylow $2$-subgroup $P$ of $G$ is dihedral.
\end{enumerate}
\end{prop}

\begin{proof} \hfill
\begin{enumerate}[label=\textup{(\arabic*)}]
\item
By the Sylow condition, the subgroup $P$ is cyclic. By the normalizer condition, $|N_G(P):C_G(P)|\leq 2$. Since Burnside's theorem does not apply, this index must be $2$.

\item
By the Sylow condition, any Sylow $2$-subgroup $P$ of $G$ is either cyclic or dihedral.
Suppose first that $P \cong C_{2^n}$. Then $ \mathrm{Aut}(P)$ is a $2$-group.

Since
$$
P\mathrel{\leq} C_GP\mathrel{\leq} N_GP
$$
and $P$ is a Sylow $2$-subgroup of $G$, the index $|N_GP:P|$ is odd. Therefore the divisor $|N_GP:C_GP|$ is also odd.

On the other hand, there is a natural injective homomorphism
\[
N_G(P)/C_G(P)\hookrightarrow \mathrm{Aut}(P).
\]
As $\mathrm{Aut}(P)$ is a $2$-group, this forces $N_G(P)/C_G(P)$ to be trivial. This contradicts the assumption. Therefore $P$ is dihedral.

\end{enumerate}
\end{proof}

We now consider the case when $G=N_G(H)$ for some cyclic $p$-subgroup $H\mathrel{\leq} G$ with $|H|\geq 3$.

\begin{lem}\label[lem]{D16normalVfour}
A dihedral group of order $16$ has no normal subgroup isomorphic to $C_2\times C_2$.
\end{lem}

\begin{proof}
Write
\[
D_{16}=\langle r,s\mid r^8=s^2=1,\ srs=r^{-1}\rangle.
\]
Every Klein four subgroup contains the unique central involution $r^4$ and a reflection, so it has the form
\[
V_i=\langle r^4,r^is\rangle.
\]
Conjugation by $r$ sends $r^is$ to $r^{i+2}s$. The two reflections in $V_i$ are $r^is$ and $r^{i+4}s$, and $i+2$ is congruent to neither $i$ nor $i+4$ modulo $8$. It follows that
\[
rV_ir^{-1}\neq V_i.
\]
\end{proof}

\begin{lem}\label[lem]{Tindexextension}
Let $T_{3^n}=C_{3^n}\ltimes E$, where $E=C_2\times C_2$, with the standard $3$-cycle action. Suppose that $L$ contains $T_{3^n}$ as a normal subgroup of index $2$, that $E\triangleleft L$, and that a Sylow $2$-subgroup of $L$ is $D_8$. Then
\[
L\cong O_{3^n}.
\]
\end{lem}

\begin{proof}
Let $A$ be a Sylow $3$-subgroup of $T_{3^n}$. An element of $E$ normalizing $A$ centralizes $A$, and $C_E(A)=1$. Therefore
\[
N_{T_{3^n}}(A)=A.
\]
The Frattini argument gives
\[
L=T_{3^n}N_L(A),
\qquad |N_L(A)|=2\cdot3^n.
\]
Choose an involution $t\in N_L(A)$. Since $N_L(A)\cap T_{3^n}=A$, the element $t$ lies outside $T_{3^n}$. The subgroup $\langle E,t\rangle$ is a Sylow $2$-subgroup of $L$, so it is isomorphic to $D_8$. Consequently, $t$ acts on $E$ as a transposition.

Let $a$ generate $A$, and write $tat^{-1}=a^u$. Since $t^2=1$,
\[
u^2\equiv1\pmod{3^n}.
\]
The actions of $a$ and $t$ on $E$ are a $3$-cycle and a transposition, respectively, so $u\equiv-1\pmod3$. These congruences force $u\equiv-1\pmod{3^n}$.

Therefore $\langle A,t\rangle\cong D_{2\cdot3^n}$, acting on $E$ through the standard quotient $D_6$. Since
\[
L=\langle A,t\rangle\ltimes E,
\]
we obtain $L\cong O_{3^n}$.
\end{proof}
\par

\begin{prop} \label[prop]{inductnormal}
Suppose that $G$ satisfies the Sylow and normalizer conditions, that we are in the situation of \cref{afterinduction}, and that
\[
G=N_GH
\]
for some non-trivial cyclic $p$-subgroup $H\mathrel{\leq} G$ with $|H|\geq 3$. Then $G$ is listed in \cref{thm:maintheorem}.
\end{prop}

\begin{proof}
The cyclic group $H$ contains a subgroup $H'$ whose order is either an odd prime or $4$. Since
\[
G=N_GH\mathrel{\leq} N_GH'\mathrel{\leq} G,
\]
we may replace $H$ by $H'$ and assume that $|H|$ is either an odd prime or $4$.

Put $C=C_GH$.
By the normalizer condition,
\[
[G:C]\in\{1,2\}.
\]
Suppose first that $[G:C]=2$. Then $C\mathrel{<}G$.
The subgroup $C$ satisfies the Sylow and normalizer conditions, so the induction hypothesis gives
\[
C=K\ltimes C_m
\]
as in \cref{thm:maintheorem}. For each prime $q\mid m$, its Sylow $q$-subgroup is the unique Sylow $q$-subgroup of $C$, since $q\nmid |K|$ and $C_m\triangleleft C$. Since $C\triangleleft G$, it follows that $C_m\triangleleft G$.
By the normalizer condition, every element of $G$ acts on each Sylow subgroup of $C_m$ either by the identity or by inversion.

If $K=1$, then a Sylow $2$-subgroup of $G$ has order $2$ and is cyclic. This contradicts \cref{afterinduction}(2), so this case does not occur.

Suppose that $K$ is a non-trivial cyclic or dihedral $2$-group, and let $S$ be a Sylow $2$-subgroup of $G$. Since $[G:C]=2$ and $C_m$ is odd,
\[
|S|=2|K|.
\]
By \cref{afterinduction}(2), the group $S$ is dihedral. The quotient map
\[
G\longrightarrow G/C_m
\]
restricts injectively to $S$, and both $S$ and $G/C_m$ have order $2|K|$. Therefore
\[
S\cong G/C_m.
\]
It follows that
\[
G=S\ltimes C_m,
\]
with the allowed action on $C_m$. This is a dihedral family in \cref{thm:maintheorem}.

Now suppose that $K=T_{3^n}$. The action on $C_m$ is trivial in this family, so
\[
C=T_{3^n}\times C_m.
\]
Since $(m,6)=1$, the subgroup $T_{3^n}$ is the unique Hall $\{2,3\}$-subgroup of $C$. Since $C\lhd G$, conjugation by $G$ preserves this unique Hall subgroup, so $T_{3^n}\lhd G$. Its normal Klein four subgroup $E$ is $O_2(C)$, the largest normal $2$-subgroup of $C$. Therefore $E$ is normal in $G$. A Sylow $2$-subgroup $S$ of $G$ has order $8$ and is $D_8$ by \cref{afterinduction}(2). Choose an involution $t\in S-E$. The subgroup
\[
L=\langle T_{3^n},t\rangle
\]
contains $T_{3^n}$ normally with index $2$, has Sylow $2$-subgroup $D_8$, and has $E\triangleleft L$. By \cref{Tindexextension},
\[
L\cong O_{3^n}.
\]
Since $|L|m=|G|$ and $L\cap C_m=1$,
\[
G=L\ltimes C_m\cong O_{3^n}\ltimes C_m.
\]
The subgroup $T_{3^n}$ acts trivially on $C_m$, while the outer involution acts on each Sylow subgroup of $C_m$ by the identity or by inversion. This is the $O_{3^n}$ family in \cref{thm:maintheorem}.

Next suppose that $K=O_{3^n}$. Let $E=C_2\times C_2$ be its distinguished normal Klein four subgroup. We claim that $E=O_2(C)$. In $O_{3^n}$, the quotient by $E$ is $D_{2\cdot 3^n}$, which has no non-trivial normal $2$-subgroup. Therefore $O_2(O_{3^n})=E$. Under the quotient map $C\to C/C_m\cong O_{3^n}$, every normal $2$-subgroup of $C$ has image contained in $E$ and therefore lies in the preimage of $E$. Since $E\mathrel{\leq} T_{3^n}$ acts trivially on $C_m$, this preimage is $E\times C_m$. Since $C_m$ has odd order, every normal $2$-subgroup of $C$ lies in $E$. It follows that $E=O_2(C)$, so $E$ is normal in $G$. A Sylow $2$-subgroup of $C$ is $D_8$. Therefore a Sylow $2$-subgroup of $G$ has order $16$ and is dihedral by \cref{afterinduction}(2). It contains $E$ normally, contradicting \cref{D16normalVfour}. This case is impossible.

Finally suppose that $K=A_5$. Then
\[
C=A_5\times C_m.
\]
Since $A_5$ is perfect and $C_m$ is abelian,
\[
[C,C]=A_5.
\]
Therefore $A_5$ is normal in $G$. Let
\[
\theta:G\longrightarrow\operatorname{Aut}(A_5)\cong S_5
\]
be the conjugation homomorphism. Its image on $C$ is the subgroup of inner automorphisms
\[
\operatorname{Inn}(A_5)\cong A_5.
\]
Since $[G:C]=2$, either $\theta(G)=A_5$ or $\theta(G)=S_5$.

Suppose first that $\theta(G)=A_5$. Every conjugation automorphism of $A_5$ induced by $G$ is inner. Since $A_5$ is centerless,
\[
G=A_5\times C_G(A_5).
\]
The centralizer has order $2m$, so it contains an involution $u$. If $E\cong C_2\times C_2$ is a Sylow $2$-subgroup of $A_5$, then
\[
\langle E,u\rangle\cong C_2\times C_2\times C_2,
\]
contradicting the Sylow condition.
\par

Suppose instead that $\theta(G)=S_5$. Let
\[
a=(0\ 1\ 2\ 3\ 4),
\qquad
b=(1\ 2\ 4\ 3).
\]
Under the affine labeling of the letters by $\mathbb F_5$, the element $a$ is translation by $1$ and $b$ is multiplication by $2$, so
\[
bab^{-1}=a^2.
\]
Choose $y\in G$ inducing conjugation by $b$. Then $y$ normalizes $\langle a\rangle\cong C_5$ and sends its generator to its square. This action is neither the identity nor inversion,
contradicting the normalizer condition.
Therefore the $A_5$ case is impossible.

This completes the case $[G:C]=2$.

Now suppose that $G=C_GH$. If $|H|=p$ is odd, let $P$ be a Sylow $p$-subgroup containing $H$. By \cref{afterinduction}(1), there is an element
\[
x\in N_G(P)-C_G(P)
\]
whose action on the cyclic group $P$ is inversion. Since $H$ is the unique subgroup of order $p$ in $P$, the element $x$ normalizes and inverts $H$. Therefore
\[
x\in N_G(H)-C_G(H),
\]
contradicting $G=N_G(H)=C_G(H)$.

If $|H|=4$, then $H$ is central in $G$. A Sylow $2$-subgroup containing $H$ is dihedral by \cref{afterinduction}(2), but a dihedral $2$-group has no central cyclic subgroup of order $4$. The Klein four group has no element of order $4$, and $D_8$ or any larger dihedral $2$-group has center of order $2$. This is a contradiction.

Therefore $G$ is listed in \cref{thm:maintheorem}.
\end{proof}

\bigskip

\section{Completion of the converse by finite-group classification}
\label{sec:classification}

We continue under the Sylow and normalizer conditions and the assumptions of \cref{afterinduction}. \cref{inductnormal} has treated the case
\[
G=N_G(H)
\]
for a nontrivial cyclic $p$-subgroup $H$ of order at least $3$. Therefore we may assume that
\begin{equation}\label{eq:propernormalizers}
N_G(H)\mathrel{<} G
\end{equation}
for every such $H$.

Let $O_{2'}(G)$ denote the largest normal subgroup of $G$ of odd order. We first claim that
\begin{equation}\label{eq:oddcoretrivial}
O_{2'}(G)=\{e\}.
\end{equation}
\cref{afterinduction} implies that $G$ has even order, so $O_{2'}(G)\mathrel{<} G$. The subgroup $O_{2'}(G)$ satisfies the Sylow and normalizer conditions, so the induction hypothesis applies. 

The only odd-order groups in \cref{thm:maintheorem} are cyclic. Therefore $O_{2'}(G)\cong C_m$ for some odd $m$. If $O_{2'}(G)\neq\{e\}$, choose a prime $p\mid m$ and let $C_p\mathrel{\leq} O_{2'}(G)$ be the unique subgroup of order $p$. Since $O_{2'}(G)\triangleleft G$ is cyclic, conjugation by any element of $G$ preserves its unique subgroup $C_p$ of order $p$. It follows that $C_p\triangleleft G$. Therefore $N_G(C_p)=G$, contradicting \eqref{eq:propernormalizers}. This proves \eqref{eq:oddcoretrivial}.

Let $P$ be a Sylow $2$-subgroup of $G$. By \cref{afterinduction}(2), the group $P$ is dihedral. We separate the Klein-four case from the case $|P|\geq 8$.

\medskip
\noindent \textbf{The Klein-four case.} Suppose that
\[
P\cong C_2\times C_2.
\]
We use the following standard form of Walter's theorem. Walter's original statement \cite[Theorem~1]{Wal69} contains an exceptional family whose identification with the Ree groups was completed later. For this identification, see \cite{Enguehard86}.

\begin{thm}[Walter]\label{thm:Walter-abelian-Sylow}
Let $X$ be a finite group with abelian Sylow $2$-subgroups. Then there are normal subgroups $M,N\triangleleft X$ with $M\mathrel{\leq} N$ such that $|M|$ and $[X:N]$ are odd, and
\[
N/M\cong A\times S_1\times\cdots\times S_k,
\]
where $A$ is an abelian $2$-group and each $S_i$ is isomorphic to one of
\[
\operatorname{PSL}_2(2^d)\quad(d\geq3),
\qquad
\operatorname{PSL}_2(q)\quad(q\geq5\text{ odd},\ q\equiv\pm3\pmod 8),
\]
\[
J_1,
\qquad
{}^2G_2(3^{2e+1})\quad(e\geq1).
\]
\end{thm}

Apply the theorem to $G$. By \eqref{eq:oddcoretrivial}, the normal odd-order subgroup $M$ is $\{e\}$. Since $[G:N]$ is odd, the image of $P$ in $G/N$ is trivial, so $P\mathrel{\leq} N$ and $P$ is also a Sylow $2$-subgroup of $N$. Therefore $N$ has Sylow $2$-subgroup $C_2\times C_2$.

The Sylow $2$-subgroups of $\operatorname{PSL}_2(2^d)$ for $d\geq3$ have order $2^d$, while those of $J_1$ and of the Ree groups in the list have order $8$. The groups $\operatorname{PSL}_2(q)$ with $q\equiv\pm3\pmod 8$ have Sylow $2$-subgroup $C_2\times C_2$. Consequently, if no simple factor occurs, then $N=A\cong C_2\times C_2$. If a simple factor occurs, then $A=\{e\}$, there is exactly one simple factor, and it must be of the second type. Therefore
\[
N\cong C_2\times C_2
\qquad\text{or}\qquad
N\cong\operatorname{PSL}_2(q),
\quad q\geq5,\quad q\equiv\pm3\pmod 8.
\]
The case $q=5$ is $\operatorname{PSL}_2(5)\cong A_5$.

If $N=C_2\times C_2$, then $C_G(N)=N$. To see this, $N$ is a central Sylow $2$-subgroup of $C_G(N)$, so \cref{burnsidecomp} gives a normal odd-order complement $K\triangleleft C_G(N)$. Since $N\triangleleft G$, we also have $C_G(N)\triangleleft G$. For every $g\in G$, the subgroup $gKg^{-1}$ is again a normal odd-order complement in $C_G(N)$. Such a normal complement is unique, so $gKg^{-1}=K$. Therefore $K\triangleleft G$, and \eqref{eq:oddcoretrivial} gives $K=\{e\}$. It follows that $C_G(N)=N$. Conjugation therefore embeds the odd-order group $G/N$ in $\operatorname{Aut}(N)\cong S_3$, giving
\[
G\cong C_2\times C_2
\qquad\text{or}\qquad
G\cong A_4.
\]

If $N\cong\operatorname{PSL}_2(q)$, then
\[
C_G(N)\cap N=C_N(N)=Z(N)=\{e\},
\]
because $\operatorname{PSL}_2(q)$ is centerless for $q\geq5$. The kernel of the restriction of the quotient map $G\to G/N$ to $C_G(N)$ is precisely $C_G(N)\cap N=\{e\}$. Therefore the quotient map restricts to an injection
\[
C_G(N)\hookrightarrow G/N.
\]
Since $[G:N]=|G/N|$ is odd, it follows that $C_G(N)$ has odd order. Also $C_G(N)\triangleleft G$ because $N\triangleleft G$. Therefore
\[
C_G(N)\mathrel{\leq} O_{2'}(G)=\{e\}.
\]
It follows that $C_G(N)=\{e\}$. Conjugation therefore gives an injection
\[
G\hookrightarrow\operatorname{Aut}(N).
\]
Since $N$ is centerless, the restriction of this map to $N$ identifies $N$ with the inner automorphism group $\operatorname{Inn}(N)$. Therefore we may regard
\[
N\triangleleft G\mathrel{\leq}\operatorname{Aut}(N),
\]
and passing to quotients gives
\[
G/N\hookrightarrow\operatorname{Out}(N).
\]
Write $q=r^d$, where $r$ is prime. For $N\cong\operatorname{PSL}_2(q)$ with $q$ odd, the automorphism group has the standard form
\[
\operatorname{Aut}(N)=\mathrm{P}\Gamma\mathrm{L}_2(q)
=\operatorname{PGL}_2(q)\rtimes\operatorname{Gal}(\mathbb F_q/\mathbb F_r)
\]
\cite{Steinberg60}. Moreover,
\[
\operatorname{PGL}_2(q)/\operatorname{PSL}_2(q)\cong C_2,
\qquad
\operatorname{Gal}(\mathbb F_q/\mathbb F_r)\cong C_d.
\]
For $q$ odd these two outer actions commute, so
\[
\operatorname{Out}(N)\cong C_2\times C_d.
\]
The first factor is the diagonal outer automorphism, while the second is the cyclic group of field automorphisms. Since $G/N$ has odd order, its image in $C_2\times C_d$ has trivial projection to the first factor. Therefore $G/N$ lies in the field automorphism subgroup $C_d$. In particular,
\[
G/N\cong C_f
\]
for some odd divisor $f$ of $d$.

The groups $C_2\times C_2$, $A_4$, and $A_5$ already occur in the list of \cref{thm:maintheorem}. Therefore it remains only to consider the possibility
\begin{equation}\label{waltercase}
\operatorname{PSL}_2(q)\triangleleft G
\mathrel{\leq} \mathrm{P}\Gamma\mathrm{L}_2(q),
\qquad
G/\operatorname{PSL}_2(q)\cong C_f,
\end{equation}
where $q=r^d>5$ is an odd prime power satisfying
\[
q\equiv\pm3\pmod 8,
\]
and $f$ is an odd divisor of $d$. The quotient $C_f$ acts on $\operatorname{PSL}_2(q)$ by field automorphisms. We show that this possibility cannot occur.

Let $U$ be the image in $\operatorname{PSL}_2(q)$ of the upper-unitriangular subgroup
\[
\left\{
\begin{pmatrix}
1&t\\
0&1
\end{pmatrix}
\,\middle|\,
t\in\mathbb F_q
\right\}.
\]
Since $r$ is odd, this matrix subgroup meets $\{\pm I\}$ trivially. Therefore
\[
U\cong(\mathbb F_q,+)
\cong
\underbrace{C_r\times\cdots\times C_r}_{d\text{ factors}}.
\]
The Sylow condition shows that $U$ is cyclic. Therefore $d=1$, so $q=r$ is prime and the field automorphism group is trivial. It follows that
\[
G\cong\operatorname{PSL}_2(r)
\]
with $r>5$.

For $t\in\mathbb F_r$ and $a\in\mathbb F_r^\times$, write
\[
u(t)=
\left[
\begin{pmatrix}
1&t\\
0&1
\end{pmatrix}
\right],
\qquad
 d(a)=
\left[
\begin{pmatrix}
a&0\\
0&a^{-1}
\end{pmatrix}
\right].
\]
Direct multiplication gives
\[
d(a)u(t)d(a)^{-1}=u(a^2t).
\]
Therefore the square subgroup of $\mathbb F_r^\times$, which has order $(r-1)/2$, acts faithfully by conjugation on $U\cong C_r$. By the normalizer condition, the image of this conjugation action is contained in the group generated by the identity and inversion. Therefore
\[
\frac{r-1}{2}\leq 2.
\]
This is impossible for $r>5$. Therefore every group arising in the Klein-four case is already contained in the list of \cref{thm:maintheorem}.

\medskip
\noindent \textbf{Dihedral Sylow $2$-subgroups of order at least $8$.} Now suppose that $|P|\geq 8$. By the Gorenstein--Walter classification \cite{GW65I,GW65II,GW65III}, and by \eqref{eq:oddcoretrivial}, one of the following holds:
\begin{enumerate}[label=\textup{(\arabic*)}]
\item $G$ is a $2$-group.

\item $G\cong A_7$.

\item For some odd prime power $q$,
\[
\operatorname{PSL}_2(q)\triangleleft G
\mathrel{\leq} \mathrm{P}\Gamma\mathrm{L}_2(q).
\]
\end{enumerate}

The first case is impossible. We would have $G=P$, so $G$ would be a dihedral $2$-group of order at least $8$. Its cyclic rotation subgroup $H$ has order at least $4$ and satisfies $N_G(H)=G$, contradicting \eqref{eq:propernormalizers}.

The group $A_7$ is impossible. Label the seven letters by $\mathbb F_7$, and let
\[
a\colon x\longmapsto x+1.
\]
Then $a$ is a $7$-cycle. Multiplication by $2$ gives the even permutation
\[
b=(1\ 2\ 4)(3\ 6\ 5)\in A_7.
\]
It satisfies
\[
bab^{-1}=a^2.
\]
Therefore $b$ normalizes $\langle a\rangle\cong C_7$ and sends a generator to its square. Since
\[
2\not\equiv\pm1\pmod 7,
\]
this contradicts the normalizer condition.

It remains to consider case (3). Write $q=r^d$. As in the Klein-four case, the upper-unitriangular $r$-subgroup of $\operatorname{PSL}_2(q)$ is isomorphic to
\[
(\mathbb F_q,+)
\cong
\underbrace{C_r\times\cdots\times C_r}_{d\text{ factors}}.
\]
The Sylow condition therefore forces $d=1$, so $q=r$ is prime and there are no nontrivial field automorphisms. Therefore
\[
\mathrm{P}\Gamma\mathrm{L}_2(r)=\operatorname{PGL}_2(r).
\]
With the same diagonal conjugation argument used above, the normalizer condition gives
\[
\frac{r-1}{2}\leq2.
\]
Therefore $r=3$ or $5$.

If $r=3$, then
\[
\operatorname{PSL}_2(3)\cong A_4.
\]
The group $\operatorname{PGL}_2(3)$ acts faithfully on the four points of $\mathbb P^1(\mathbb F_3)$ and has order $24$. Therefore
\[
\operatorname{PGL}_2(3)\cong S_4.
\]
Since
\[
[\operatorname{PGL}_2(3):\operatorname{PSL}_2(3)]=2,
\]
the group $G$ is isomorphic to $A_4$ or $S_4$. The Sylow $2$-subgroup of $A_4$ is $C_2\times C_2$, so only $S_4$ occurs in the present branch. This group is $O_3$ and is contained in \cref{thm:maintheorem}.

If $r=5$, then
\[
\operatorname{PSL}_2(5)\cong A_5.
\]
Since $\operatorname{PSL}_2(5)$ has index $2$ in $\operatorname{PGL}_2(5)$, the only possibilities are $A_5$ and $\operatorname{PGL}_2(5)$. We exclude the second possibility. In $\operatorname{PGL}_2(5)$, the class of
\[
\begin{pmatrix}
2&0\\
0&1
\end{pmatrix}
\]
normalizes the unipotent subgroup
\[
U=
\left\{
\left[
\begin{pmatrix}
1&t\\
0&1
\end{pmatrix}
\right]
\,\middle|\,
t\in\mathbb F_5
\right\}
\cong C_5
\]
and acts by
\[
u(t)\longmapsto u(2t).
\]
Multiplication by $2$ has order $4$ in $\operatorname{Aut}(C_5)$, contradicting the normalizer condition. Therefore $\operatorname{PGL}_2(5)$ cannot occur. The remaining group $A_5$ is already in \cref{thm:maintheorem}. Its Sylow $2$-subgroup is Klein four, so it belongs to the previous branch in this case division.

We have proved that every group in the final branch with $|P|\geq8$ is isomorphic to $S_4$. This group is contained in the list of \cref{thm:maintheorem}.

\medskip
\noindent \textbf{Completion of the induction.} The normal-complement case above treated the branch in which Burnside's normal $p$-complement theorem applies. \cref{inductnormal} treated the branch in which $G=N_G(H)$ for a nontrivial cyclic $p$-subgroup $H$ of order at least $3$. The argument in this section treats the complementary branch in which all such normalizers are proper. These alternatives exhaust the induction. We have proved that every finite group satisfying the Sylow and normalizer conditions lies in the list of \cref{thm:maintheorem}. Every group in this list is Anderson-dualizing by the constructions in \cref{sec:Verification}. Together with part \textup{(1)} of \cref{localcriterion}, this completes the proof of \cref{thm:maintheorem}.

\section{\texorpdfstring{Anderson duality for $H\uFtwo$ and its
algebraic consequences}{Anderson duality for HF2 and its algebraic consequences}}
\label{sec:HF2application}

In this section we prove part \textup{(2)} of \cref{thm:HF2groups-v2}, parts \textup{(2)} of \cref{thm:HF2shifts-v2,thm:HF2periodicity-intro}, and \cref{thm:HF2selfinjective-v2}. We first calculate $I_{\mathbb Z}^G(H\uFtwo)$ and establish Sylow $2$-detection. The integral periodicity criterion then determines $H\uFtwo$-periodicity for every finite group. The local integral shifts, subgroup obstructions, and \cref{lem:prime-power-realization} give the classification of dualizing groups and shifts. We finish with the universal coefficient theorem, coefficient-ring duality, and graded self-injectivity.
\par

\subsection{From integral duality to Anderson duality for \texorpdfstring{$H\uFtwo$}{HF2}}

\begin{prop}\label[prop]{prop:HF2dual-v2}
For every finite group $G$, there is a natural equivalence of $H\uFtwo$-modules
\[
I_{\mathbb Z}^G(H\uFtwo)\simeq \Sigma^{-1}H\uFtwo^{*},
\]
where $\uFtwo^{*}$ denotes the Mackey dual of the constant Mackey functor $\uFtwo$. Consequently, for $V\in RO(G)$, $(G,V)$ is $H\uFtwo$-Anderson-dualizing if and only if there is an $H\uFtwo$-linear equivalence
\begin{equivalence}\label{eq:HF2linearcriterion}
H\uFtwo\wedge S^{1-V}\simeq H\uFtwo^{*}.
\end{equivalence}
\end{prop}

\begin{proof}
Let $E=H\uFtwo$ and $X=S^0$. Since
\[
\uHom_{\mathbb Z}(\uFtwo,\mathbb Z)=0,
\qquad
\uExt^1_{\mathbb Z}(\uFtwo,\mathbb Z)\cong\uFtwo^{*},
\]
the universal coefficient sequence \labelcref{ses:equivariantUCT} shows that $I_{\mathbb Z}^G(H\uFtwo)$ has integer-graded homotopy concentrated in degree $-1$, with Mackey functor $\uFtwo^{*}$. Therefore
\[
I_{\mathbb Z}^G(H\uFtwo)\simeq\Sigma^{-1}H\uFtwo^{*}.
\]

The Anderson dual carries its natural $H\uFtwo$-module structure. The argument of \cref{lem:HZPostnikov} applies equally to $H\uFtwo$-modules concentrated in degree zero, because every map of their homotopy Mackey functors is automatically $\uFtwo$-linear. Applying this to $\Sigma I_{\mathbb Z}^G(H\uFtwo)$ makes the displayed equivalence $H\uFtwo$-linear.

A duality equivalence now gives
\[
H\uFtwo\wedge S^{1-V}\simeq H\uFtwo^{*}.
\]
Both sides are concentrated in integer degree zero, so the same argument makes this equivalence $H\uFtwo$-linear. The converse is immediate.
\end{proof}

The following proposition records the precise way in which the integral theorem enters the $H\uFtwo$-Anderson-duality problem.

\begin{prop}\label[prop]{prop:HZtoHF2}
If $(G,U)$ is an Anderson-dualizing pair for $H\uZ$, then $(G,U+1)$ is an $H\uFtwo$-Anderson-dualizing pair.  The resulting equivalence can be chosen in the category of $H\uFtwo$-modules.
\end{prop}

\begin{proof}
Taking $n=2$ in \cref{rmk:reductionmodn} gives the shift $U+1$. The equivalence may be chosen to be $H\uFtwo$-linear by \cref{prop:HF2dual-v2}.
\end{proof}
\par

The next theorem is the $H\uFtwo$ analogue of \cref{decomposetopsubgps}.  It also explains why the classification depends only on a Sylow $2$-subgroup.

\begin{thm}[Sylow $2$-detection]\label[thm]{thm:HF2Sylow-v2}
Let $P\mathrel{\leq} G$ be a Sylow $2$-subgroup and let $V\in RO(G)$.  Then $(G,V)$ is $H\uFtwo$-Anderson-dualizing if and only if $(P,V|_P)$ is $H\uFtwo$-Anderson-dualizing.  When these conditions hold, the duality equivalence can be chosen to be $H\uFtwo$-linear.
\end{thm}

\begin{proof}
The ``only if'' direction follows by restriction. For the ``if'' direction, put
\[
X=H\uFtwo\wedge S^{1-V}.
\]
\Cref{prop:HF2dual-v2} and the hypothesis give $X|_P\simeq H\uFtwo^*$. Every $2$-subgroup is conjugate into $P$, so the same identification holds on every $2$-subgroup.

Regard $X$ as an $H\uZ$-module by restriction of scalars. Its homotopy groups are $\Ftwo$-vector spaces, so localization at $2$ does not change them. For $K\mathrel{\leq} G$, choose a Sylow $2$-subgroup $Q\mathrel{\leq} K$ and apply \cref{prop:LiuLimit} to $X|_K$.

The diagram vanishes outside integer degree zero. In degree zero, its values are $\Ftwo$, its proper restriction maps are zero, and its conjugation maps are identities. A compatible family has equal components at the Sylow $2$-subgroups and zero components at every non-Sylow $2$-subgroup. Consequently,
\[
\res_Q^K\colon M(K)\xrightarrow{\;\cong\;}M(Q)\cong\Ftwo,
\qquad M=\underline\pi_0(X),
\]
and all other integer-graded homotopy Mackey functors vanish.

The cohomological relation gives
\[
\tr_Q^K\res_Q^K=[K:Q]=1
\]
over $\Ftwo$. Therefore $\tr_Q^K$ is an isomorphism. Since the local identification with $\uFtwo^*$ makes $\tr_e^Q$ an isomorphism, every $\tr_e^K$ is an isomorphism.

Identifying the values through these transfers makes all transfers identities. The cohomological relation makes restrictions multiplication by subgroup indices, and conjugation is trivial. Therefore $M\cong\uFtwo^*$. The module-linearity argument in \cref{prop:HF2dual-v2} gives the required $H\uFtwo$-linear equivalence.
\end{proof}
\par

\subsection{\texorpdfstring{Periodicity of $H\uFtwo$}{Periodicity of HF2}}

We prove part \textup{(2)} of \cref{thm:HF2periodicity-intro} using integral periodicity and Sylow $2$-detection.

\enlargethispage{\baselineskip}
\begin{lem}\label[lem]{lem:HF2geofixed-v2}
Let $H$ be a finite $2$-group. The spectrum $\Phi^H H\uFtwo$ is connective and
\[
\pi_0\Phi^H H\uFtwo\cong \Ftwo.
\]
Consequently, if $d\in\mathbb Z$ and
\[
\Sigma^d\Phi^H H\uFtwo\simeq \Phi^H H\uFtwo,
\]
then $d=0$.
\end{lem}

\begin{proof}
If $H=e$, the assertion is immediate. If $H$ is nontrivial, the same $H$-CW filtration argument as in \cref{geofixedhz} proves connectivity and identifies $\pi_0$ with the quotient by the proper-subgroup transfers. Every such transfer is multiplication by the even index $[H:K]$ and therefore vanishes over $\Ftwo$, so $\pi_0\Phi^H H\uFtwo\cong\Ftwo$.
\end{proof}

\begin{prop}
\label[prop]{prop:HF2pgroup-periodicity}
Let $P$ be a finite $2$-group and let $W\in RO(P)$. The following conditions are equivalent.
\begin{enumerate}[label=\textup{(\arabic*)},nosep]
\item $H\uFtwo\wedge S^W\simeq H\uFtwo$.
\item $|W^H|=0$ for every subgroup $H\mathrel{\leq} P$.
\item $W\in\Lambda_P(H\uZ)$.
\end{enumerate}
When these conditions hold, the equivalence in \textup{(1)} may be chosen to be $H\uFtwo$-linear.
\end{prop}

\begin{proof}
If \textup{(1)} holds, restriction to $H$ and geometric fixed points give
\[
\Sigma^{|W^H|}\Phi^H H\uFtwo\simeq\Phi^H H\uFtwo.
\]
\cref{lem:HF2geofixed-v2} gives \textup{(2)}. The equivalence of \textup{(2)} and \textup{(3)} is \cref{prop:pgroup-periodicity}. If \textup{(3)} holds, the proof of \cref{thm:maintheoremC} gives an $H\uZ$-linear equivalence $H\uZ\wedge S^W\simeq H\uZ$. Extension of scalars along $H\uZ\to H\uFtwo$ gives an $H\uFtwo$-linear equivalence in \textup{(1)}.
\end{proof}

By \cref{lem:cyclic-dimension-detection}, it is enough in \textup{(2)} to test the trivial subgroup and the nontrivial cyclic subgroups of $P$.

\begin{lem}
\label[lem]{lem:HF2periodicity-Sylow}
Let $G$ be a finite group, let $P\mathrel{\leq} G$ be a Sylow $2$-subgroup, and let $W\in RO(G)$. Then
\[
H\uFtwo\wedge S^W\simeq H\uFtwo
\quad\Longleftrightarrow\quad
H\uFtwo\wedge S^{W|_P}\simeq H\uFtwo.
\]
When these conditions hold, the global equivalence may be chosen to be $H\uFtwo$-linear.
\end{lem}

\begin{proof}
The forward implication follows by restriction. For the converse, set $X=H\uFtwo\wedge S^W$. Every $2$-subgroup of $G$ is conjugate to a subgroup of $P$, so the hypothesis gives $X|_H\simeq H\uFtwo$ for every $2$-subgroup $H\mathrel{\leq} G$. The spectrum $X$ is an $H\uZ$-module by restriction of scalars, and all its homotopy groups are $\Ftwo$-vector spaces. Localizing at $2$ therefore does not change them.

For every subgroup $K\mathrel{\leq} G$, apply \cref{prop:LiuLimit} to $X|_K$ at the prime $2$. The diagram is zero in integer degrees other than zero. In degree zero, all its values are $\Ftwo$, and all restrictions between $2$-subgroups are identities. The conjugation maps are also identities, since $\operatorname{Aut}_{\Ftwo}(\Ftwo)=1$. Every object maps by restriction to the trivial subgroup, so the limit is $\Ftwo$. The compatibility with subdiagrams in \cref{prop:LiuLimit} shows that all restrictions between the groups $\pi_0^K(X)$ are identities.

The Mackey functor $\underline\pi_0(X)$ is a module over $\uFtwo$, so the cohomological relation
\[
\tr_L^K\res_L^K=[K:L]
\]
shows that its transfers are multiplication by subgroup indices. Therefore $\underline\pi_0(X)\cong\uFtwo$, and all other integer-graded homotopy Mackey functors vanish. The module-linearity argument in the proof of \cref{prop:HF2dual-v2} gives an $H\uFtwo$-linear equivalence $X\simeq H\uFtwo$.
\end{proof}

\begin{proof}[Proof of part \textup{(2)} of \cref{thm:HF2periodicity-intro}]
Let $P$ be a Sylow $2$-subgroup. By \cref{lem:HF2periodicity-Sylow} and \cref{prop:HF2pgroup-periodicity}, $W$ is a period if and only if $|W^H|=0$ for every $H\mathrel{\leq} P$. By \cref{lem:cyclic-dimension-detection}, it is enough to test $e$ and the nontrivial cyclic subgroups of $P$. Every cyclic $2$-subgroup of $G$ is conjugate into $P$, so this is exactly the stated criterion. The equivalence may be chosen to be $H\uFtwo$-linear by \cref{lem:HF2periodicity-Sylow}.

Let $C_1,\ldots,C_r$ represent the $G$-conjugacy classes of nontrivial cyclic $2$-subgroups. The periodicity lattice $\Lambda_G(H\uFtwo)$ is the kernel of
\[
RO(G)\longrightarrow\mathbb Z^{r+1},
\qquad W\longmapsto\bigl(|W|,|W^{C_1}|,\ldots,|W^{C_r}|\bigr).
\]
These coordinate functionals are linearly independent by \cref{lem:fixed-dimension-rank}. The quotient by the kernel is the image, which is free abelian of rank $r+1$.
\end{proof}

\begin{cor}
\label[cor]{cor:HF2integral-periodicity}
Let $G$ be a finite group and let $P$ be a Sylow $2$-subgroup. Then
\[
\Lambda_G(H\uFtwo)
=(\res_P^G)^{-1}(\Lambda_P(H\uZ)).
\]
This lattice contains $\Lambda_G(H\uZ)$, and its quotient by $\Lambda_G(H\uZ)$ is free abelian of rank equal to the number of $G$-conjugacy classes of nontrivial cyclic subgroups of odd prime-power order. The integral and mod-$2$ periodicity lattices coincide if and only if $G$ is a $2$-group.
\end{cor}

\begin{proof}
The displayed equality follows from \cref{lem:HF2periodicity-Sylow} and \cref{prop:HF2pgroup-periodicity}. The inclusion of $\Lambda_G(H\uZ)$ follows from \cref{thm:maintheoremC} and part \textup{(2)} of \cref{thm:HF2periodicity-intro}. The quotient by $\Lambda_G(H\uZ)$ is a subgroup of the free abelian group $RO(G)/\Lambda_G(H\uZ)$, so it is free abelian. Subtracting the ranks in these two parts gives the stated rank. This rank is zero exactly when $G$ has no nontrivial cyclic subgroup of odd prime-power order, which by Cauchy's theorem is equivalent to $G$ being a $2$-group.
\end{proof}

\begin{rmk}\label[rmk]{rmk:HF2suspension-periodicity}\rm
Applying part \textup{(2)} of \cref{thm:HF2periodicity-intro} to $V-W$ also classifies when $H\uFtwo\wedge S^V\simeq H\uFtwo\wedge S^W$: the underlying dimensions must agree, as must the fixed-point dimensions on every nontrivial cyclic $2$-subgroup.

By the module-linearity in the proof of part \textup{(2)} of \cref{thm:HF2periodicity-intro}, $\Lambda_G(H\uFtwo)$ is also the kernel of the representation-sphere map
\[
RO(G)\longrightarrow
\Pic\!\left(\operatorname{Mod}_{H\uFtwo}(Sp^G)\right),
\qquad U\longmapsto[H\uFtwo\wedge S^U],
\]
and the set of degrees containing homogeneous units in $\pi_\star^G H\uFtwo$.

If $|G|$ is odd, the periodicity lattice is exactly the subgroup of virtual representations of dimension zero.  For $G=C_3$, this includes $\lambda(1,3)-2$. Its $C_3$-fixed dimension is $-2$, so it is not an integral period.
\end{rmk}

\begin{rmk}\label[rmk]{rmk:HF2explicit-periodicity}\rm
The explicit local lattices are those in \cref{cor:cyclicperiodicity} and \cref{cor:dihedralperiodicity}. For $P=C_{2^n}$, a restricted class $U=W|_P$ is a period if and only if $|U|=0$ and $\mu_j(U)=0$ for $0\leq j\leq n-1$. For $P=C_2\times C_2$, this condition is $U=0$. For $P=C_2\ltimes C_{2^n}$ with $n\geq2$, the trivial and sign coefficients of $U$ must vanish, and $\mu_j(U)=0$ for $0\leq j\leq n-2$.
\end{rmk}

\subsection{The local \texorpdfstring{$2$}{2}-group classification}

The local coefficient calculation is inherited from \cref{sec:LocalClassification}. The periodicity results above allow us to compare every local dualizing shift with a basic one.

\begin{prop}\label[prop]{prop:HF2local-v2}
Let $P$ be a nontrivial cyclic or dihedral $2$-group and let $V\in RO(P)$. Then $(P,V)$ is $H\uFtwo$-Anderson-dualizing if and only if $V-1$ satisfies the corresponding coefficient condition in \cref{cyclicpgroup} or \cref{dihedralgoodrep}.

In particular, the basic $H\uFtwo$-Anderson-dualizing shifts are
\begin{enumerate}
    \item $P=C_{2^n}$: $\lambda(1,2^n)-1$.
    \item $P=C_2\times C_2$: $\sigma_1+\sigma_2+\sigma_3-2$.
    \item $P=C_2\ltimes C_{2^n}$, $n\geq 2$: $\rho(1,2^n)+\sigma_1-2$, where $\sigma_1$ is the sign representation with kernel the rotation subgroup.
\end{enumerate}
\end{prop}

\begin{proof}
Let $U_0$ be the corresponding basic integral shift from \cref{pgrouprep}, and put $V_0=U_0+1$. By \cref{prop:HZtoHF2}, $V_0$ is an $H\uFtwo$-Anderson-dualizing shift. Comparing with this duality gives
\[
I_{\mathbb Z}^P(H\uFtwo)\wedge S^V
\simeq H\uFtwo\wedge S^{V-V_0}.
\]
By \cref{prop:HF2pgroup-periodicity}, $V$ is a dualizing shift if and only if $V-V_0\in\Lambda_P(H\uZ)$. Since $U_0\in\mathcal S_P$, \cref{prop:affine-shifts} identifies this condition with $V-1\in\mathcal S_P$. The coefficient conditions follow from \cref{cyclicpgroup} and \cref{dihedralgoodrep}.
\end{proof}
\par

\begin{lem}\label[lem]{lem:HF2obstructions-v2}
The groups
\[
C_2\times C_2\times C_2,\qquad C_2\times C_4,\qquad Q_8
\]
are not $H\uFtwo$-Anderson-dualizing.
\end{lem}

\begin{proof}
Suppose that $(G,V)$ were $H\uFtwo$-Anderson-dualizing for one of these groups, and put $U=V-1$. Restriction to every nontrivial cyclic subgroup $C\mathrel{\leq} G$, followed by \cref{prop:HF2local-v2}, shows that $U|_C$ satisfies the integral cyclic coefficient condition.  The proof of \cref{blacklist} uses only these cyclic restrictions, so the same contradiction applies to $U$.  This excludes all three groups.
\end{proof}

\begin{cor}\label{cor:HF2localgroups-v2}
A finite $2$-group is $H\uFtwo$-Anderson-dualizing if and only if it is cyclic or dihedral.  
\end{cor}

\begin{proof}
The trivial group is $H\uFtwo$-Anderson-dualizing with shift $1$. Therefore we may assume that the $2$-group is nontrivial.
The cyclic and dihedral groups are covered by \cref{prop:HF2local-v2}.  A finite $2$-group which is neither cyclic nor dihedral contains a subgroup isomorphic to $C_2\times C_2\times C_2$, $C_2\times C_4$, or $Q_8$ (see  \cite[Lemma~6.1 and Corollary~6.2]{BoucStancuWebb17}).  An $H\uFtwo$-Anderson duality restricts to subgroups, and \cref{lem:HF2obstructions-v2} excludes each of these three groups.
\end{proof}

\subsection{\texorpdfstring{Proofs of
part \textup{(2)} of \cref{thm:HF2groups-v2} and part \textup{(2)} of \cref{thm:HF2shifts-v2}}{Proofs of part (2) of Theorem A and part (2) of Theorem C}}

\begin{proof}[Proof of part \textup{(2)} of \cref{thm:HF2groups-v2}]
Suppose first that $G$ admits an $H\uFtwo$-Anderson-dualizing shift.  By \cref{thm:HF2Sylow-v2}, a Sylow $2$-subgroup is $H\uFtwo$-Anderson-dualizing, and \cref{cor:HF2localgroups-v2} shows that it is cyclic or dihedral.

Conversely, let $P$ be a cyclic or dihedral Sylow $2$-subgroup.  If $P=e$, then $|G|$ is odd and $\uFtwo^{*}=\uFtwo$.  The equivalence in \cref{prop:HF2dual-v2} shows that $(G,1)$ is $H\uFtwo$-Anderson-dualizing.  We may therefore assume that $P\neq e$. Let $V_P\in RO(P)$ be the corresponding basic shift from \cref{prop:HF2local-v2}.

We proceed in three steps.
\begin{enumerate}[label=\textup{(\arabic*)}]
\item Compute the fixed-point dimension function of $V_P$.
\item  Extend this function to the prime-power subgroups of $G$ and check the hypotheses of \cref{lem:prime-power-realization}.
\item Realize the function by $V\in RO(G)$, compare $V|_P$ with $V_P$ using periodicity, and apply Sylow $2$-detection.
\end{enumerate}

We first compute
\begin{equation}\label{eq:HF2basic-fixeddim}
\dim_{\mathbb R}(V_P^H)=
\begin{cases}
1,&H=e,\\
-1,&H\neq e\text{ is cyclic},\\
-2,&H\text{ is noncyclic},
\end{cases}
\qquad H\mathrel{\leq} P.
\end{equation}
For $P=C_{2^n}$, the faithful rotation representation has no fixed vectors on any nontrivial subgroup.  For $P=C_2\times C_2$, each subgroup of order $2$ is the kernel of exactly one of the three nontrivial sign representations.  For a higher dihedral group, a nontrivial rotation subgroup fixes $\sigma_1$ and no vector in $\rho(1,2^n)$, while a reflection subgroup fixes a line in $\rho(1,2^n)$ and no vector in $\sigma_1$.  A noncyclic subgroup fixes no vector in either summand.

Next, define $f\colon\mathcal P(G)\to\mathbb Z$ by
\[
f(H)=
\begin{cases}
1,&H=e\text{ or }H\text{ has odd prime-power order},\\
-1,&H\neq e\text{ is a cyclic }2\text{-group},\\
-2,&H\text{ is a noncyclic }2\text{-group}.
\end{cases}
\]
This function is invariant under conjugation. On an odd Sylow subgroup, it is the dimension function of the trivial representation. On $P$, it is the dimension function of $V_P$ by \cref{eq:HF2basic-fixeddim}.

We check the congruence in \cref{lem:prime-power-realization}. If $p$ is odd, then $f(L)=f(H)=1$. If $p=2$, the action on $H/L\cong C_2$ is trivial, so the modulus is $1$. Therefore \cref{lem:prime-power-realization} gives $V\in RO(G)$ with
\[
|V^H|=f(H)
\]
for every prime-power subgroup $H\mathrel{\leq} G$. In particular, $V|_P$ and $V_P$ have the same fixed-point dimensions on every subgroup of $P$.
By \cref{prop:HF2pgroup-periodicity}, $V|_P-V_P$ is an $H\uFtwo$-period. Comparing with the duality for $V_P$ shows that $V|_P$ is a dualizing shift for $P$. Sylow $2$-detection, \cref{thm:HF2Sylow-v2}, gives the required duality for $G$.
\end{proof}

\begin{prop}
\label[prop]{prop:HF2affine-shifts}
Let $G$ be a finite group and let $V_0\in RO(G)$ be an $H\uFtwo$-Anderson-dualizing shift. Then the set of all such shifts is
\[
V_0+\Lambda_G(H\uFtwo).
\]
\end{prop}

\begin{proof}
The duality with shift $V_0$ gives
\[
I_{\mathbb Z}^G(H\uFtwo)\wedge S^V
\simeq H\uFtwo\wedge S^{V-V_0}.
\]
Therefore $V$ is a dualizing shift if and only if $V-V_0\mathrel{\in}\Lambda_G(H\uFtwo)$.
\end{proof}

\begin{proof}[Proof of part \textup{(2)} of \cref{thm:HF2shifts-v2}]
First assume that $G$ is $H\uFtwo$-Anderson-dualizing. Let $P$ be a Sylow $2$-subgroup. By part \textup{(2)} of \cref{thm:HF2groups-v2}, $P$ is cyclic or dihedral. If $P=e$, \cref{thm:HF2Sylow-v2} and \cref{prop:HF2dual-v2} give the condition $|V|=1$.

If $P\neq e$, let $V_P$ be the basic shift in \cref{prop:HF2local-v2}. By \cref{thm:HF2Sylow-v2} and \cref{prop:HF2affine-shifts}, $V$ is a dualizing shift for $G$ if and only if $V|_P-V_P$ is a period for $P$. By \cref{prop:HF2pgroup-periodicity}, \cref{lem:cyclic-dimension-detection}, and \cref{eq:HF2basic-fixeddim}, this is equivalent to
\[
|V|=1
\qquad\text{and}\qquad
|V^C|=-1
\]
for every nontrivial cyclic subgroup $C\mathrel{\leq} P$. Every cyclic $2$-subgroup of $G$ is conjugate into $P$, which proves the result in this case. The general case follows from \cref{rmk:HF2dimension-existence}.
\end{proof}

\begin{rmk}
\label[rmk]{rmk:HF2dimension-existence}\rm
The dimension conditions in part \textup{(2)} of \cref{thm:HF2shifts-v2} already imply that $G$ is $H\uFtwo$-Anderson-dualizing. More precisely, suppose that $G$ is any finite group and $V\in RO(G)$ satisfies $|V|=1$ and $|V^C|=-1$ for every nontrivial cyclic $2$-subgroup $C\mathrel{\leq} G$. Put $U=V-1$. For each such $C$, the representations $U|_C$ and $\lambda(1,|C|)-2$ have the same fixed-point dimensions on every subgroup of $C$. By \cref{thm:maintheoremC}, their difference lies in $\Lambda_C(H\uZ)$. Since $\lambda(1,|C|)-2$ is a basic integral dualizing shift, \cref{prop:affine-shifts} shows that $U|_C$ satisfies the integral cyclic coefficient condition in \cref{cyclicpgroup}.

The cyclic restriction calculations in the proof of \cref{blacklist} therefore exclude $C_2\times C_2\times C_2$, $C_4\times C_2$, and $Q_8$ as subgroups of a Sylow $2$-subgroup $P$. It follows that $P$ is cyclic or dihedral \cite[Lemma~6.1]{BoucStancuWebb17}. By part \textup{(2)} of \cref{thm:HF2groups-v2}, $G$ is $H\uFtwo$-Anderson-dualizing, and the case just proved shows that $V$ is a dualizing shift.
\end{rmk}

\begin{cor}
\label[cor]{cor:HF2Sylow-coefficients}
Let $G$ be a finite group whose Sylow $2$-subgroup $P$ is cyclic or dihedral, and let $V\in RO(G)$. If $P=e$, then $(G,V)$ is $H\uFtwo$-Anderson-dualizing if and only if $|V|=1$. If $P\neq e$, then $(G,V)$ is $H\uFtwo$-Anderson-dualizing if and only if $V|_P-1$ satisfies the cyclic coefficient condition of \cref{cyclicpgroup} when $P$ is cyclic, or the dihedral coefficient condition of \cref{dihedralgoodrep} when $P$ is dihedral.
\end{cor}

\begin{proof}
Let $P$ be a Sylow $2$-subgroup of $G$.  By \cref{thm:HF2Sylow-v2}, the pair $(G,V)$ is $H\uFtwo$-Anderson-dualizing if and only if $(P,V|_P)$ is.  If $P=e$, \cref{prop:HF2dual-v2} gives the condition $|V|=1$.  If $P\neq e$, \cref{prop:HF2local-v2} identifies the condition with the cyclic coefficient condition of \cref{cyclicpgroup} or the dihedral coefficient condition of \cref{dihedralgoodrep} for $V|_P-1$.
\end{proof}

\begin{cor}
\label{cor:HF2basic-consequences-v2}
The following statements hold.
\begin{enumerate}
\item If $|G|$ is odd, then $(G,V)$ is $H\uFtwo$-Anderson-dualizing if and only if $|V|=1$.
\item If $G\mathrel{\leq} SO(3)$ is finite and $W$ is the restriction of the standard three-dimensional representation, then $(G,W-2)$ is $H\uFtwo$-Anderson-dualizing.
\end{enumerate}
\end{cor}

\begin{proof}
For an odd-order group, the Sylow $2$-subgroup is trivial, and part \textup{(2)} of \cref{thm:HF2shifts-v2} reduces the duality condition to $|V|=1$.  For a finite subgroup of $SO(3)$, \cref{listremark} gives the integral dualizing shift $W-3$, and \cref{prop:HZtoHF2} increases the shift by one.
\end{proof}

\subsection{\texorpdfstring{Universal coefficients and
shift-Frobenius duality}{Universal coefficients and shift-Frobenius duality}}

\begin{proof}[Proof of part \textup{(1)} of \cref{thm:HF2selfinjective-v2}]
Let $X$ be a $G$-spectrum.  By adjunction and the chosen $H\uFtwo$-Anderson duality, there are natural equivalences
\[
I_{\mathbb Z}^G(H\uFtwo\wedge X)
\simeq
F_G\!\left(X,I_{\mathbb Z}^G(H\uFtwo)\right)
\simeq
\Sigma^{-V}F_G(X,H\uFtwo).
\]
Apply \cref{ses:equivariantUCT} with $E=H\uFtwo$. Every orbitwise homology group is an $\Ftwo$-vector space. For any $\Ftwo$-vector space $A$, the connecting homomorphism for
\[
0\longrightarrow\mathbb Z\longrightarrow\mathbb Q
\longrightarrow\mathbb Q/\mathbb Z\longrightarrow0
\]
gives natural identifications
\[
\begin{gathered}
\operatorname{Hom}_{\mathbb Z}(A,\mathbb Z)=0,\\
\operatorname{Ext}_{\mathbb Z}^1(A,\mathbb Z)
\cong \operatorname{Hom}_{\mathbb Z}(A,\mathbb Q/\mathbb Z)
\cong \operatorname{Hom}_{\Ftwo}(A,\Ftwo).
\end{gathered}
\]
Therefore the $\operatorname{Hom}_{\mathbb Z}$ term in \cref{ses:equivariantUCT} vanishes orbitwise, while the $\operatorname{Ext}_{\mathbb Z}^1$ term is the orbitwise $\Ftwo$-linear dual.  Substituting the duality equivalence into that sequence and reindexing gives
\[
H\uFtwo^{\star+1-V}(X)
\cong
\uHom_{\Ftwo}\!\bigl({H\uFtwo}_{\star}(X),\Ftwo\bigr).
\]
Naturality of both the universal coefficient sequence and the Ext--Hom identification gives naturality in $X$, with restrictions dual to homology transfers and transfers dual to homology restrictions.
\end{proof}

Let $R_G=\pi_\star^G H\uFtwo$ be the $RO(G)$-graded coefficient ring.  For an $RO(G)$-graded $\Ftwo$-vector space $M$, define
\[
D(M)_\alpha=\operatorname{Hom}_{\Ftwo}(M_{-\alpha},\Ftwo),
\qquad
(\Sigma^W M)_\alpha=M_{\alpha-W}.
\]
The $R_G$-action on $D(R_G)$ is the contragredient action: for homogeneous $r\in(R_G)_\beta$, $\phi\in D(R_G)_\alpha$, and $x\in(R_G)_{-\alpha-\beta}$,
\[
(r\cdot\phi)(x)=\phi(xr).
\]

\begin{lem}\label[lem]{lem:HF2dualinjective-v2}
The graded $R_G$-module $D(R_G)$ is injective in the category of $RO(G)$-graded left $R_G$-modules and degree-zero homomorphisms.
\end{lem}

\begin{proof}
The category of $RO(G)$-graded $\Ftwo$-vector spaces is a product of copies of $\operatorname{Vect}_{\Ftwo}$ indexed by $RO(G)$, so every object is injective.  The forgetful functor from graded $R_G$-modules to graded $\Ftwo$-vector spaces is exact, and its right adjoint is graded coinduction. Therefore coinduction preserves injective objects. Coinducing $\Ftwo$ concentrated in degree zero gives
\[
\underline{\operatorname{Hom}}_{\Ftwo}(R_G,\Ftwo)_\alpha
=\operatorname{Hom}_{\Ftwo}((R_G)_{-\alpha},\Ftwo)
=D(R_G)_\alpha.
\]
This identifies $D(R_G)$ with a coinduced module, so it is injective.
\end{proof}

\begin{proof}[Proof of part \textup{(2)} of \cref{thm:HF2selfinjective-v2}]
Taking $X=S^0$ in part \textup{(1)} and evaluating at $G/G$ gives an isomorphism of $RO(G)$-graded $\Ftwo$-vector spaces
\[
D(R_G)\cong \Sigma^{1-V}R_G.
\]
By \cref{prop:HF2dual-v2}, the duality equivalence may be chosen to be $H\uFtwo$-linear.  For homogeneous $r\in(R_G)_\beta$, multiplication is represented by an $H\uFtwo$-module map $m_r\colon\Sigma^\beta H\uFtwo\to H\uFtwo$.  Naturality identifies precomposition by $m_r$ with $\phi\mapsto(x\mapsto\phi(xr))$.  Therefore the displayed isomorphism is $R_G$-linear.

Put $\delta=V-1$, and let
\[
\theta\colon \Sigma^{-\delta}R_G\xrightarrow{\ \cong\ }D(R_G)
\]
be the inverse of this isomorphism.  The image under $\theta$ of the suspended unit in $(\Sigma^{-\delta}R_G)_{-\delta}=(R_G)_0$ is a linear map
\[
\epsilon\colon (R_G)_\delta\longrightarrow\Ftwo.
\]
For every $\gamma\in RO(G)$, the degree $-\gamma$ part of $\theta$ is an isomorphism
\[
(R_G)_{\delta-\gamma}
\xrightarrow{\ \cong\ }
\operatorname{Hom}_{\Ftwo}((R_G)_\gamma,\Ftwo).
\]
By $R_G$-linearity and the contragredient action, this isomorphism sends $y$ to the functional $x\mapsto\epsilon(xy)$.  Therefore multiplication followed by $\epsilon$ gives the perfect pairing in the statement.  Since $S^\gamma$ is a finite $G$-CW spectrum and $\uFtwo$ has finite values, each graded piece of $R_G$ is finite-dimensional.  Therefore
\[
\dim_{\Ftwo}(R_G)_\gamma
=\dim_{\Ftwo}(R_G)_{V-1-\gamma}.
\]
This is the locally finite $RO(G)$-graded shift-Frobenius property with shift $V-1$.

The module $D(R_G)$ is injective by \cref{lem:HF2dualinjective-v2}, and grading suspension preserves injective objects.  Therefore $R_G$ is injective as a graded module over itself.
\end{proof}

\begin{cor}\label[cor]{cor:HF2finite-free-splitting}
Let $G$ be a finite group whose Sylow $2$-subgroups are cyclic or dihedral, and put $R_G=\pi_\star^G H\uFtwo$. Every finite-rank graded free submodule of an $RO(G)$-graded left $R_G$-module is a graded direct summand.
\end{cor}

\begin{proof}
By part \textup{(2)} of \cref{thm:HF2groups-v2} and part \textup{(2)} of \cref{thm:HF2selfinjective-v2}, $R_G$ is injective in the category of $RO(G)$-graded left $R_G$-modules and degree-zero homomorphisms. Grading suspensions and finite direct sums preserve injectivity. Therefore every finite-rank graded free module is injective, so its inclusion into a graded $R_G$-module admits a degree-zero retraction.
\end{proof}

For $G=C_2$, May proved self-injectivity of the coefficient ring \cite[Proposition~4.2]{May20}. Related $C_2$ and Klein-four Anderson-duality calculations appear in \cite[Propositions~3.8 and~4.8]{EllisBloor20} and \cite[Proposition~4.2]{GuillouYarnall20}. \Cref{thm:HF2selfinjective-v2} gives an Anderson-dual explanation and extends the self-injectivity statement to all finite groups with cyclic or dihedral Sylow $2$-subgroups.

\bibliographystyle{alpha}
\bibliography{References}

@misc{And69,
  author       = {Anderson, D. W.},
  title        = {Universal coefficient theorems for {K}-theory},
  note         = {Univ. California, Berkeley, Calif. (1969)},
  year         = {1969}
}

@book{IsaacsFGT,
  author    = {Isaacs, I. Martin},
  title     = {Finite Group Theory},
  series    = {Graduate Studies in Mathematics},
  volume    = {92},
  publisher = {American Mathematical Society},
  year      = {2008}
}

@article{GW65I,
  author  = {Gorenstein, Daniel and Walter, John H.},
  title   = {The characterization of finite groups with dihedral {S}ylow 2-subgroups. {I}},
  journal = {Journal of Algebra},
  volume  = {2},
  number  = {1},
  pages   = {85--151},
  year    = {1965},
  doi     = {10.1016/0021-8693(65)90027-X}
}

@article{GW65II,
  author  = {Gorenstein, Daniel and Walter, John H.},
  title   = {The characterization of finite groups with dihedral {S}ylow 2-subgroups. {II}},
  journal = {Journal of Algebra},
  volume  = {2},
  number  = {2},
  pages   = {218--270},
  year    = {1965},
  doi     = {10.1016/0021-8693(65)90019-0}
}

@article{GW65III,
  author  = {Gorenstein, Daniel and Walter, John H.},
  title   = {The characterization of finite groups with dihedral {S}ylow 2-subgroups. {III}},
  journal = {Journal of Algebra},
  volume  = {2},
  number  = {3},
  pages   = {354--393},
  year    = {1965},
  doi     = {10.1016/0021-8693(65)90015-3}
}

@article{Wal69,
  author  = {Walter, John H.},
  title   = {The characterization of finite groups with abelian {S}ylow 2-subgroups},
  journal = {Ann. of Math. (2)},
  volume  = {89},
  number  = {3},
  pages   = {405--514},
  year    = {1969},
  doi     = {10.2307/1970648}
}

@article{May20,
  author  = {May, Clover},
  title   = {A structure theorem for {$RO(C_2)$}-graded {B}redon cohomology},
  journal = {Algebr. Geom. Topol.},
  volume  = {20},
  number  = {4},
  pages   = {1691--1728},
  year    = {2020},
  doi     = {10.2140/agt.2020.20.1691}
}

@misc{EllisBloor20,
  author       = {Ellis-Bloor, Benjamin},
  title        = {The {K}lein four homotopy {M}ackey functor structure of {$H\underline{\mathbb F}_2$}},
  howpublished = {Bachelor of Philosophy (Honours) thesis, Australian National University},
  year         = {2020},
  month        = oct,
  doi          = {10.25911/RM29-M298},
  url          = {https://doi.org/10.25911/RM29-M298}
}

@article{GuillouYarnall20,
  author  = {Guillou, Bertrand and Yarnall, Carolyn},
  title   = {The {K}lein four slices of {$\Sigma^n H\underline{\mathbb F}_2$}},
  journal = {Math. Z.},
  volume  = {295},
  number  = {3--4},
  pages   = {1405--1441},
  year    = {2020},
  doi     = {10.1007/s00209-019-02433-3}
}

@article{BC76,
  author       = {Brown, Edgar H. and Comenetz, Michael},
  title        = {Pontrjagin duality for generalized homology and cohomology theories},
  journal      = {Amer. J. Math.},
  volume       = {98},
  number       = {1},
  pages        = {1--27},
  year         = {1976},
  doi          = {10.2307/2373610}
}

@article {BHSZ21,
    AUTHOR = {Beaudry, Agn\`es and Hill, Michael A. and Shi, XiaoLin Danny and
              Zeng, Mingcong},
     TITLE = {Models of {L}ubin-{T}ate spectra via real bordism theory},
   JOURNAL = {Adv. Math.},
  FJOURNAL = {Advances in Mathematics},
    VOLUME = {392},
      YEAR = {2021},
     PAGES = {Paper No. 108020, 58},
      ISSN = {0001-8708},
   MRCLASS = {55N22 (11S31 55P42)},
  MRNUMBER = {4313964},
       DOI = {10.1016/j.aim.2021.108020},
       URL = {https://doi.org/10.1016/j.aim.2021.108020},
}

@article {BHLSZ25,
    AUTHOR = {Beaudry, Agn\`es and Hill, Michael A. and Lawson, Tyler and
              Shi, XiaoLin Danny and Zeng, Mingcong},
     TITLE = {On the slice spectral sequence for quotients of norms of real
              bordism},
   JOURNAL = {J. Topol.},
  FJOURNAL = {Journal of Topology},
    VOLUME = {18},
      YEAR = {2025},
    NUMBER = {1},
     PAGES = {Paper No. e70015, 66},
      ISSN = {1753-8416,1753-8424},
   MRCLASS = {55P91 (55P92 55Q91 55T15 55T25)},
  MRNUMBER = {4877605},
       DOI = {10.1112/topo.70015},
       URL = {https://doi.org/10.1112/topo.70015},
}

@article{Greenlees92,
  author  = {Greenlees, J. P. C.},
  title   = {Some remarks on projective {M}ackey functors},
  journal = {Journal of Pure and Applied Algebra},
  volume  = {81},
  number  = {1},
  year    = {1992},
  pages   = {17--38},
  doi     = {10.1016/0022-4049(92)90132-Y}
}

@article{GreenleesMay92,
  author  = {Greenlees, J. P. C. and May, J. P.},
  title   = {Some remarks on the structure of {M}ackey functors},
  journal = {Proceedings of the American Mathematical Society},
  volume  = {115},
  number  = {1},
  year    = {1992},
  pages   = {237--243},
  doi     = {10.1090/S0002-9939-1992-1076574-2}
}

@article{Kai71,
  author       = {Kainen, Paul C.},
  title        = {Universal coefficient theorems for generalized homology and stable cohomotopy},
  journal      = {Pacific J. Math.},
  volume       = {37},
  pages        = {397--407},
  year         = {1971},
  url          = {https://projecteuclid.org/euclid.pjm/1102975793}
}

@article{Ric16,
  author       = {Ricka, Nicolas},
  title        = {Equivariant {A}nderson duality and {M}ackey functor duality},
  journal      = {Glasg. Math. J.},
  volume       = {58},
  number       = {3},
  pages        = {649--676},
  year         = {2016},
  doi          = {10.1017/S0017089515000397},
  eprint       = {1408.1581},
  archivePrefix= {arXiv},
  primaryClass = {math.AT}
}

@article{HM17,
  author       = {Hill, Michael A. and Meier, Lennart},
  title        = {The {$C_2$}-spectrum {$\mathrm{Tmf}_1(3)$} and its invertible modules},
  journal      = {Algebr. Geom. Topol.},
  volume       = {17},
  number       = {4},
  pages        = {1953--2011},
  year         = {2017},
  doi          = {10.2140/agt.2017.17.1953},
  eprint       = {1507.08115},
  archivePrefix= {arXiv},
  primaryClass = {math.AT}
}

@article {HHR,
    AUTHOR = {Hill, M. A. and Hopkins, M. J. and Ravenel, D. C.},
     TITLE = {On the nonexistence of elements of {K}ervaire invariant one},
   JOURNAL = {Ann. of Math. (2)},
  FJOURNAL = {Annals of Mathematics. Second Series},
    VOLUME = {184},
      YEAR = {2016},
    NUMBER = {1},
     PAGES = {1--262},
      ISSN = {0003-486X},
   MRCLASS = {55P91 (55N22 55P42 55Q45 55T15 55U35 57R15)},
  MRNUMBER = {3505179},
MRREVIEWER = {Paul G. Goerss},
       DOI = {10.4007/annals.2016.184.1.1},
       URL = {https://doi.org/10.4007/annals.2016.184.1.1},
}

@article{Zen18,
  author       = {Zeng, Mingcong},
  title        = {Equivariant {E}ilenberg--{M}ac {L}ane spectra in cyclic {$p$}-groups},
  journal={ArXiv 1710.01769},
  year         = {2018},
  eprint       = {1710.01769},
  archivePrefix= {arXiv},
  primaryClass = {math.AT}
}

@article{Sto12,
  author       = {Stojanoska, Vesna},
  title        = {Duality for topological modular forms},
  journal      = {Documenta Math.},
  volume       = {17},
  pages        = {271--311},
  year         = {2012},
  url          = {https://emis.de/journals/DMJDMV/vol-17/09.html}
}

@article{HS14,
  author       = {Heard, Drew and Stojanoska, Vesna},
  title        = {{$K$}-theory, reality, and duality},
  journal      = {J. K-Theory},
  volume       = {14},
  number       = {3},
  pages        = {526--555},
  year         = {2014},
  doi          = {10.1017/is014007001jkt275},
  eprint       = {1401.2581},
  archivePrefix= {arXiv},
  primaryClass = {math.AT}
}

@article{Liu23,
  author        = {Liu, Yutao},
  title         = {Computing equivariant homotopy with a splitting method},
  journal       = {Topology and its Applications},
  volume        = {334},
  pages         = {Paper No. 108554},
  year          = {2023},
  doi           = {10.1016/j.topol.2023.108554},
  eprint        = {2203.14361},
  archivePrefix = {arXiv},
  primaryClass  = {math.AT},
}

@incollection{GS18,
  author       = {Greenlees, J. P. C. and Stojanoska, Vesna},
  title        = {Anderson and {G}orenstein duality},
  booktitle    = {Geometric and Topological Aspects of the Representation Theory of Finite Groups},
  series       = {Springer Proc. Math. Stat.},
  volume       = {242},
  pages        = {105--130},
  year         = {2018},
  publisher    = {Springer},
  doi          = {10.1007/978-3-319-94033-5_5},
  eprint       = {1705.02664},
  archivePrefix= {arXiv},
  primaryClass = {math.AT}
}

@article{Bauer89,
  author  = {Bauer, Stefan},
  title   = {A linearity theorem for group actions on spheres with applications to homotopy representations},
  journal = {Commentarii Mathematici Helvetici},
  volume  = {64},
  number  = {1},
  pages   = {167--172},
  year    = {1989},
  url     = {https://eudml.org/doc/140146}
}

@article{ReehYalcin18,
  author  = {Reeh, Sune Precht and Yal\c{c}{\i}n, Erg\"un},
  title   = {Representation rings for fusion systems and dimension functions},
  journal = {Mathematische Zeitschrift},
  volume  = {288},
  number  = {1--2},
  pages   = {509--530},
  year    = {2018},
  doi     = {10.1007/s00209-017-1898-8},
  eprint  = {1603.08237},
  archivePrefix = {arXiv},
  primaryClass = {math.AT}
}

@article{BoucStancuWebb17,
  author  = {Bouc, Serge and Stancu, Radu and Webb, Peter J.},
  title   = {On the projective dimensions of {M}ackey functors},
  journal = {Algebras and Representation Theory},
  volume  = {20},
  number  = {6},
  pages   = {1467--1481},
  year    = {2017},
  doi     = {10.1007/s10468-017-9695-y},
  eprint  = {1503.03955},
  archivePrefix = {arXiv},
  primaryClass = {math.RT}
}

@misc{Angeltveit22,
  author  = {Angeltveit, Vigleik},
  title   = {Recovering a cohomological {M}ackey functor from its restriction to {S}ylow subgroups},
  year    = {2022},
  eprint  = {2201.01444},
  archivePrefix = {arXiv},
  primaryClass = {math.AT},
  note    = {arXiv:2201.01444}
}

@misc{AngeltveitPicard,
  author        = {Angeltveit, Vigleik},
  title         = {The {P}icard group of the category of {$C_n$}-equivariant stable homotopy theory},
  year          = {2021},
  eprint        = {2109.04625},
  archivePrefix = {arXiv},
  primaryClass  = {math.AT},
  note          = {arXiv:2109.04625}
}

@article{BasuGhosh24,
  author        = {Basu, Samik and Ghosh, Surojit},
  title         = {Equivariant cohomology for cyclic groups of square-free order},
  journal       = {Res. Math. Sci.},
  volume        = {11},
  pages         = {Paper No. 29},
  year          = {2024},
  doi           = {10.1007/s40687-024-00443-0},
  eprint        = {2006.09669},
  archivePrefix = {arXiv},
  primaryClass  = {math.AT}
}

@incollection{Enguehard86,
  author    = {Enguehard, Michel},
  title     = {Caract{\'e}risation des groupes de {Ree}},
  booktitle = {R{\'e}vision dans les groupes finis: groupes du type de {Lie} de rang 1},
  series    = {Ast{\'e}risque},
  number    = {142--143},
  pages     = {49--139},
  publisher = {Soci{\'e}t{\'e} Math{\'e}matique de France},
  year      = {1986},
  url       = {https://www.numdam.org/item/AST_1986__142-143__49_0/}
}

@article{Steinberg60,
  author  = {Steinberg, Robert},
  title   = {Automorphisms of finite linear groups},
  journal = {Canadian Journal of Mathematics},
  volume  = {12},
  pages   = {606--615},
  year    = {1960},
  doi     = {10.4153/CJM-1960-054-6}
}

@unpublished{HogleMayZou,
  author = {Hogle, Eric and May, Clover and Zou, Foling},
  title  = {{$RO(G)$}-graded cohomology of a point is self-injective},
  note   = {Preprint}
}

\end{document}